\documentclass[11pt,reqno,a4paper, english]{amsart}
\usepackage[top=3.5cm,bottom=3cm,left=2.5cm,right=2.5cm]{geometry}
\usepackage[mathscr]{euscript}

\usepackage{amsmath, mathtools, amsthm, amsfonts, amssymb,enumerate,xcolor}
\usepackage{microtype}
\usepackage{hyperref}
\hypersetup{colorlinks={true},linkcolor={blue},citecolor=blue}
\usepackage{comment}

\usepackage[style=nature,sorting=nyt,doi=false,isbn=false,url=false,eprint=false,intitle=true]{biblatex}
\DeclareNameAlias{author}{family-given}
\numberwithin{equation}{section}

\theoremstyle{plain}
\newtheorem{theorem}{Theorem}[section]
\newtheorem{proposition}[theorem]{Proposition}
\newtheorem{corollary}[theorem]{Corollary}
\newtheorem{lemma}[theorem]{Lemma}

\theoremstyle{definition}
\newtheorem{definition}[theorem]{Definition}

\theoremstyle{remark}
\newtheorem{remark}[theorem]{Remark}

\newcommand{\N}{\mathbb{N}}
\newcommand{\Z}{\mathbb{Z}}

\newcommand{\R}{\mathbb{R}}
\newcommand{\C}{\mathbb{C}}

\DeclareMathOperator{\e}{e}
\DeclareMathOperator{\Op}{Op}
\DeclareMathOperator{\Pc}{P_c}
\DeclareMathOperator{\Pp}{P_p}
\DeclareMathOperator{\Id}{I}
\DeclareMathOperator{\W}{W_\pm}

\DeclareMathOperator{\Dz}{D}

\renewcommand{\Re}{\operatorname{Re}}
\renewcommand{\Im}{\operatorname{Im}}

\newcommand{\prodscal}[2]{(#1\mid#2)}
\newcommand{\prodscalr}[2]{\left\langle#1,#2\right\rangle}
\newcommand{\abs}[1]{\left\vert #1 \right\vert}
\newcommand{\norm}[1]{\Vert #1 \Vert}

\newcommand{\interval}[4]{\mathopen{#1}#2\mathclose{}\mathpunct{},#3\mathclose{#4}}
\newcommand{\intervaloo}[2]{\interval{(}{#1}{#2}{)}}
\newcommand{\intervaloc}[2]{\interval{(}{#1}{#2}{]}}
\newcommand{\intervalco}[2]{\interval{[}{#1}{#2}{)}}
\newcommand{\intervalcc}[2]{\interval{[}{#1}{#2}{]}}

\newcommand{\set}[1]{\mathopen{}\big\{#1\mathclose{}\big\}}
\newcommand{\parent}[1]{\bigl(#1\bigr)}

\newcommand{\bigo}[1]{O\mathopen{}\left(#1\right)}

\newcommand{\bigO}[1]{\ensuremath{\mathop{}\mathopen{}O\mathopen{}\left(#1\right)}}
\newcommand{\smallO}[1]{\ensuremath{\mathop{}\mathopen{}o\mathopen{}\left(#1\right)}}

\newcommand{\1}{\mathbf{1}}
\newcommand{\function}[5]{\begin{array}{ccccc}
#1 & \colon & #2 & \to & #3 \\
 & & #4 & \mapsto & #5 \\
\end{array}}

\newcommand{\disp}{\displaystyle}

\DeclareMathOperator{\ran}{Ran}

\newcommand{\japbrak}[1]{\langle#1\rangle}

\DeclareMathOperator{\supp}{supp}

\begin{document}


\title[Asymptotic stability under random perturbations]{Asymptotic stability of small ground states for NLS under random perturbations}

\author{Nicolas Camps}
\address{Universit\'e Paris-Saclay, Laboratoire de mathématiques d'Orsay, UMR 8628 du CNRS, B\^atiment 307, 91405 Orsay Cedex,}
\email{nicolas.camps@universite-paris-saclay.fr}


\date{December, 2019}
\vspace{-8pt}
\begin{abstract}
We consider the cubic Schrödinger equation on the Euclidean space perturbed by a short-range potential $V$. The presence of a negative simple eigenvalue for $-\Delta+V$ gives rise to a curve of small and localized nonlinear ground states that yield some time-periodic solutions known to be asymptotically stable in the energy space. We study the persistence of these coherent states under rough perturbations. We shall construct a large measure set of small scaling-supercritical solutions below the energy space that display some asymptotic stability behavior. The main difficulty is the need to handle the interactions of localized and dispersive terms in the modulation equations. To do so, we use a \textit{critical-weighted strategy} to combine probabilistic nonlinear estimates in critical spaces based on $U^p, V^q$ (controlling higher order terms) with some local energy decay estimates (controlling lower order terms). We also revisit in the perturbed setting the analysis of~\cite{benyi2015} on the probabilistic global well-posedness and scattering for small supercritical initial data. We use a distorted Fourier transform and semiclassical functional calculus to generalize probabilistic and bilinear Strichartz estimates. 
\end{abstract}

\ \vskip -1cm \hrule \vskip 1cm \vspace{-8pt}
 \maketitle
{ \textwidth=4cm \hrule}
\vspace{-10pt}
\tableofcontents


\section{Introduction}
In this paper we consider the cubic Schrödinger equation perturbed by a short-range potential in the Euclidean space of dimension $d\geq3$ :
\begin{equation}
\label{eq-nls}\tag{NLS}
\begin{cases}
i\partial_t \psi + \Delta \psi = \mu \abs{\psi}^2\psi+V \psi\,,\\
\psi\big |_{t=0} = \psi_0\,.
\end{cases}
\end{equation}
The constant $\mu$, whose sign dictates whether the nonlinearity is focusing or defocusing, will play no role since we look at small initial data. Hence, we shall fix $\mu=1$. The short-range potential $V$, whose properties are detailed below, is assumed to be in the Schwartz class $\mathcal{S}(\R^d)$. Our main result concerns the case when $d=3$ and $\sigma(H)=\set{e_0}\cup\sigma_c(H)$ with no resonance at zero, and where $e_0<0$ is a simple negative eigenvalue with positive and normalized eigenfunction $\phi_0$. Then, the 1-dimensional eigenspace spanned by $\phi_0$ bifurcates around zero to a family of small and localized nonlinear ground states. These ground states satisfy the elliptic equation
\begin{equation}
\label{eq-GS}
\parent{\Delta-V+\abs{Q}^2}Q=EQ\,,
\end{equation}
and are written under the form
\begin{equation}
Q(z)=z\phi_0+q(z),\quad E(z)=e_0+e(z)\,,
\end{equation}
where $z$ is the complex modulation parameter. They give rise to periodic solutions to~\eqref{eq-nls} of the form $u(t,x)=\e^{-itE}Q(x)$.  Soffer and Weistein~\cite{soffer1990}, followed by Pillet and Wayne~\cite{pillet1997} and Gustafson \& al.~\cite{gustafson2004} addressed the problem of asymptotic stability of these small ground states in the energy space. More precisely, any small local solution $\psi$ to~\eqref{eq-nls} in $\mathcal{H}^1(\R^3)$ is global and can be decomposed into 
\begin{equation}
\label{ansatz0}
\psi(t) = Q(z(t))+\eta(t,z(t)),\\
\end{equation}
where the radiation term $\eta$ satisfies a time dependent orthogonality condition (see~\eqref{orth}). In large time, asymptotic stability holds in the sense that $z(t)$ has a limit as $t$ goes to infinity, and $\eta(t)$ scatters in $\mathcal{H}^1(\R^3)$. This result is a particular instance of the so-called \textit{soliton resolution conjecture} (see~\cite{tao2009} for a general presentation). Our aim is to prove that the asymptotic stability property of these coherent states still holds at very low regularity, in a supercritical regime where we have a local probabilistic flow.

Before we present our main result, we recall that the critical exponent $s_c$ for the homogeneous~\eqref{eq-nls}  equation (with V=0)  is the regularity exponent for which the homogeneous Sobolev norm is invariant under the scaling
\[u_\lambda(t,x)=\lambda^{-1}u(\lambda^{-2}t,\lambda^{-1}x)\,,\quad u_{0,\lambda}(x)=\lambda^{-1}u_0(\lambda^{-1}x)\,.\]
For the cubic Schrödinger equation in $\R^d$, we have $\displaystyle s_c = \frac{d-2}{2}$. In light of the conservation laws of mass and energy
\begin{equation*}
\mathcal{M}(\psi)(t)=\int_{\R^d}\abs{\psi(t,x)}^2dx,\quad \mathcal{E}(\psi)(t)=\int_{\R^d}\frac{1}{2}\abs{\nabla\psi(t,x)}^2+\frac{\mu}{4}\abs{\psi(t,x)}^4+V\abs{\psi(t,x)}^2dx\,,
\end{equation*}
we say that the problem is mass-critical when $s_c=0$ and energy-critical when $s_c=1$. There exists a vast literature on the Cauchy theory for~\eqref{eq-nls} and we refer to the books~\cite{bourgain2007,cazenave2003,dodson2019,tao2006} or to the survey~\cite{ginibre1979-I,ginibre1979-II} and the references therein for a general presentation. Note that the equation is energy-subcritical for $d=3$, and we have a local flow for $s\in\intervalcc{1/2}{1}$, where $s_c=1/2$. On the other hand, in the scaling-\textit{supercritical} regime where $s<s_c$, the local Cauchy problem for~\eqref{eq-nls} is known to be ill-posed (see ~\cite{carles2007,ckstt2003}). Nevertheless, the \textit{probabilistic Cauchy theory} initiated by Bourgain~\cite{bourgain1996} and developed by many authors since then provides some large measure sets made of scaling-supercritical initial data $u_0^\omega$ in $\mathcal{H}^s(\R^d)$ for $s<s_c$ that give rise a local solutions to~\eqref{eq-nls} of the form $\psi = u^\omega+v$. Here, $u^\omega=\e^{-itH}u_0^\omega$ is the propagation of $u_0^\omega$ under the linear flow and $v$ is solution to the cubic Schrödinger equation with a random forcing term, but with $v(0)=0$.
\subsection{Main result}
We consider~\eqref{eq-nls} in dimension $d=3$ with small randomized initial data that lie in $\mathcal{H}^s(\R^3)\cap\ran\Pc(H)$ for some $s\in\intervaloc{1/4}{1/2}$, where $\Pc(H)$ is the projection onto the continuous spectral subspace for $H$, and we assume that $\sigma(H)=\set{e_0}\cup\sigma_c(H)$ with no resonance at zero, where $e_0<0$ is a simple negative eigenvalue. We generalize the result of~\cite{benyi2015} in the presence of such of potential by first constructing a local probabilistic flow, and extending to a global flow. Then, we prove that at each time of the evolution, the solution decouples into a sum of a ground state and a radiation term that scatters at infinity. 

More precisely, we fix $s\in\intervaloc{1/4}{1/2}$, a function $u_0$ in $\mathcal{H}^s$ and a small parameter $\epsilon$. Then, we perform the \textit{Wiener randomization procedure} on $u_0$ as detailed in section~\ref{sec-rand}, and we get a large measure set $\widetilde\Omega_\epsilon$ made of  rough and small initial data $u_0^\omega\in\mathcal{H}^s(\R^3)$ with improved Strichartz estimates. These initial data give rise to global solutions to~\eqref{eq-nls} under the form~\eqref{ansatz0} that display an asymptotic stability dynamic.
\begin{theorem}[Probabilistic asymptotic stability of small ground states]
\label{theorem-soliton}
Assume that $H$ has no resonance at zero and that $\sigma(H)=\set{e_0}\cup\sigma_c(H)$ with $e_0<0$ a simple eigenvalue. There exists a set $\widetilde\Omega_\epsilon$ and $\delta_0$ such that for all $\psi_0$ with $\displaystyle\norm{\psi_0}_{\mathcal{H}^{1/2}}<\delta_0$ and all $\omega\in\widetilde\Omega_\epsilon$,  the initial value problem
\begin{equation}
\label{cauchy-soliton}
\begin{cases}
i\partial_t\psi+\Delta\psi=\abs{\psi}^2\psi+V\psi\,,\quad (t,x)\in\R\times\R^3\,,\\
\psi_{\lvert{t=0}}=\epsilon u_0^\omega+\psi_0\,,
\end{cases}
\end{equation}
admits a unique global-in-time solution $\psi$ of the form
\[
\psi(t) = \epsilon \e^{it(\Delta-V)}u_0^\omega+v(t),\ \text{where}\ v\in C(\R,\mathcal{H}^{1/2}(\R^3))\,.
\]
Moreover, the solution uniquely resolves into
$
\psi(t) = Q(z(t))+\eta(t)$, and there exist $z_+\in\C$ and a final state $\eta_+\in \mathcal{H}^{1/2}(\R^3)\cap\ran(\Pc)$ such that
\begin{equation}
\label{eq-conv-z}
\underset{t\to+\infty}{\lim}z(t)\exp\parent{i\int_0^tE(z)d\tau} = z_+,\quad \underset{t\to+\infty}{\lim}\norm{\eta-\e^{-itH}\parent{\epsilon u_0^\omega+\eta_+}}_{\mathcal{H}^{1/2}(\R^3)} = 0\,.
\end{equation}
\end{theorem}
\begin{remark}
The measure of the set $\widetilde{\Omega}_\epsilon$ of initial with improved Strichartz estimates is all the more large since $\epsilon\norm{u_0}_{\mathcal{H}^s}$ is small. There exist $C,c>0$ such that for all $u_0$ and $\epsilon>0$,
\[
\operatorname{\mathbf{P}}\parent{\Omega\setminus\widetilde\Omega_\epsilon}\leq C\exp\parent{-c\epsilon^{-2}\norm{u_0}_{\mathcal{H}^s(\R^3)}^{-2}}\,.
\]
\end{remark}
\begin{remark}
We have conditional uniqueness for $v$ in the critical space defined in~\eqref{eq-Xcr}, and embedded into $L^\infty(\R,\mathcal{H}^{1/2})$. Hence, $v$ gains some regularity and lies in a space where there exists a deterministic Cauchy theory.
\end{remark}
By taking $u_0=0$ in the statement of Theorem~\ref{theorem-soliton} and using persistence of regularity, we may extend the deterministic result of~\cite{gustafson2004} to the intercritical regime, where $1/2\leq s\leq 1$.
\begin{corollary}[Deterministic asymptotic stability]
For $1/2\leq s\leq1$, there exists $\delta_0>0$ such that for all $\norm{\psi_0}_{\mathcal{H}^s}\leq\delta_0$, the Cauchy problem~\eqref{eq-nls} with initial data $\psi_0$ has a unique global solution $\psi\in C(\R;\mathcal{H}^s)$ that resolves into $\psi(t)=Q(z(t))+\eta(t)$, and asymptotic stability~\eqref{eq-conv-z} holds in $\mathcal{H}^{1/2}$.
\end{corollary}
\vspace{8pt}
\subsection{Background}
\subsubsection{Asymptotic stability for small solitons in the energy space}
Previous results on the asymptotic stability of small ground states hold for initial data $\psi_0$ in the energy space $\mathcal{H}^1(\R^3)$, where~\eqref{eq-nls} is known to be well-posed. Soffer and Weinstein proved in~\cite{soffer1990} that the equation displays some multichannel scattering for small and localized data in the energy space. More precisely, any initial data small in $\mathcal{H}^1(\R^3)\cap L_x^1(\R^3)$ gives rise to a global solution that resolves into a fixed ground state and a radiation term whose $L_x^6(\R^3)$-norm decays to zero, as well as some of its $L_x^2(\R^3)$-weighted norm (see Theorem 4.1 in~\cite{soffer1990}). By the use of the \textit{center stable manifold method}, Pillet and Wayne then extended this result in~\cite{pillet1997} to the case where the initial data are localized in some $L_x^2(\R^3)$-weighted spaces rather than in $L_x^1(\R^3)$. While these works impose some fixed orthogonality conditions to the modulation parameters, Gustafson, Nakanishi and Tsai introduced a time-dependent orthogonality condition that leads to an asymptotic stability result in $\mathcal{H}^1(\R^3)$ (Theorem 1.7 in \cite{gustafson2004}) without any assumption on the decay of the initial data, and the radiation term scatters in $\mathcal{H}^1(\R^3)$. A natural question that arises as  was to whether the asymptotic stability holds true below the energy space. Indeed, we still have a local flow for $1/2\leq s\leq1$, and a local probabilistic flow for $1/4<s<1/2$. We note that our proof does not use any decay assumption on the initial data. Besides, the randomized initial are not likely to decay for a general $u_0$ (see the discussion in paragraph~\ref{section-wiener-decomposition}). 
\subsubsection{Probabilistic well-posedness theory}
\label{sub-intro-proba}
 Even though the Cauchy problem~\eqref{eq-nls} is in general ill-posed for scaling-supercritical initial data below $\mathcal{H}^{1/2}(\R^3)$, the probabilistic method can provide some large measure sets of initial data that give rise to global-in-time solutions. Many works were done in this direction for different dispersive PDE's after the pioneer work of Jean Bourgain~\cite{bourgain1996} for the NLS equation on $\mathbb{T}^2$, followed by the work of Burq and Tzvetkov on the nonlinear wave equation (NLW) on a 3D compact Riemannian manifold without a boundary~\cite{burq2008,burq2008II}. The idea is to find a randomization procedure which improves some Strichartz estimates on the free evolution of the randomized data, but does not improve their regularity. 

To extend local well-posedness below the scaling-critical Sobolev space, the strategy consists in decomposing the solution into $\psi = u^\omega+v$ where $u^\omega\coloneqq \e^{-itH}u_0^\omega$ is the linear evolution of the randomized data. The remaining term $v$ is smoother and satisfies~\eqref{eq-nls} with a stochastic forcing term :
\begin{equation}
\label{eq-ivp0}
\begin{cases}
i\partial_t\psi-H\psi = \mathcal{N}(\psi),\\
\psi\vert_{t=0}=\psi_0+u_0^\omega.
\end{cases}
\iff
\begin{cases}
i\partial_tv-Hv = \mathcal{N}(v+u^\omega),\\
v\vert_{t=0}=\psi_0.
\end{cases}
\end{equation}
Besides, we need to solve equation~\eqref{eq-ivp0} for $v$ in a well-chosen subspace of $C(I,\mathcal{H}^{s_c})$ by a contraction mapping argument. Thanks to its random structure, the linear evolution of the randomized data $u^\omega$ displays some enhanced integrability properties that make it possible to gain regularity on the stochastic Duhamel term. After that, we globalize the local solutions by using the~\textit{Bourgain invariant measure argument} in the case of $\mathbb{T}^d$, or by using a priori estimates on some critical norm of $v$.

The general randomization procedure consists in finding a well-chosen decomposition of the initial data and in decoupling the terms of this decomposition by multiplying each of them by some independent random variables~$\set{g_n(\omega)}_n$ centered around zero. Then, taking averages cancels interference and improves therefore the integrability of the data. There exists many versions of the randomization procedure, and we refer the interested reader to the survey~\cite{benyi2019} and the references therein for a detailed presentation. One has to distinguish between two cases. In the confining case where the physical space is a compact manifold, or when the equation contains a confining potential, we consider the decomposition $u_0^\omega=\sum_n g_n(\omega)u_n\e_n(x)$ where $\set{e_n}$ is an orthonormal basis made of eigenfunctions of the Schrödinger operator. In the Euclidean case, the Schrödinger operator does not provide such a natural decomposition, and we use instead unit scale frequency decomposition on Wiener cubes (see~\cite{benyi2015,benyi2015-local,luhrmann2020,zhang2012}), or microlocal decomposition (see~\cite{bringmann2020,burq-krieger2019,murphy2017}). In some specific cases, we can also use a compactifying transformation such as the Lens transform for NLS (see~\cite{burq-thomann2020,btt2013}) or the Penrose transform for NLW (see~\cite{asds2013}) to apply in the Euclidean setting variants of the Bourgain invariant measure argument. Nevertheless, it must be emphasized that the probabilistic method used in the aforementioned works concerns perturbations of the zero solution and ends up with asymptotic results like global well-posedness and scattering for small data or in the defocusing case. 

Still, Kenig and Mendelson recently addressed in~\cite{kenig2019} the problem of asymptotic stability of large solitons for the quintic focusing wave equation with randomized radial initial data. They introduced a randomization procedure based on the distorted Fourier transform adapted to the linearized operator around a soliton, and proved some intricate kernel estimates due to the presence of a resonance at zero for this operator. Then, Bringmann showed the stability of the ODE blow-up for the radial energy critical NLW in 4D under random perturbations below the energy space~\cite{bringmann2020-2}. However, for nonlinear focusing Schrödinger equations the asymptotic stability around solitons is more subtle. Indeed, the solitons for~\eqref{eq-nls} with $V=0$ are stable when the nonlinearity is small and the critical exponent is negative. Hence, we cannot run the probabilistic strategy since the smoother term $v$, which is expected to lie in the critical space $C(\R;\mathcal{H}^{s_c}(\R^d))$, cannot serve as a substitute for the conservation laws. In addition, the linearized operator around the soliton is not self-adjoint. We refer to~\cite{cuccagna2020survey,tao2009} and the references therein for a general insight on the results and techniques about stability of solitons for NLS. 
\subsubsection{Schrödinger equation with a short-range potential}
The study of dispersive PDE's with a potential is a very general problem. For instance, a potential can arise in the modulation equations obtained to address stability problems around soliton solutions by the analytic method (see~\cite{tao2009}). There exists a vast literature about Schrödinger equations perturbed by a localized potential (see for instance~\cite{walsh2013,rousset2018,pusateri2020}). These works rely on the use of a \textit{distorted Fourier transform} $\mathcal{F}_V$, which is the analogue adapted to $H$ of the Fourier transform. When it exists, the distorted Fourier transform defines a partial isometric map onto $L_{ac}^2(\R^d)$ that conjugates $H$ with the operator of multiplication by $\abs{\xi}^2$. Additionally, the distorted Fourier transform is related with the wave operators $\W$. We refer the interested reader to the seminal work of Agmon~\cite{agmon1975}, grounded on the previous works~\cite{alsholm1971,ikebe1960}. In the probabilistic context, we shall follow the strategy used by Kenig and Mendelson for the quintic focusing NLW in~\cite{kenig2019} to provide a randomization procedure based on a \textit{distorted frequency} decomposition. This randomization procedure commutes with the flow $\e^{-itH}$ and is therefore suited to the underlying linear dynamic of~\eqref{eq-nls}. In the present work we consider the Schrödinger equation perturbed by a potential in order to generate a curve of small nonlinear ground states, and our analysis is easier than in~\cite{kenig2019} since we can assume no resonance at zero. Let us now precise the assumptions made on $V$.
\vspace{8pt}
\paragraph{\textbf{Assumptions on the potential} :} $V$ is a real-valued potential in the Schwartz class $\mathcal{S}(\R^d;\R)$. We write
\begin{equation*}
\label{as}
H\coloneqq-\Delta + V,\quad H_0 = -\Delta.
\end{equation*}
Note that $V$ is \textit{short-range} in the sense of Agmon : there exists $\epsilon>0$ such that the operator
\begin{equation*}
\label{eq-sr}
u\in\mathcal{H}^2(\R^d)\mapsto (1+\abs{x})^{1+\epsilon}Vu\in L^2(\R^d)
\end{equation*}
is a compact operator. Namely, $H$ is $H_0$-compact. Hence, $H$ admits a unique self-adjoint realization on $L^2(\R^d)$ with domain $\mathcal{D}(H)=\mathcal{H}^2(\R^d)$ and has the same essential spectrum as $H_0$, that is the half line $\intervalco{0}{+\infty}$. In addition, Agmon proved in~\cite{agmon1975} that if $V$ is short-range then the spectrum of $H$ is given by
\begin{equation*}
\sigma(H)=\sigma_p(H)\cup\sigma_{ess}(H)=\sigma_p(H)\cup\intervalco{0}{+\infty},
\end{equation*}
where $\sigma_p(H)$ (the discrete spectrum of $H$) is  a countable set made of eigenvalues with finite multiplicity. We have the spectral decomposition :
\begin{equation*}
\quad L^2(\R^d)=L_{ac}^2(\R^d)\oplus L_p^2(\R^d),
\end{equation*}
where $L^2_p(\R^d)$ is the space spanned by the associated eigenvectors and $L^2_{ac}(\R^d)$ is the absolute continuous spectral subspace for $H$. In what follows, we assume that $V$ is \textit{generic}, in the sense that
\begin{enumerate}[(i)]
\item$\sigma_p(H)\subset \intervaloo{-\infty}{0},$
\item0 is not a resonance for $H$.
\end{enumerate}
We refer to~\cite{walsh2013} and the references therein for precise discussions about optimal assumptions on $V$. In the present work, we assume extra regularity on $V$ to get the global-in-time local smoothing estimates of Proposition 1.33 in~\cite{rodnianski2015}. We also chose to take $V$ smooth in order to derive a bilinear estimate for perturbed linear Schrödinger evolutions based on \textit{semiclassical functional calculus}, and that does not rely on an explicit structure theory for wave operators (see~\cite{beceanu2014,hong2013}). The assumption that $V$ is in $\mathcal{S}(\R^d)$ is far from being optimal and is made for simplicity.
\subsection{Organization of the paper}
\subsubsection{Spectral theory}
We present in section~\ref{section-fourier} the basic properties of the \textit{distorted Fourier transform}, and it's connections with wave operators and functional calculus. Under our assumptions on $-\Delta+V$, the distorted Fourier transform is a unitary operator onto the continuous spectral subspace $L_{c}^2(\R^d)$ that provides very useful frequency decompositions adapted to the perturbed framework. Indeed, the distorted Fourier multipliers involved in these decompositions commute with the perturbed linear flow $\e^{-itH}$. In particular, they preserve the continuous spectral subspace $\ran(\Pc)$. We state an analogue of the Fourier multiplier theorem for this transformation that follows from the $L^p$-boundedness of wave operators, and we also stress out that distorted Fourier multipliers by radial functions coincide with spectral multipliers defined by the functional calculus on self-adjoint operators. Then, we shall briefly revisit some nonlinear tools such as the Littlewood-Paley inequality in order to perform standard harmonic analysis techniques for nonlinear dispersive PDE's in the perturbed setting.  
\subsubsection{Probabilistic and Bilinear Strichartz estimates for the inhomogeneous Schrödinger evolution}
First, we generalize the refined global-in-time Strichartz estimates for randomized data from~\cite{benyi2015}, Lemma 2.3. To do so, we introduce a \textit{distorted Wiener decomposition} by using distorted Fourier multipliers localized on unit cubes (whereas standard Fourier multipliers are used in the flat case). As mentioned above, the reason for this is to come with a randomization that commutes with $\e^{-itH}$ and that preserves $\ran \Pc(H)$. Proposition~\eqref{prop-rand-str} yields a large measure set $\widetilde\Omega$ of initial data in $\ran \Pc(H)$ that display refined Strichartz estimates. In order to revisit the standard proof written in~\cite{benyi2015}, we use the variant of a Bernstein's estimate for distorted Fourier multipliers~\eqref{eq-distorted-multiplier} presented in section~\ref{section-fourier}. Furthermore, we stress out that since $u_0^\omega$ is in $\ran(\Pc)$ we can apply the local smoothing property~\eqref{eq-local-smoothing} for $\e^{-itH}$. For all these reasons, randomization adapted with a distorted frequency decomposition turns out to be the best suited to the dynamic of~\eqref{eq-nls}. Nevertheless, the distorted Fourier transform does not change a product into a product of convolution. Therefore, the nonlinear analysis is more intricate in this setting, and we cannot generalize bilinear estimates so easily. Still, we prove an extended version of the bilinear Strichartz estimate from J. Bourgain to the inhomogeneous case with $V$ in the Schwartz class. Our proof does not rely on a structure formula for the wave operator, and we use instead semiclassical functional analysis to intertwine Fourier and distorted Fourier multipliers. More precisely, we quantify the interactions between functions of the form $\varphi(N^{-1}H)$ and $\varphi(M^{-1}H_0)$ in order to intertwine localization with respect to $H$ and $H_0$. Then, we decompose the perturbed evolution $\e^{-itH}$ into a superposition of free evolution $\e^{it\Delta}$ of flat-Fourier localized data, and we apply the original Bourgain's estimate for the free evolution $\e^{it\Delta}$, as well as local energy decay. 
\subsubsection{Probabilistic global existence and scattering on the continuous spectral subspace} 
The purpose of section~\ref{section-continuous} is to lay the foundations for the analysis conducted in section~\ref{section-soliton} about the stability of small nonlinear ground states under rough and randomized perturbations. First, we recall the definitions and key properties of critical spaces of functions from an interval $I$ to $\mathcal{H}^\frac{d-2}{2}(\R^d)$, written $X^\frac{d-2}{2}(I)$. This space is built upon the space of functions of finite $q$-variation $V^q$ and its predual, the atomic space $U^p$. Next, we generalize Theorem 1.2 in~\cite{benyi2015} to the perturbed framework in order to understand the dynamic for solutions to~\eqref{eq-nls} projected on the continuous spectral subspace for $H$. To do so, we shall specify and develop a bit on the random nonlinear estimates derived in~\cite{benyi2015}. We obtain the following result.
\begin{theorem}
\label{theorem-continuous}
Assume that $H$ has no resonance nor eigenvalue at zero. Let $s_d = \frac{d-1}{d+1}\cdot s_c$ and $s\in\intervaloc{s_d}{s_c}$. For all $\omega\in\widetilde\Omega_\epsilon$ and $\psi_0\in\mathcal{H}^\frac{d-2}{2}(\R^d)\cap \ran(\Pc)$ small enough, the Cauchy problem
\begin{equation}
\label{eq-nls-continuous}
\tag{$\text{NLS}_c$}
\begin{cases}
i\partial_t\psi +(\Delta-V)\psi=\Pc\parent{\abs{\psi}^2\psi}\,,\\
\psi_{t=0}=\epsilon u_0^\omega+\psi_0\,,
\end{cases}
\end{equation}
admits a unique global-in-time solution $\psi$ in the class
\begin{equation*}
\epsilon\e^{-itH}u_0^\omega + X^\frac{d-2}{2}\intervalco{0}{\infty}\subseteq C\parent{\intervalco{0}{\infty},\mathcal{H}^{s}(\R^d)\cap\ran(\Pc)}\,.
\end{equation*}
In addition, there exists $v_+\in \mathcal{H}^\frac{d-2}{2}(\R^d)\cap\ran(\Pc)$ such that
\begin{equation*}
\underset{t\to+\infty}{\lim}\norm{\psi-\e^{-itH}\epsilon u_0^\omega+v_+}_{\mathcal{H}^\frac{d-2}{2}(\R^d)} = 0\,.
\end{equation*}
\end{theorem}
\begin{remark}
If the pure point spectrum of $-\Delta+V$ is empty then $\Pc = \operatorname{Id}$ and~\eqref{eq-nls-continuous} is~\eqref{eq-nls}. By time reversibility of~\eqref{eq-nls} the same results hold true for negative times, and conditional uniqueness holds in the following sense : let $v_1,v_2$ be two solutions in $X^\frac{d-2}{2}(0,\infty)$. If there exists $t\in\intervalco{0}{\infty}$ such that $v_1(t)=v_2(t)$ then $v_1\equiv v_2$.
\end{remark}
In section~\ref{section-soliton}, we consider perturbed solutions of~\eqref{eq-nls} around a ground state under the form $Q+\eta$. In particular, the radiation term $\eta$ satisfies a nonlinear Schrödinger equation with a stochastic forcing term that contains localized and linear terms. The main difficulty is that the higher order terms are handled in critical spaces, while the localized lower order ones can only be controlled in weighted Sobolev spaces by the use of some \textit{local smoothing estimates}.
Therefore, we shall present in Proposition~\ref{proposition-trick} a \textit{critical-weighted strategy} that gives a way to perform nonlinear analysis in critical spaces and in weighted spaces simultaneously. In dimension $d=3$, our analysis shows that this technique also gives an alternative version of Theorem~\ref{theorem-continuous} where global existence and uniqueness for the nonlinear part of the solution hold in $V^2$ intersected with the weighted Sobolev space $L_t^2(\R;\mathcal{H}^{1,-1/2-}(\R^3))$. A similar approach for Korteweg-de Vries equation can be found in~\cite{koch2012}.
\subsubsection{Outline of the proof of Theorem~\ref{theorem-soliton}} 
Let us now present the framework and some notations used in~\cite{gustafson2004}. In what follows, we identify $\C\simeq\R^2$, with $z=z_1+iz_2=(z_1,z_2)$. For convenience, we denote the operator $\Dz_z f(z)\coloneqq (\partial_{z_1}f(z),\partial_{z_2}f(z))$ acting on $\C$ seen as a real vector space of dimension 2 endowed with the real scalar product $\japbrak{(z_1,z_2),(z_1',z_2')}=z_1z_1'+z_2z_2'$. Then, we construct a local probabilistic flow for~\eqref{eq-nls}, and we decompose the solution at each time  under the form~\eqref{ansatz0} with the following time-dependent orthogonality condition imposed for the radiation term $\eta(z)$ :
\begin{equation}
\label{orth}
\eta(z)\in\mathcal{H}_c(z)\coloneqq\set{\eta\in L^2(\R^3)\mid\japbrak{i\eta,\partial_{z_1} Q(z)}= \japbrak{i\eta,\partial_{z_2} Q(z)} = 0}\,.
\end{equation}
In particular, observe that $\mathcal{H}_c(0)=\ran(\Pc)=L_c^2(\R^3)$. In broad outline, we ask for $\eta$ to be orthogonal to the center manifold of ground states at each time, and this has for effect to cancel some linear terms in the modulation equations~\eqref{eq-modulation}. Note that these orthogonality conditions are all the more natural since any small function $\phi$ in $L^2(\R^3)$ can actually be decomposed into  
\begin{equation}
\label{eq-decomposition}
\phi=Q(z)+\eta\,,
\end{equation}  
where $z\in\C$ and $\eta\in\mathcal{H}_c(z)$ satisfies ~\eqref{orth} (see Lemma 2.3 in \cite{gustafson2004}). Furthermore, the decomposition~\eqref{eq-decomposition} is explicit and $\Pp(\eta)$, the discrete part of $\eta$, can be expressed as a function of $\Pc(\eta)$ and $z$.  Indeed, for each $z$ small enough there exists a bijective operator $\operatorname{R}(z):\mathcal{H}_c(0)\to\mathcal{H}_c(z)$ written in Lemma~\ref{lemma-proj} such that
\begin{equation*}
\eta=\operatorname{R}(z)\Pc(\eta)\,.
\end{equation*}
We introduce $\nu= \Pc(\eta)-\epsilon u^\omega$. Since $u^\omega$ is in $\ran(\Pc)$, we have that $\nu\in\ran \Pc$, and
\begin{equation*}
\Pp(\eta)=(\operatorname{R}(z)-I)\parent{\epsilon u^\omega+\nu}\,.
\end{equation*}
Therefore, the evolution reduces to a system with two degrees of freedom $z$ and $\nu$ 
\begin{equation}
\label{ansatz}
\psi = Q(z)+\eta(z)\,,\quad
\eta(z)=\operatorname{R}(z)\parent{\epsilon u^\omega+\nu(z)}\,.
\end{equation}
The aim is to obtain some global a prioi estimates for $(\nu,m(z))$, solution to the coupled system of modulation equations~\eqref{eq-modulation},~\eqref{eq-nu}, where $m(z)$ denotes bijective gauge defined in~\eqref{eq-gauge-transformation}. By injecting the ansatz~\eqref{ansatz} into the equation~\eqref{eq-nls}, we see that $\nu$ satisfies the perturbed Schrödinger equation with a stochastic forcing term~\eqref{eq-nu}. It has the form
\begin{equation*}
i\partial_t\nu+(\Delta-V)\nu = F(\nu,Q(z),z,\epsilon u^\omega)\,,
\end{equation*}
and stress out that the forcing term contains some linear terms with respect to $\eta$ and its complex conjugate $\overline{\eta}$. The main difficulty is to prove scattering with these linear terms. Indeed, they are time dependent, and they are not self-adjoint perturbations of $H$. Hence, we can not absorb them in the left-hand side in order to get a perturbed operator as studied in our framework. We rather need to consider linear terms as source terms. Fortunately, these terms come with powers of $Q$ or $\phi_0$ which are small in $\mathcal{H}^2$ and localized. Hence, we shall use local energy decay to control them. On the other hand, we need the critical spaces based on $U^2$ and $V^2$ to control the term $\mathcal{N}(u^\omega+\nu)$. Consequently, we shall exploit the aforementioned \textit{critical-weighted strategy} (see Proposition~\ref{proposition-trick}) to simultaneously cope with the localized and critical terms and to get the desired a priori estimate on $\nu$.

The dynamic of $m(z)$ is governed by an ODE that arises as when differentiating the orthogonality conditions~\eqref{orth} with respect to time. As explained above, orthogonality conditions~\eqref{orth} cancel terms which are linear with respect to $\eta$. Hence, the obtained ODE contains terms which are at least quadratic in $\eta$. Since the local smoothing estimate controls some global-in-time and weighted quadratic norms of $\eta$, a global-in-time control of $\dot m(z)$ in $L_t^1(\R)$ follows from the analysis of the ODE and yields the convergence result written in~\eqref{eq-conv-z}.

Eventually, the analysis leads to the global a priori estimates  on $\nu$ and $m(z)$ stated in Proposition~\ref{prop-a-priori-eta}, from which we deduce global existence by the use of a continuity argument. The desired asymptotic dynamic also follows from these a priori estimates.
\vspace{8pt}
\paragraph{\textbf{Notations}}
Let $s,\sigma\in\R$. We denote the Sobolev space by $\mathcal{H}^s(\R^d)$, and we define the weighted Sobolev space by
\[L^{2,\sigma}(\R^d)=\set{u\mid \japbrak{x}^\sigma u\in L^2(\R^d)}\,,\quad \mathcal H^{1,\sigma}(\R^d)= \set{u\mid \japbrak{x}^\sigma u\,,\japbrak{x}^\sigma\nabla u\in L^2(\R^d)}\,.\]
The projection onto the continuous spectral subspace $L^2_c(\R^d)$ for $H$ is written $\Pc$, and $\Pp=\Id-\Pc=\prodscal{\cdot}{\phi_0}\phi_0$ is the projection onto the pure point subspace for $H$. Note that 
\[\ran(\Pc)=L_c^2(\R)=\mathcal{H}_c(0)\,,\]
where spaces $\mathcal{H}_c(z)$ encode the orthogonality conditions~\eqref{orth} to the ground state manifold. We write $\e^{-itH}=\e^{-itH}\Pc$ the perturbed Schrödinger evolution projected on the continuous spectral subspace, whereas  by $\e^{it\Delta}=\e^{-itH_0}$ is the free Schrödinger evolution. For $I\subseteq\R$, the space-time Lebesgue space $L_t^q(I;L_x^p(\R^d))$ is written $L_t^qL_x^r(I)$, or $L_t^qL_x^r$ when the dependence on $I$ is clear. Given a Borel function $f:\R_+\to\C$, we denote the operators defined by the usual functional calculus by $f(H_0)$ and $f(H)=f(H)\operatorname{P}_c(H)$. In particular, we define the Littlewood-Paley multipliers (see~\eqref{eq-littlewood-paley}) around the dyadic frequency $N\in2^\N$ by 
\[\Delta_Nu=\varphi(N^{-1}H_0)\,,\ \Pi_N=\sum_{K\leq N}\Delta_Ku\,,\quad \widetilde\Delta_N=\varphi(N^{-1}H)\,,\ \widetilde\Pi_N=\sum_{K\leq N}\widetilde\Delta_Ku\,.\] 
Given a function $m:\R^d\to\C$ we denote the distorted Fourier multiplier $\mathcal{F}_V^*\Pc(H) m\mathcal{F}_V$ by $\operatorname{M}_m(H)$ (see Definition~\ref{def-distorted-multiplier}) where $\mathcal{F}_V$ is the distorted Fourier transform (see Proposition~\ref{theo-fv}). 
Given a function $u_0$ in $L_{c}^2(\R^d)$, we define its Wiener randomization by $u_0^\omega$ in Definition~\ref{def-rand}, and $u^\omega = \e^{-itH}u_0^\omega$ is the linear propagation of the randomized data by the under perturbed flow.
\section{Spectral theory for generic short-range Schrödinger operators}
\label{section-fourier}
In this section, we present well known facts about functional calculus on perturbed Schrödinger operators that gives rise to a natural framework to work with $-\Delta+V$ and $-\Delta$ all together. First, we collect some definitions and basic properties of \textit{wave operators}, such as the intertwining property as well as their $L^p$-boundedness. As a consequence, we will see how dispersive estimates known for free Schrödinger evolution carry over to the perturbed setting. Additionally, we recall the construction of a \textit{distorted Fourier transform} based on the \textit{limiting absorption principle}. We shall see how the wave operators connect this transformation to the usual - or flat - Fourier transform, and lead to a useful generalization of the Fourier multiplier theorem. Some of the results we state here may not hold true in low dimensions $d<3$.
\subsection{Wave operators and the intertwining property}
The wave operators are defined by the strong limit 
\begin{equation*}
\W = \underset{t\to\pm\infty}{s-\lim}\e^{itH}\e^{-itH_0}\quad \text{in}\ \mathcal{B}(L^2(\R^d))\,.
\end{equation*}
They aim at understanding scattering of Schrödinger operators $-\Delta+V$ by comparing the asymptotic behavior of $\e^{-itH}$ with the asymptotic behavior of $\e^{-itH_0}$. We say that the wave operators are asymptotically complete if 
\begin{enumerate}[(i)]
\item $\W$ is bounded and surjective from $L^2(\R^d)$ onto $L_{ac}^2(\R^d)$
\item $\sigma_{sc}(H)=\emptyset$, and hence $L_{ac}^2(\R^d)=L_c^2(\R^d)$.
\end{enumerate}
Agmon proved in~\cite{agmon1975} that the wave operators $\W$ exist and are asymptotically complete under the decaying assumption $\japbrak{x}^{1+\epsilon}V\in L^\infty(\R^d)$. Consequently, wave operators are partial isometries onto $L_{c}^2$ in this case, that is
\begin{equation*}
\W^*\W = \Id,\quad \W \W^* = \Pc\,.
\end{equation*}
It follows from the definition that
\begin{equation*}
\e^{-itH}\W=\W\e^{-itH_0}\,.
\end{equation*}
This leads to the so-called intertwining property : for any Borel function $f:\R\to\C$
\begin{equation}
\label{eq-intertwining}
H\Pc = \W H_0\W^*,\quad f(H)\Pc = \W f(H_0)\W^*\,,
\end{equation}
where $f(H)$ and $f(H_0)$ are defined by the functional calculus on self-adjoint operators. In a seminal paper~\cite{yajima1995}, Yajima proved the $W^{k,p}(\R^d)$ boundedness for wave operators when some extra regularity and decay on $V$ and $\mathcal{F}(V)$ are assumed. We state the result in $L^p(\R^d)$, since we do not need more in our study.
\begin{proposition}[$L^p$-bound of wave operators, Theorem 1.1 in~\cite{yajima1995}]
\label{prop-Lp-bound}
Let $V$ be a generic potential with some regularity and decaying assumptions detailed in~\cite{yajima1995}, and $1\leq p\leq +\infty$. The wave operators $\W$ can be extended to bounded operators in $L^p(\R^d)$: there exists $C_p$ such that for all $f$ in $L^2(\R^d)\cap L^p(\R^d)$,
\begin{equation}
\label{eq-Lp-bound}
\norm{\W f}_{L^p(\R^d)}\leq C_p\norm{f}_{L^p(\R^d)}\,.
\end{equation}
\end{proposition}
 Given $1\leq p,q\leq+\infty$, it follows from the intertwining property~\eqref{eq-intertwining} and from the $L^p$ bound~\eqref{eq-Lp-bound} that $f(H)$ and $f(H_0)$ have equivalent norms on $\mathcal{B}(L^p(\R^d),L^q(\R^d))$. There exists $C_{p,q}$ such that for any Borel function $f: \R\to\C$,
\begin{equation*}
C_{p,q}^{-1}\norm{f(H)}_{\mathcal{B}(L^p,L^q)}\leq\norm{f(H_0)}_{\mathcal{B}(L^p, L^q)}\leq C_{p,q}\norm{f(H)}_{\mathcal{B}(L^p, L^q)}.
\end{equation*}
One striking but straightforward consequence of Proposition~\ref{prop-Lp-bound} is the generalization of the dispersive estimate to perturbed linear Schrödinger evolutions.
\begin{proposition}[Dispersive estimate for perturbed Schrödinger evolution]
\label{prop-pertrubed-str}
Let $V$ be a potential as in Proposition~\ref{prop-Lp-bound}, $2\leq p\leq +\infty$ and $q$ such that $1/p+1/q=1$. There exists a constant $C_p$ such that for all $f\in L^2(\R^d)\cap L^q(\R^d)$ and $t\in\R\setminus\{0\}$
\begin{equation}
\label{eq-pertrubed-str}
\norm{\e^{-itH}\Pc f}_{L^p(\R^d)}\leq C^p\abs{t}^{-d(1/2-1/p)}\norm{f}_{L^q(\R^d)}.
\end{equation}
By the $TT^*$ argument, we deduce from the above dispersive estimate the global-in-time Strichartz estimate for the perturbed linear evolution : for all Schrödinger admissible pairs $2\leq q,r\leq\infty$, with $3\leq d$ and $\displaystyle\frac{2}{q}+\frac{d}{r}=\frac{d}{2}$
\begin{equation}
\label{eq-global-str}
\norm{\e^{-itH}\Pc f}_{L_t^q(\R;L_x^r(\R^d))}\leq C_{q,r,d} \norm{f}_{L^2(\R^d)}\,.
\end{equation}
\end{proposition}
This result was originally proved in~\cite{journe1991}, under the assumption that $H$ has no eigenvalue nor resonance at zero, and that $V$ satisfies the decay condition
\[\abs{V(x)} \lesssim \japbrak{x}^{-2-\epsilon}\quad \text{for some $0<\epsilon$
}.\]
\subsection{Distorted Fourier multipliers}
In order to refine the perturbed Strichartz estimates of Proposition~\ref{prop-pertrubed-str} for randomized initial data as in~\cite{benyi2015}, we have to use unit-scale decomposition of the frequency space adapted to the operator $-\Delta+V$. More precisely, we need to set out a \textit{distorted Fourier transform} that conjugates the operator $H$ with the multiplication by $\abs{\xi}^2$, and to generalize the Fourier multiplier theorem. To construct such a transformation, the usual strategy is to determine some \textit{generalized plane waves} $\e(x,\xi)$ that are perturbations of the plane waves $\e^{ix\cdot\xi}$ and that formally satisfy the Helmoltz equation 
\begin{equation}
\label{eq-pw}
\parent{-\Delta+V}\e(\cdot,\xi)=\abs{\xi}^2\e(\cdot,\xi)\,,
\end{equation}
with the asymptotic condition
\begin{equation*}
v(x,\xi)\coloneqq \e(x,\xi)-\e^{ix\cdot\xi}=O_{\abs{x}\to\infty}\mathopen{}\left(\abs{x}^{-1}\right)\,,
\end{equation*}
as well as the Sommerfeld radiation condition
\begin{equation*}
r\parent{\partial_r-i\abs{\xi}}v(x,\xi)\underset{\abs{x}\to0}{\longrightarrow}0\,.
\end{equation*}
Equation~\eqref{eq-pw} cannot be solved in $L^2(\R^d)$ since $\abs{\xi}^2$ is in the essential spectrum of $H$. However, this equation can be solved in weighted spaces by a procedure called the \textit{limiting absorption principle}. This principle states that the resolvent $R(z)\coloneqq \parent{H-z}^{-1}$, which is bounded from $L^2(\R^d)$ to $\mathcal{H}^2(\R^d)$ and that defines an analytic operator valued function on $\C\setminus\intervalco{0}{+\infty}$, can be extended to the boundary of its domain of definition in the following sense.
\begin{proposition}[Limiting absorption principle, Theorem 4.2 in~\cite{agmon1975}]
\label{prop-abs}
Let $\lambda>0$ and $\delta>1/2$. The sequence 
\begin{equation*}
\underset{\substack{z\to\lambda\\ \Im(z)>0}}{\lim}\ R(z)\eqqcolon R^+(\lambda)
\end{equation*}
converges in $\mathcal{B}\parent{\japbrak{x}^{-\delta}L^2(\R^d),\japbrak{x}^{\delta}\mathcal{H}^2(\R^d)}$ endowed with the uniform operator norm topology.
\end{proposition}
Then, the existence of a family of \textit{distorted plane waves} follows from the above limiting absorption principle. Indeed, the solutions to equation~\eqref{eq-pw} satisfy the \textit{Lippman-Schwinger} equation
\begin{equation}
\label{eq-lip-schwinger}
\e(x,\xi)=\e^{ix\cdot\xi}-R^+(\abs{\xi}^2)(V\e^{ix\cdot\xi})\,.
\end{equation}
\begin{proposition}[Generalized plane wave, Theorem 5.1 in~\cite{agmon1975}]
There exists a measurable function $\e$ in $L_{loc}^2\parent{\R^d\times\parent{\R^d\setminus\{0\}}}$ such that for every fixed $\xi$ in $\R^d\setminus\{0\}$, the function $\e(\cdot,\xi)$ belongs to $H_{loc}^2(\R^d)\cap C(\R^d)$ and is a solution to equation~\eqref{eq-pw} in $H_{loc}^2(\R^d)$.
\end{proposition}
See~\cite{ikebe1960} for further properties of the generalized plane waves when $V$ is regular. Another important consequence of the limiting absorption principle is the \textit{local smoothing estimate}, or \textit{local energy decay}, for the perturbed linear Schrödinger evolution. We state a version taken from~\cite{rodnianski2015} (Proposition 1.33) where some extra regularity and decay assumptions on $V$ are required. This local smoothing effect and its transferred version into the space $U^2$ is a key ingredient in the proof of Theorem~\ref{theorem-soliton}, when we study the perturbations around the ground state.
\begin{proposition}[Global-in-time local smoothing estimate for $H$, Proposition 1.33 in~\cite{rodnianski2015}]
\label{prop-local-smoothing}
Under our assumptions made on $V$, and for $\psi$ solution to the forced Schrödinger equation 
\[\displaystyle{i\partial_t\psi+(\Delta-V) \psi=F},\quad \psi_{\mid{t=0}}=\psi_0\,,\]
we have the following global-in-time control of the local energy of $\psi$
\begin{equation}
\label{eq-local-smoothing}
\int_\R\norm{\Pc(\psi)}_{\mathcal{H}^{1,-1/2-}}^2dt\lesssim_V\norm{\psi_0}_{\mathcal{H}^{1/2}(\R^d)}^2+\int_\R\norm{F(t)}_{L_x^{2,1/2+}}^2dt\,.
\end{equation}
\end{proposition}
The next step is to define a \textit{distorted Fourier transform} from the generalized plane waves constructed in the above paragraph. Since $H$ may have some pure point spectrum in our assumptions, the operator we get is only a partial isometry with range $L_{c}^2(\R^d)$. 
\begin{proposition}[Distorted Fourier transform, Theorem 5.1 in~\cite{agmon1975}]
\label{theo-fv}
There exists a partial isometry $\mathcal{F}_V$ from $L^2(\R^d)$ to $L_{c}^2(\R^d)$ which diagonalizes $H$ on $\mathcal{H}^2(\R^d)$. For all $f\in \mathcal{H}^2(\R^d)$,
\begin{equation*}
(-\Delta+V)f = \parent{\mathcal{F}_V^* M_{\abs{\ }^2} \mathcal{F}_V}f\,.
\end{equation*}
Moreover, for any $f$ in $L^2(\R^d)$ the following representation formula holds in $L^2(\R^d)$ : 
\begin{align}
\label{fou}
\mathcal{F}_V f (\xi) &= (2\pi)^{-d/2} \underset{n\to\infty}{\lim}\int_{\abs{x}<n}\overline{e(x,\xi)}f(x)dx\,, \\
\label{inv-fou}
\mathcal{F}_V^*f(x) &= (2\pi)^{-d/2} \underset{n\to\infty}{\lim}\int_{\abs{x}<n} e(x,\xi) \Pc(H)f(\xi) d\xi\,.
\end{align}
Furthermore, we have
\begin{equation}
\label{eq-w-f}
\operatorname{W}_+ = \mathcal{F}_V^*\mathcal{F}\,.
\end{equation}
\end{proposition}
Thanks to the distorted Fourier transform we are now able to define some analogues of the Fourier multipliers, that commute with $H$.
\begin{definition}
\label{def-distorted-multiplier}
Let $m:\R^d\to\C$ be a function in $L^\infty$. The distorted Fourier multiplier
\begin{equation}
\label{eq-def-mult}
\operatorname{M}_m(H)\coloneqq \mathcal{F}_V^*m(\xi)\mathcal{F}_V
\end{equation}
is a bounded operator on $L^2(\R^d)$, where $m(\xi)$ denotes the operator $u\in L_\xi^2\mapsto mu$. We have
\begin{equation*}
\norm{\operatorname{M}_m(H)}_{\mathcal{B}(L^2(\R^d))}\leq \norm{m}_{L^\infty(\R^d)}\,.
\end{equation*}
\end{definition}
It follows from equation~\eqref{eq-w-f} that 
\begin{equation*}
\operatorname{W}_+^*\operatorname{M}_m(H)\operatorname{W}_+ = \operatorname{M}_m(H_0)\,,
\end{equation*}
where $\operatorname{M}_m(H_0)$ is the usual Fourier multiplier by $m$. Hence, we deduce from the $L^p$ boundedness of the wave operator $\W$ mentioned in Proposition~\ref{prop-Lp-bound} that $\operatorname{M}_m(H)$ is bounded on $L^p(\R^d)$ if and only if $\operatorname{\operatorname{M}_m}(H_0)$ in bounded on $L^p(\R^d)$. As a consequence of this principle, we generalize the Fourier multiplier theorem and its variations to the perturbed linear Schrödinger evolution.
\begin{lemma}
\label{lemma-dist-mult}
Let $m:\R^d\to\C$ be a function in $L^\infty(\R^d)$ supported on a compact set $E$. Given $q$ in $\intervalcc{2}{+\infty}$, there exists a constant $C=C(q,V,\norm{m}_{L^\infty})$ such that for all $f$ in $L^2(\R^d)$,
\begin{equation}
\label{eq-distorted-multiplier}
\norm{\operatorname{M}_m(H)f}_{L^q}\leq C \abs{E}^{\frac{1}{2}-\frac{1}{q}}\norm{f}_{L^2}\,.
\end{equation}
\end{lemma}
\begin{proof}
Let $f$ be a function in $L^2(\R^d)$. We have
\begin{equation}
\label{eq-dist-mult-bern}
\abs{\parent{\mathcal{F}^*m\mathcal{F}}f(x)} = \abs{\int_E\e^{i\xi\cdot x}m(\xi)\mathcal{F}f(\xi)d\xi}\leq \norm{m}_{L^\infty(\R^d)}\norm{\e(\cdot,\xi)}_{L^\infty}\abs{E}^\frac{1}{2}\norm{f}_{L^2}.
\end{equation}
Then~\eqref{eq-distorted-multiplier} for $\operatorname{M}_m(H_0)$ (i.e. when $V=0$) follows from complex interpolation, and we deduce the estimate for $\operatorname{M}_m(H)$ from the intertwining property~\eqref{eq-intertwining} and from the $L^p$-bound~\eqref{eq-Lp-bound} on $\W$.
\end{proof}
Similarly, one can generalize the Mikhlin multiplier theorem (see Theorem 3.2 in the book~\cite{stein1970}) to the perturbed setting. 
\begin{lemma}[Mikhlin multiplier theorem]
\label{mikhlin}
Let $m:\R^d\setminus\{0\}\to\C$ be a multiplier and $1<p<+\infty$. Assume that for all $\alpha$ in $\N^d$ such that $0\leq\abs{\alpha}\leq\frac{d-1}{2}+1$, 
\begin{equation}
\label{eq-mikhlin}
\abs{\operatorname{D}^\alpha m(\xi)}\leq C_{\alpha}\abs{\xi}^{-\alpha}\,.
\end{equation}
Then,
\begin{equation*}
\norm{m(H)f}_{L^p(\R^d)}\lesssim C_\alpha\norm{f}_{L^p(\R^d)}\,.
\end{equation*}
\end{lemma}
As a consequence of the Mikhlin multiplier theorem, we show that the Sobolev spaces defined with $H$ or $H_0$ are equivalent.
\begin{lemma}
\label{lemma-sobolev}
Let $1<p<+\infty$ and $0\leq s$. There exists $C=C(s,p,d)$ such that for all $f$ in $\mathcal{S}(\R^d)$
\begin{equation*}
C^{-1}\norm{\langle\sqrt H_0\rangle^sf}_{L^p}\leq\norm{\langle\sqrt{H}\rangle^sf}_{L^p}\leq C\norm{\langle\sqrt H_0\rangle^sf}_{L^p}
\end{equation*}
\end{lemma}
\begin{proof}
Let $q$ be the conjugated exponent of $p$. Let us show that
$\langle\sqrt H_0\rangle^s\langle\sqrt H\rangle^{-s}$ 
can be extended to a bounded operator in $\mathcal{B}(L^p)$.~\footnote{\ The proof provides the same result for the operator $\langle\sqrt H\rangle^s\langle\sqrt H_0\rangle^{-s}$.} We proceed by using complex interpolation and duality in $L^p$. Take $f,g$ in $\mathcal{S}(\R^d)$ such that the support of $\widehat{f}$ is compact, and note that such functions are dense in $L^p$. Next, we consider the map 
\[F:z\mapsto \left( f\mid\e^{z^2}\langle\sqrt H_0\rangle^z\langle\sqrt H\rangle^{-z} g\right)\]
with domain $D=\set{z\in\C\mid 0< Re(z)<2n}$ for a given $n\in\N$. We see that $F$ is a function that is analytic on $D$ and  continuous on the closure of $D$. Moreover, it follows from the Mikhlin multiplier theorem and its distorted version~\eqref{eq-mikhlin} that for $f$ in $L^p$ with $1<p<\infty$ and $\delta$ in $\R$ we have 
\begin{equation*}
\abs{F(\e^{i\delta})}\lesssim \norm{f}_{L^p}\norm{g}_{L^q}\,.
\end{equation*}
Indeed, 
\begin{equation*}
\norm{\langle\sqrt H_0\rangle^{i\delta}f}_{L^p} \lesssim (1+\abs{\delta})^N\norm{f}_{L^p},\quad \norm{\langle\sqrt H\rangle^{i\delta}f}_{L^p} \lesssim (1+\abs{\delta})^N\norm{f}_{L^p}\,.
\end{equation*}
Note that the constant can be made independent of $\delta$ thanks to the term $\e^{-\delta^2}$ that arises from the function $\e^{z^2}$. In addition, the distorted version of the Mikhlin multiplier theorem and the fact that $V$ is in $\mathcal{S}$ yield
\begin{equation*}
\norm{\langle\sqrt H_0\rangle^{2n}\langle\sqrt H\rangle^{-2n}}_{\mathcal{B}(L^p)}\lesssim\sum_{i+j+k=n}\norm{V^iH^kV^j\langle\sqrt H\rangle^{-2n}}_{\mathcal{B}(L^p)}=C_{n,p}<+\infty\,,
\end{equation*}
and we deduce from this that 
\begin{equation*}
\abs{F(\e^{2n+i\delta})}\lesssim \norm{f}_{L^p}\norm{g}_{L^q}\,.
\end{equation*} 
It follows from the \textit{three lines lemma}, we prove that $\abs{F(z)}\lesssim \norm{f}_{L^p}\norm{g}_{L^q}$ on $D_n$ uniformly in $n$, and the  density of functions like $f$ and $g$ in $L^p$ and $L^q$ yields the desired result.
\end{proof}
A straightforward consequence of Lemma~\ref{lemma-sobolev} is the generalization to the perturbed setting of the fractional Leibniz rule.
\begin{lemma}[Fractional Leibniz rule]
Let $1< p, p_1, p_2, q_1, q_2  <+ \infty,\ \frac{1}{p} = \frac{1}{p_1}+\frac{1}{p_2} = \frac{1}{q_1}+\frac{1}{q_2}$, and let $0\leq s$. Assume that $f,g$ are in $\mathcal{S}(\R^d)$. Then,
\begin{equation*}
\label{eq-fractional}
\norm{\langle\sqrt H\rangle^s\parent{fg}}_{L^p}\lesssim \norm{\langle\sqrt H\rangle^sf}_{L^{p_1}}\norm{g}_{L^{p_2}} + \norm{\langle\sqrt H\rangle^sg}_{L^{q_1}}\norm{f}_{L^{q_2}}\,.
\end{equation*}
\end{lemma}
We prove in the next lemma that if $m$ is a radial function, the Fourier multiplier by $m$ is precisely the spectral multiplier defined by the functional calculus on self-adjoint operators.
\begin{lemma}
\label{lemma-radial}
If there exists a Borel function $f:\ \R\to\C$ such that $m(\xi) = f(\abs{\xi}^2)$ for all $\xi$ in $\R^d$, then
\begin{equation}
\label{eq-radial}
\operatorname{M}_m(H) = f(H)\,.
\end{equation}
\end{lemma}
\begin{proof}
Let us first write the result for the free Laplacian $H_0$, that is $\mathcal{F}^*m\mathcal{F}= f(H_0)$. By an approximation argument, the proof reduces to the cases where $f$ is a Schwartz function. In this setting, $f(H_0)$ is given by the formula
\begin{equation*}
f(H_0) = \int_\R \e^{itH_0} \widehat{f}(t)dt,\quad \text{where}\quad \e^{itH_0} = \mathcal{F}^* \e^{it\abs{\, \cdot\, }^2}\mathcal{F}\,.
\end{equation*}
By applying the Fourier inversion formula for $f$, we have that
\[f(H_0) =  \mathcal{F}^*\parent{ \int_\R \e^{it\abs{\, \cdot\, }^2}\widehat{f}(t)dt } \mathcal{F} = \mathcal{F}^*f(\abs{\, \cdot\, }^2)\mathcal{F} = \operatorname{M}_m(H_0)\,.\]
Hence, 
\begin{equation*}
\mathcal{F}f(H_0)\mathcal{F}^*=\operatorname{M}_m(H_0)\,.
\end{equation*}
Also, the intertwining property~\eqref{eq-intertwining} and the equation~\eqref{eq-w-f} give
\[f(H)= \operatorname{W}_+ f(H_0) \operatorname{W}_+^* = \mathcal{F}_V^*\mathcal{F} f(H_0) \mathcal{F}^*\mathcal{F}_V = \mathcal{F}_V^*\operatorname{M}_m \mathcal{F}_V = \operatorname{M}_m(H)\,.\qedhere\]
\end{proof}
\subsection{Application to the distorted Littlewood-Paley theory}
There exists a nonnegative radial function $\psi$ in $\mathcal{C}^\infty(\R^d)$ such that $\varphi=\psi-\psi(2\cdot)$ is supported on $\{x\in\R^d\mid \frac{1}{2}\leq\abs{x}\leq2\}$ and satisfies for all $x$ in $\R^d\setminus\set{0}$
\[\sum_{N\in2^\Z}\varphi(N^{-1}x)=1.\]
This partition of unity provides frequency decomposition on dyadic annuli for functions $f$ in $L^2(\R^d)$ called the Littlewood-Paley decomposition. We consider such a decomposition for both the flat Fourier transform and the distorted one:
\begin{equation}
\label{eq-littlewood-paley}
f = \sum_{N\in2^\Z}\Delta_Nf = \Delta_{\leq1}f+\sum_{N\in2^\N}\Delta_Nf,
\end{equation}
where $\displaystyle\Delta_N\coloneqq \varphi(N^{-1}H_0)$ is the spectral multiplier around frequencies of size $N$ for $H_0$ as well as
\begin{equation}
\label{eq-dist-lp}
f = \sum_{N\in2^\Z}\widetilde\Delta_Nf = \widetilde\Delta_{\leq1}f+\sum_{N\in2^\N}\widetilde\Delta_Nf,
\end{equation}
where $\displaystyle\widetilde\Delta_N\coloneqq \varphi(N^{-1}H)$. Note that the Bernstein's estimates are still true in the perturbed setting, as a consequence of the intertwining property, and of the $L^p$ bound on the wave operators. For $1\leq p\leq q\leq\infty$, $s\geq0$ there exist $C_{p,q}$ and $C_{p,s}$ and  such that for all $N\geq1$,
\begin{equation}
\begin{split}
\label{eq-bernstein-v}
\norm{\widetilde\Delta_Nf}_{L^q(\R^d)}&\leq C_{p,q}N^{d(\frac{1}{p}-\frac{1}{q})}\norm{\widetilde\Delta_Nf}_{L^p(\R^d)}\,,\\
\norm{\widetilde\Delta_N f}_{W^{s,p}(\R^d)}&\leq C_{p,s} N^s\norm{\widetilde\Delta_Nf}_{L^p(\R^d)}\,.
\end{split}
\end{equation}
Given  $f$ in $L^2(\R^d)$, the associated distorted Littlewood-Paley square function is defined by
\label{def-lpsf}
\begin{equation*}
\Lambda f(x) = \parent{\sum_{N\in2^\Z} \abs{\widetilde\Delta_Nf(x)}^2}^{1/2}.
\end{equation*}
One can see from Plancherel's theorem that the $L^2(\R^d)$-norm of the Littlewood-Paley square function is equivalent to the $L^2(\R^d)$-norm of the function itself. The so-called \textit{Littlewood-Paley square function} theorem extends this result to the $L^p(\R^d)$-norms for $1<p<\infty$. We state a generalized version of this theorem in the perturbed framework.
\begin{proposition}
\label{prop-lpsf}
For $1<p<+\infty$ and $f$ in $L^p(\R^d)$, the Littlewood-Paley square function $\Lambda f$ is in $L^p(\R^d)$, and
\begin{equation}
\label{eq-LP}
\norm{\Lambda f}_{L^p(\R^d)}\sim \norm{f}_{L^p(\R^d)}\,.
\end{equation}
The constant involved in the above equivalence only depends on $p$ and on the cutoff function $\psi$ used in the Littlewood-Paley decomposition.
\end{proposition}
We outline the classical proof (see Theorem 5 in the book~\cite{stein1970}) that relies on the Mikhlin multiplier theorem from Lemma~\ref{mikhlin}. The strategy is to randomize the sum $\Delta f$ with independent Rademacher variables, and to use Khinchin's estimates to reduce the proof to the $L^2$-case. As mentioned above, this case is a straightforward consequence of Plancherel theorem.
\begin{proof}
Let us fix $1<p<+\infty$ and $f$ in $L^2(\R^d)\cap L^p(\R^d)$. Take $(\varepsilon_N)_{N\in2^\Z}$ a sequence of independent Rademacher variables on a probability space $(\Omega,\mathcal{A},\mathbf{P})$, \textit{i.e.} $\varepsilon_N\sim\mathcal{U}(\set{-1,1})$. Next, we define the random Fourier multiplier for $\omega$ in $\Omega$ and $\xi$ in $\R^d$ by
\[m^\omega(\xi)=\sum_{N\in2^\Z}\varepsilon_N(\omega)\varphi(N^{-1}\xi)\,.\]
Since $\varphi(N^{-1}\cdot)$ is supported on frequencies of size $\sim N$, the random multiplier $m^\omega$ satisfies the Mikhlin's assumption : for all $\xi$ in $\R^d$ and $\alpha$ a multi-index, 
\begin{equation*}
\abs{\operatorname{D}^\alpha m^\omega (\xi)} \leq \sum_{N\in2^\Z} N^{-\abs{\alpha}}\abs{(\operatorname{D}^\alpha\varphi)(N^{-1}\xi)} \leq C_\alpha \abs{\xi}^{-\abs{\alpha}}\,.
\end{equation*}
Besides, $m^\omega$ is a radial multiplier, and it follows from Lemma~\ref{lemma-radial} and from Lemma~\ref{mikhlin} that
\begin{equation}
\label{eq-momega}
\norm{m^\omega(H)f}_{L^p(\R^d)}\leq\norm{f}_{L^p(\R^d)}.
\end{equation}
That being said, the Khinchin's estimate claims that
\begin{equation}
\label{eq-kin-rev}
\parent{1-\frac{1}{p}}^{1/2}\parent{\sum_n\abs{c_n}^2}^{1/2}\lesssim \biggl(\mathbf{E}\left\lvert\sum_n\varepsilon_N(\omega)c_n\right\rvert^p\biggr)^{1/p}\,,
\end{equation}
and yields 
\begin{equation*}
\Lambda f = \biggl(\sum_{N\in2^\Z}\abs{\widetilde\Delta_Nf}^2\biggr)^{1/2}\lesssim \biggl(\mathbf{E}\abs{\sum_{N\in2^\Z}\varepsilon_N(\omega)\widetilde\Delta_Nf}^p\biggr)^{1/p}\lesssim \bigl(\mathbf{E}\abs{m^\omega(H)f}^p\bigr)^{1/p}\,.
\end{equation*}
By taking the $L^p(\R^d)$-norm on both sides and by using~\eqref{eq-momega}, we have
\begin{equation*}
\norm{\Lambda f}_{L^p(\R^d)} \lesssim \mathbf{E}\norm{m^\omega(H)f}_{L^p(\R^d)} \lesssim \norm{f}_{L^p(\R^d)}\,.
\end{equation*}
We prove the reversed inequality by duality.
\end{proof}
\section{Probabilistic and bilinear Strichartz estimates in the perturbed setting }
\label{sec-imp-str}
\subsection{Wiener randomization with respect to the distorted Fourier transform}
\label{sec-rand}
Let us now detail how to adapt the Wiener randomization for the operator $-\Delta+V$ and to generalize probabilistic Strichartz estimates. This will be made possible thanks to the distorted Fourier transform presented in the above section. We shall follow the framework detailed by~\cite{benyi2015} which can be adapted with no difficulty to the perturbed case.
\subsubsection{Wiener decomposition on unit cubes}
\label{section-wiener-decomposition}
We denote the unit cube $\intervalcc{-1}{1}^d$ by $Q_0$ and the translated cube centered around $n\in\Z^d$ by $Q_n=Q_0+n$. Then, we take a well-chosen bump function $\psi$ supported on $Q_0$ that provides a partition of unity on the frequency space, that is $\sum_{n\in\Z^d}\psi(\cdot-n)\equiv1$, and we define the Wiener's decomposition of a function in $L^2(\R^d)$ as
\begin{equation}
\label{eq-part-wiener}
f=\sum_{n\in\Z^d}\operatorname{M}_{\psi_n}(H)f\,,
\end{equation}
where we defined the distorted Fourier multiplier $\operatorname{M}_{\psi_n}(H)$ in~\eqref{eq-def-mult}. After that, we consider a probability space $\parent{\Omega,\mathcal{A},\mathbf{P}}$ and a sequence of mean zero complex valued random variables $\parent{g_n}_{n\in\Z^d}$ of laws $\parent{\mu_n}_{n\in\Z^d}$, with uniform bound : 
\begin{equation}
\label{hyp-va}
\abs{\int_{\R}\e^{\gamma x}d\mu_n(x)}\leq\e^{c\gamma^2}\quad \text{for all $n$ in $\Z^d$ and $\gamma$ in $\R$}\,.
\end{equation}
Such a bound is satisfied by Gaussian random variables, Bernoulli variables or any random variables with compactly supported distributions. We also assume that $(\operatorname{Re}(g_n),\operatorname{Im}(g_n))_{n\in\Z^d}$ are independent variables. 
\begin{definition}
\label{def-rand}
The Wiener randomization of a function $u_0$ in $L^2(\R^d)$ is 
\begin{equation}
\label{eq-wiener-rand}
u_0^\omega \coloneqq \sum_{n\in\Z^d}g_n(\omega)\operatorname{M}_{\psi_n}(H)u_0\,.
\end{equation}
\end{definition}
The Khinchin's inequality asserts that
\begin{equation}
\label{eq-kin}
\left(\mathbf{E}\abs{\sum_ng_n(\omega)c_n}^p\right)^{1/p}\lesssim\sqrt{p}\parent{\sum_n\abs{c_n}^2}^{1/2}\quad \text{for}\ (c_n)\in \ell^2,\quad 1\leq p<+\infty\,,
\end{equation}
and we deduce from it that the randomization procedure improves integrability and preserves regularity. Namely, for all $2\leq p <+\infty$ and $s$ in $\R$ there exist some constants $0<c$, $0<C$ such that for all $0<\alpha$, $u$ in $L^2(\R^d)$ and $v$ in $\mathcal{H}^s(\R^d)$ we have
\begin{equation*}
\textbf{P}(\norm{u^\omega}_{L^p(\R^d)}>\lambda)\leq C\e^{-c\lambda^2\norm{u}_{L^2(\R^d)}^{-2}},\quad \textbf{P}(\norm{v^\omega}_{\mathcal{H}^s(\R^d)}>\lambda)\leq C\e^{-c\lambda^2\norm{u}_{\mathcal{H}^s(\R^d)}^{-2}}.
\end{equation*}
We emphasize that the procedure does not improve regularity provided that the distribution $\mu_n$ does not concentrate around zero when $n$ goes to $\infty$, i.e. $u_0^\omega\in\mathcal{H}^s(\R^d)\setminus\mathcal{H}^{s+}(\R^d)$ almost surely when $u_0\in\mathcal{H}^s(\R^d)\setminus\mathcal{H}^{s+}(\R^d)$ (see Lemma B.1 in~\cite{burq2008}). As for the possible gain of decay, we use the independence of the $\set{g_n}_{n\in\Z^d}$'s, and we observe that for $\epsilon>0$, 
\begin{equation*}
\begin{split}
\mathbf{E}\parent{\norm{\abs{x}^\epsilon u_0^\omega}_{L_x^2}^2}&=\sum_{(n,m)\in\Z^{2d}}\mathbf{E}\parent{g_n\overline{g_m}}\int \abs{x}^{2\epsilon}\operatorname{M}_{\psi_n}(H)u_0 \overline{\operatorname{M}_{\psi_m}(H)f}dx \\
&= \sum_{n\in\Z^d}\norm{\abs{x}^\epsilon\operatorname{M}_{\psi_n}(H)u_0}_{L_x^2(\R^d)}^2
\end{split}
\end{equation*}
diverges if we assume that 
\begin{equation}
\label{eq:ass:decay}
\sum_{n\in\Z^d}\norm{\abs{x}^\epsilon\operatorname{M}_{\psi_n}(H)u_0}_{L_x^2(\R^d)}^2 = +\infty\,.
\end{equation}
Hence, $\norm{\abs{x}^\epsilon u_0^\omega}_{L^2(\R^d)}$ diverges in $L^2(\Omega)$ when $u_0$ is chosen such that~\eqref{eq:ass:decay}. This indicates that the randomization does not gain decay in general.~\footnote{\ Since $\norm{\abs{x}^\epsilon \sum_n \operatorname{M}_{\psi_n} u_0 }_{L^2}^2\neq \sum_n\norm{\abs{x}^\epsilon \operatorname{M}_{\psi_n} u_0}_{L_x^2}^2$, we are not able to reproduce the proof of Lemma B.1 in~\cite{burq2008}, and to preclude any gain of decay by our randomization procedure almost-surely.} Still, the initial data $\psi_0=u_0^\omega+v_0$ in the Cauchy problem~\eqref{cauchy-soliton} does not decay for a general $v_0$ in $\mathcal{H}^{1/2}$. Besides, we don't use any decay of the data in our arguments to prove Theorem~\ref{theorem-soliton}.
\subsubsection{Improved probabilistic Strichartz estimates}
Let us now recall how to improve the Strichartz estimates~\eqref{eq-pertrubed-str} for data randomized according to the distorted Wiener procedure~\eqref{eq-wiener-rand} (see Lemma 2.3 in~\cite{benyi2015}).
\begin{proposition}[Improved global-in-time Strichartz estimate]
\label{prop-rand-str}
Given a Schrödinger admissible pair $(q,r)$ as in Proposition~\eqref{prop-pertrubed-str}, a real number $\widetilde r$ with $r\leq\widetilde r <+\infty$ and a function $f\in \mathcal{H}^s(\R^d)$ for some $s\geq0$, there exist constants $0<C$, $0<c$ such that 
\begin{equation}
\label{eq-rand-str}
\mathbf{P}\parent{\norm{\langle\sqrt{H}\rangle^s\e^{-itH}u_0^\omega}_{L_t^q(\R ;L_x^{\widetilde r}(\R^d))}>\lambda}\leq C\exp\parent{-c\lambda^2\norm{f}_{\mathcal{H}_x^s(\R^d)}^{-2}}\,.
\end{equation}
Consequently, $\norm{\langle\sqrt{H}\rangle^s\e^{-itH}u_0^\omega}_{L_t^q(\R;L_x^{\widetilde r}(\R^d))} <+\infty$ almost-surely when $f$ is in $\mathcal{H}^s$.
\end{proposition}
The strategy of the proof is to exploit the enhanced space integrability of Wiener randomized data. More precisely, the unit scale frequency components $\psi_n(H)f$ benefit from better integrability (uniformly in $n$) thanks to the distorted Fourier multiplier theorem, while the randomization procedure cancels interference between these different pieces $\psi_n(H)f$. Hence, the space integrability of $\e^{-iH}u_0^\omega$ is improved and one can apply the deterministic Strichartz estimate for the admissible pair $(q,r)$. The main point of our revisited proof is that the distorted Fourier transform commutes with the perturbed linear flow $\e^{-itH}$, and satisfies the unit-scale Bernstein estimate of Lemma~\ref{lemma-dist-mult}.
\begin{proof}
We write the proof for $s=0$. The large deviation bound~\eqref{eq-rand-str} will be deduced from an estimate on the moments of the random variable $\disp\norm{\e^{-itH}u_0^\omega}_{L_t^q(\R;L_x^{\widetilde r}(\R^d))}$. To get such an estimate, we take $p$ with $\max(q,\widetilde r)<p$ and we use Minkowski inequality to get
\begin{equation*}
\mathbf{E}\parent{\norm{\e^{-itH}u_0^\omega}_{L_t^q(\R;L_x^{\widetilde r}(\R^d))}^p}^{1/p}\leq \norm{\e^{-itH}u_0^\omega}_{L_t^q(\R;L_x^{\widetilde r}(\R^d; L_\omega^p(\Omega)))}\,.
\end{equation*}
At fixed $(t,x)\in\R\times\R^d$, it follows from Khinchin inequality~\eqref{eq-kin} that
\begin{equation*}
\norm{\e^{-itH}u_0^\omega}_{L_\omega^p(\Omega)}=\norm{\sum_{n\in\Z^d}g_n(\omega)\e^{-itH}\operatorname{M}_{\psi_n}(H)f}_{L_\omega^p(\Omega)}\lesssim\sqrt{p}\norm{\operatorname{M}_{\psi_n}(H)\e^{-itH}f}_{\ell_n^2(\Z^d)}\,.
\end{equation*}
Hence, we use Minkowski inequality once again with $2\leq \min(q,\widetilde r)$, and get 
\begin{equation}
\label{eq-pr-rand}
\mathbf{E}\parent{\norm{\e^{-itH}u_0^\omega}_{L_t^q(\R;L_x^{\widetilde r}(\R^d))}^p}^{1/p}\lesssim\sqrt{p}\norm{\operatorname{M}_{\psi_n}(H)\e^{-itH}f}_{\ell_n^2(\Z^d;L_t^q(\R;L_x^{\widetilde r}(\R^d)))}\,.
\end{equation}
Next, we use the distorted Fourier multiplier Lemma~\ref{lemma-dist-mult} at fixed $t$ with $\widetilde r,r$ and the unit-scale multiplier $\xi\mapsto\e^{-it\abs{\xi}^2}\psi(\xi-n)$ (in particular the volume of the support does not depend on $n$):
\begin{equation*}
\norm{\e^{-itH}\operatorname{M}_{\psi_n}(H)f}_{L_x^{\widetilde r}(\R^d)}\lesssim_{r,\widetilde r} \norm{\e^{-itH}\operatorname{M}_{\psi_n}(H)f}_{L_x^r(\R^d)}\,.
\end{equation*}
Now, we apply the deterministic global-in-time Strichartz estimate for the perturbed Schrödinger evolution~\eqref{eq-pertrubed-str} with the admissible pair $(q,r)$ and we get 
\begin{equation*}
\norm{\e^{-itH}\operatorname{M}_{\psi_n}(H)f}_{L_t^q(\R;L_x^r(\R^d))}\lesssim\norm{\operatorname{M}_{\psi_n}(H)f}_{L_x^2(\R^d)}\,. 
\end{equation*}
Hence, there exists $0<C$ such that
\begin{equation*}
\mathbf{E}\parent{\norm{\e^{-itH}u_0^\omega}_{L_t^q(\R;L_x^{\widetilde r}(\R^d))}^p}^{1/p}\lesssim\sqrt{p}\,\norm{\operatorname{M}_{\psi_n}(H)f}_{\ell_n^2L_x^2(\Z^d\times\R^d)} \leq \sqrt{p}\,C\norm{f}_{L_x^2(\R^d)}\,. 
\end{equation*}
To conclude, we apply Markov inequality to get
\begin{equation*}
\mathbf{P}\parent{\norm{\e^{-itH}u_0^\omega}_{L_t^q(\R;L_x^{\widetilde r}(\R^d)}>\lambda}\leq \parent{\sqrt{p}\,C\lambda^{-1}\norm{f}_{L_x^2(\R^d)}}^p\,.
\end{equation*}
After that we chose $p$ such that $\sqrt{p}\coloneqq \e^{-1}\parent{C\lambda^{-1}\norm{f}_{L_x^2}}^{-1}$ and we distinguish the cases when $\max(q,\widetilde r)< p$ or not.
\end{proof}
Given $0<\epsilon,R$ we define the set
\begin{equation}
\label{set-eps-R}
\widetilde\Omega_{\epsilon,R} =\bigcup_{(q,\widetilde r)}\set{\omega\in\Omega\mid \norm{\epsilon u^\omega}_{L_t^qL_x^{\widetilde r}(\R\times\R^d)}\leq R}\,,
\end{equation}
where the union is taken over a finite number of pairs $(q,\widetilde r)$ as in Proposition~\ref{prop-rand-str} that occur in the nonlinear analysis presented in the section~\ref{section-continuous}. It follows from~\eqref{eq-rand-str} that there exist $0<c,C$ such that for all $0<\epsilon,R$
\begin{equation*}
\mathbf{P}(\Omega\setminus\widetilde\Omega_{\epsilon,R})\leq C\exp(-R^2\epsilon^{-2}\norm{u_0}_{\mathcal{H}^s}^{-2})\,.
\end{equation*}
When $R$ is an irrelevant universal constant we denote this set by~$\widetilde\Omega_\epsilon$.
\subsection{Bilinear estimate for the perturbed Schrödinger evolution}
\label{sec-bili}
Bourgain's bilinear estimate states that given $N,M$ two dyadic numbers with $N\leq M$ and $u_0,v_0$ in $L^2(\R^d)$ localized in the Fourier space, say $\supp \widehat u_0\subseteq\set{\abs{\xi}\lesssim N}$ and $\supp \widehat v_0\subseteq\set{\abs{\xi}\sim M}$, we have
\begin{equation}
\label{eq-bili}
\norm{(\e^{it\Delta}u_0)(\e^{it\Delta}v_0)}_{L^2_{t,x}(\R\times\R^d)}\lesssim N^\frac{d-1}{2}M^{-\frac{1}{2}}\norm{u_0}_{L_x^2(\R^d)}\norm{v_0}_{L_x^2(\R^d)}\,.
\end{equation}
We refer to Lemma 5 in~\cite{bourgain1998} for the original proof of this result, and to~\cite{ckstt2008} for a more detailed proof. Basically, the proof relies on Plancherel's formula, on the explicit form of the free Schrödinger flow and on the multiplicative property of the flat Fourier transform. As mentioned in the first section, the distorted Fourier transform does not have such a property. That being said, we propose to generalize~\eqref{eq-bili} to the perturbed case, and we consider the interaction of two initial data localized by perturbed Littlewood-Paley projectors $\widetilde \Delta_N$,   $\widetilde\Delta_M$ introduced in~\eqref{eq-dist-lp} that evolve under the linear flow $\e^{-itH}$.
\begin{proposition}[Bilinear estimate for perturbed linear Schrödinger evolutions]
\label{prop-bili-V}
Given $N, M$ two dyadic integers, we have that for all $u_0,v_0\in L^2(\R^d)$,
\begin{equation}
\label{eq-bili-v}
\norm{(\e^{-itH}\widetilde\Delta_Nu_0)(\e^{-itH}\widetilde\Delta_Mv_0)}_{L_{t,x}^2(\R\times\R^d)} \lesssim N^\frac{d-1}{2}M^{-\frac{1}{2}} \norm{\widetilde\Delta_Nu_0}_{L_x^2(\R^d)}\norm{\widetilde\Delta_Mv_0}_{L_x^2(\R^d)}\,. 
\end{equation}
\end{proposition}
The general strategy of our proof is to reduce~\eqref{eq-bili-v} to~\eqref{eq-bili}. To do so, we use a perturbative argument that relies on the Duhamel formula that connects the free evolution to the perturbed one. More precisely, given $N,K$ two dyadic numbers and $u=\e^{-itH}\widetilde\Delta_N u_0$ a solution to the perturbed linear Schrödinger equation with initial data localized around a distorted Fourier frequencies of size $N$, we have
\begin{equation}
\label{eq-duhamel-inter}
\Delta_K u(t) = \e^{it\Delta}\Delta_K\widetilde\Delta_Nu_0-i\int_0^t\e^{i(t-\tau)\Delta}\Delta_K(Vu(\tau))d\tau\,.
\end{equation}
Hence, we will need to estimate some space-time Lebesgue norms
of terms like $\Delta_K(V\widetilde\Delta_N\e^{-itH}u_0)$.
\subsubsection{Semiclassical functional calculus}
\label{sub-sec-semiclassical}
This section is devoted to the statement of some lemmas used in the proof of Proposition~\ref{prop-bili-V}. When studying the interactions of a solution localized around a high frequency $N\gg 1$ with other solutions, we introduce the small parameter $h=N^{-1}$ to place ourselves in the semiclassical analysis framework that provides precise asymptotic expansions with respect to $h$. This strategy, that can essentially be used when $V$ is smooth, provides an easy but efficient way to intertwine Fourier and distorted Fourier localization. For the convenience of the reader, we detail some notations and results from semiclassical analysis (we refer to~\cite{zworski2012} for a general presentation). The Weyl quantization of a given symbol $a\in C^\infty(\R^{2d})$ is 
\begin{equation}
\label{def-weyl}
\Op_h(a)u(x) \coloneqq(2\pi h)^{-2d} \iint_{\R^d\times\R^d}\e^{i\frac{(x-y)\cdot\xi}{h}}a(\frac{x+y}{2},\xi)u(y)dyd\xi\,.
\end{equation}
It defines an operator acting on $\mathcal{S}(\R^d)$ and for instance $h^2H_0=\Op_h(\abs{\xi}^2),\ h^2H=\Op_h(\abs{\xi}^2+h^2V)$. Moreover, if the symbol $a$ is in the class
\[S(1)\coloneqq\set{a(h)\in C^\infty(\R^{2d})\mid \forall \alpha,\beta\in\N^d\ \underset{h\in\intervalco{0}{1}}{\sup}\underset{\rho\in\R^{2d}}{\sup}\abs{\partial_x^\alpha\partial_\xi^\beta a(\rho)}<\infty}\,,\]
we can extend $\Op_h(a)$ to a bounded operator on $L^2(\R^d)$ thanks to the Calderon-Vaillancourt theorem that yields :
\begin{equation}
\label{eq-calderon}
\norm{\Op_h(a)}_{\mathcal{B}(L^2(\R^d))}\lesssim\sum_{\abs{\alpha}\leq6d+2}h^{\abs{\alpha}/2}\norm{\partial_{x,\xi}^\alpha a}_{L^\infty(\R^{2d})}\,.
\end{equation}
Moreover, given $a$ and $b$ two symbols in the class $S(1)$ there exists a symbol $c$ also in $S(1)$ such that $\Op_h(a)\circ\Op_h(b) = \Op_h(c)$.
In addition, $c$ has an explicit semiclassical asymptotic expansion of the form
\begin{equation}
\label{eq-moyal}
c(\rho)= a(\rho)b(\rho)-h\frac{i}{2}\{a,b\}(\rho)+\dots+ \frac{1}{(n-1)!}\parent{\frac{ih}{2}}^{n-1}a(\omega(\overleftarrow{D},\overrightarrow{D}))^{n-1}b(\rho) +\bigo{h^n}_{\mathcal{S}(\R^d)}\,,
\end{equation}
where $a\omega(\overleftarrow{D},\overrightarrow{D})b(\rho)\coloneqq \parent{\partial_{\xi_0} a (\rho_0)\partial_{x_1} b (\rho_1) - \partial_{\xi_1} b (\rho_1) \partial_{x_0} a (\rho_0)}_{\vert\rho_0 = \rho_1 =\rho}$. The symbol $c$ is written $c\eqqcolon (a\#_hb)$.
In particular, we observe from~\eqref{eq-calderon} and~\eqref{eq-moyal} that whenever $a$ and $b$ are in $S(1)$ with disjoint supports then
\begin{equation}
\label{eq-moyal-disj}
\Op_h(a)\circ\Op_h(b) =  \bigo{h^\infty}_{\mathcal{B}(L^2(\R^d))}\,.
\end{equation}
We can now state an approximation result of spectral multipliers by pseudo-differential operators. This result comes from~\cite{robert1987}, and was extended to more general multipliers in ~\cite{sogge17}. We chose to write a version presented in the context of compact manifolds in~\cite{bgt2004}, Proposition 2.1. The proof relies on the so-called Helffer-Sjöstrand formula and can be rewritten in our setting with no difficulty. 
\begin{lemma}
\label{lemma-approx}
Let $\eta$ in $C_c^\infty(\R)$ be a smooth cutoff function. There exists a sequence of symbols $(c_j)_{j\geq0}$ in $S(1)$ such that for every $h\in\intervaloc{0}{1},\ n\in\N$ and $0\leq\sigma\leq n$,
\begin{equation}
\label{eq-approx}
\norm{\eta(h^2H)-\sum_{0\leq j\leq n-1}h^j\Op_h(c_j)}_{\mathcal{B}(L^2(\R^d),\mathcal{H}^\sigma(\R^d))}\lesssim_{n,\sigma}h^{n}\,.
\end{equation}
Besides, for $(x,\xi)\in\R^{2d}$
\begin{equation*}
c_0(x,\xi)=\eta(\abs{\xi}^2),\quad c_j(x,\xi)=\sum_{k\geq2}\frac{1}{(k-1)!}\eta^{(k-1)}(\abs{\xi}^2)q_{j,k}(x,\xi)\,,
\end{equation*}
where the finite sum representing $c_j$ is made of functions $q_{j,k}$ which are polynomials of degree less than $2(k-1)$ in the frequency variable $\xi$, and they are Schwartz functions with respect to the space variable $x$. 
\end{lemma}
We emphasize that $\supp c_j\subseteq\set{(x,\xi)\in\R^{2d}\mid\abs{\xi}^2\in\supp\eta}$.
Hence, we can approximate $\widetilde\Delta_N\coloneqq \varphi(N^{-1}H)$ by pseudo-differential operators supported on frequencies of size $N$. Since the Fourier multiplier $\Delta_K\coloneqq \varphi(K^{-1}H_0)$ is already a pseudo-differential operator localized on frequencies of size $K$, we can use the semiclassical asymptotic expansion~\eqref{eq-moyal} to see that $\Delta_K\widetilde\Delta_N$ is negligible when $1\ll\abs{K-N}$. Namely, for all $s,\sigma\in\R$ and $0<\alpha$ there exists $C_{s,\sigma,\alpha}$ such that for all $N=2^n$ and $K=2^k$ with $3\leq\abs{k-n}$,
\begin{equation}
\label{eq-negligible}
\norm{\Delta_K\widetilde\Delta_N}_{\mathcal{B}(\mathcal{H}^{-s},\mathcal{H}^\sigma)}\leq C_{s,\sigma,\alpha}2^{-\alpha\max(n,k)}\,.
\end{equation}
Therefore one can intertwine localization with respect to spectral Littlewood-Paley multiplier for $H$ and $H_0$ up to a negligible term. The next lemma, written in light of equation~\eqref{eq-duhamel-inter}, encapsulate the above discussion. 
\begin{lemma}
\label{lemma-loc-weight}
For any $\alpha\in\N$, there exists $C=C(\alpha,V)>0$ depending on some weighted norms of $\ V$ and its derivatives~\footnote{\ More precisely, our brutal computations give $C\lesssim\norm{\japbrak{\nabla}^{\frac{3d}{2}+3\alpha}\parent{\japbrak{x}^2V}}_{L_x^2}$, but they are far from being optimal.} such that for any dyadic integer $K=2^k$ and $f\in L_x^2(\R^d)$, we have
\begin{equation}
\label{eq-loc2}
\parent{\sum_{\abs{k-l}\geq 3}\norm{\Delta_K(V\Delta_Lf)}_{L_x^2}^2}^{1/2}\leq C K^{-\alpha} \norm{\japbrak{x}^{-2}f}_{L_x^2}\,.
\end{equation}
\end{lemma}
Note that we decided to exploit not only the localization in frequency, but also the decay of $V$ in order to use the local smoothing effect later on.
\begin{proof}
We prove that there exists $C_\alpha>0$ depending on weighted norms of the potential $V$ and its derivatives such that for all $L=2^{-l}$ with $3\leq\abs{k-l}$, we have
\begin{equation}
\label{eq-loc1}
\norm{\Delta_K(V\Delta_Lf)}_{L_x^2}\leq C 2^{-2\alpha\max(k,l)}\norm{\japbrak{x}^{-2}f}_{L_x^2}\,.
\end{equation}
Let us first consider the case when $k+3\leq l$. Plancherel's formula yields
\begin{equation*}
\norm{\Delta_K(V\Delta_L\japbrak{x}^2f)}_{L^2(\R^d)} = \sup_{\norm{g}_{L^2(\R^d)}\leq 1}\abs{ \iint_{\R^d\times\R^d} \varphi(K^{-1}\xi)\varphi(L^{-1}\eta)\widehat V(\xi-\eta)\japbrak{\nabla_\eta}^2\widehat f(\eta)\overline{\widehat g(\xi)} d\xi d\eta}\,,
\end{equation*}
where we used that $\mathcal{F}_{x\to\eta}\japbrak{x}^2f\sim\japbrak{\nabla_\eta}^2\mathcal{F}_{x\to\eta}f$. Integrating by parts~\footnote{\ We stress out that the operator  $\japbrak{\nabla_\eta}^2\sim1+\sum_{i=1}^d\partial_{\eta_i}^2$ is a local operator and has no effect on the frequency localization $\eta\sim L$.} with respect to $\eta$ gives
\begin{equation*}
\norm{\Delta_K(V\Delta_L\japbrak{x}^2f)}_{L^2(\R^d)} = \sup_{\norm{g}_{L^2(\R^d)}\leq 1}\abs{ \iint_{\R^d\times\R^d} \varphi(K^{-1}\xi)\japbrak{\nabla_\eta}^2\parent{\varphi(L^{-1}\eta)\widehat V(\xi-\eta)}\widehat f(\eta)\overline{\widehat g(\xi)} d\xi d\eta}\,.
\end{equation*}
Since $\varphi_K$ (resp. $\varphi_L$) is supported on $2^{k-1}\leq\abs{\xi}\leq2^{k+1}$ (resp. $2^{l-1}\leq\abs{\eta}\leq2^{l+1}$), we have that $2^{l-2}\leq\abs{\xi-\eta}$ on the support of the integrand. Therefore, for all $0\leq\delta$ we have that $1\leq2^{-\delta (l-2)}\abs{\xi-\eta}^\delta$, and the right-hand-side of the above estimate is less than
\begin{equation*}
\begin{split}
&\lesssim2^{-\delta (l-2)}\iint_{\R^d\times\R^d}\abs{ \varphi(K^{-1}\xi)\abs{\xi-\eta}^{\delta}\japbrak{\nabla_\eta}^2\parent{\widehat \varphi(L^{-1}\eta)\widehat V(\xi-\eta)}\widehat f(\eta)\overline{\widehat g(\xi)}} d\xi d\eta\\
&\lesssim2^{-\delta (l-2)}\norm{\japbrak{\nabla}^\delta\parent{\japbrak{x}^2V}}_{L_x^2}\norm{f}_{L_x^2}\norm{g}_{L_x^2}\norm{\varphi(K^{-1}\cdot)}_{L_\xi^2}\\
&\leq C(V)2^{-\delta(l-2)+\frac{d}{2}k}\norm{f}_{L_x^2}\leq C(V)2^{-\alpha l}\norm{f}_{L_x^2}\,,
\end{split}
\end{equation*}
by choosing $\delta\geq \frac{d}{2}+3\alpha$, and by using Cauchy-Schwarz and Plancherel. In the case when $l+3\leq k$, the same computations yield
\[
\norm{\Delta_K(V\Delta_L\japbrak{x}^2f)}_{L^2(\R^d)}\leq C(V)2^{-\delta(k-2)+\frac{d}{2}k}\norm{f}_{L_x^2}\leq C(V)2^{-\alpha k}\norm{f}_{L_x^2}\,,
\]
provided that $\delta \geq \frac{3d}{2}+3\alpha$. We finish the proof of estimate~\eqref{eq-loc2} by summing over $L$.
\end{proof}
\subsubsection{Proof of bilinear estimate for perturbed Schrödinger evolutions}
We now come to the proof of Proposition~\ref{prop-bili-V} itself. For now on, we use the notation 
\[u(t)\coloneqq \e^{-itH}\widetilde\Delta_Nu_0\,,\quad \quad v(t)\coloneqq \e^{-itH}\widetilde\Delta_Mv_0\,.\]
\paragraph{\textbf{Step 1 : Reduction to the case where initial data are flat Fourier localized.}} First, we reduce Proposition~\ref{prop-bili-V} to the following proposition where the initial data are also localized by flat Fourier multipliers. In what follows, $\Delta_K$ is the projector around a frequency $K$, $\Delta_{\lesssim N}$ is the projector on frequencies below $2^5N$ and $\Delta_{\sim M}$ is a fattened projector around frequency $M$. Given $N, M$ two dyadic numbers, we write 
\[u(t)\coloneqq \e^{-itH}\widetilde\Delta_Nu_0,\quad \quad v(t)\coloneqq \e^{-itH}\widetilde\Delta_Mv_0\,,\]
where $\widetilde{\Delta}_N,\widetilde{\Delta}_M$ are the distorded Fourier multipliers around frequencies $N,M$ as defined in~\eqref{eq-dist-lp}.
\begin{proposition}  There exists $C>0$ such that for any time interval $I\subseteq\R$ and any dyadic integers $N,M$ with $N\leq M$, we have
\label{prop-bilinear}
\begin{equation}
\label{eq-red1}
\norm{uv}_{L_{t,x}^2(I)}\leq \norm{\Delta_{\lesssim N}u\Delta_{\sim M}v}_{L_{t,x}^2(I)}+ \parent{\sum_{K>2^5N}\norm{\Delta_K u\Delta_{\sim M}v}_{L_{t,x}^2(I)}^2}^{1/2}
+ N^\frac{d-1}{2}M^{-\frac{1}{2}}\gamma(I)\,.
\end{equation}
For any interval $I$ and for any partition $\R=\bigcup_\ell I_\ell $, we have
\begin{equation}
\label{eq-fr-sub}
\gamma(I)\leq C\norm{\widetilde{\Delta}_Nu_0}_{L_x^2}\norm{\widetilde{\Delta}_Mv_0}_{L_x^2},\quad \parent{\sum_{\ell} \gamma(I_\ell)^2}^{1/2} \leq C\norm{\widetilde{\Delta}_Nu_0}_{L_x^2}\norm{\widetilde{\Delta}_Mv_0}_{L_x^2}\,.
\end{equation}
\end{proposition}
\begin{proof}
We introduce the fattened projector $\Delta_{\sim M}$, a flat Fourier multiplier $\widetilde \varphi(M^{-1}H_0)$ by a smooth cutoff function $\widetilde \varphi$ chosen such that $(1-\widetilde\varphi)\varphi\equiv0$. Under this assumption, we deduce from Lemma~\ref{lemma-approx} that for all $\sigma\in\R$ and for all $\alpha\in\N$, 
\begin{equation}
\label{eq-redu-pr}
\norm{(1-\Delta_{\sim M})\widetilde\Delta_M}_{\mathcal{B}(L^2(\R^d,\mathcal{H}^\sigma(\R^d)))}\lesssim_{\alpha,\sigma} M^{-\alpha}.
\end{equation}
Then, $v$ is decomposed into $v=\Delta_{\sim M}v+(1-\Delta_{\sim M}v)$, and
\begin{equation}
\label{eq-dec-pr}
uv = u\Delta_{\sim M}v + u(1-\Delta_{\sim M})v\,.
\end{equation}
We deduce form Cauchy-Schwarz, Sobolev embedding, Strichartz estimate and~\eqref{eq-redu-pr} that the second term on the right-hand side of~\eqref{eq-dec-pr} is negligible : 
\begin{equation*}
\begin{split}
\norm{u(1-\Delta_{\sim M})v}_{L_{t,x}^2(I)}&\leq \norm{u}_{L_{t,x}^4(I)}\norm{(1-\Delta_{\sim M})v}_{L_{t,x}^4(I)}\\
&\lesssim N^{\frac{d}{12}}\norm{u}_{L_t^4L_x^3(I)}\norm{\japbrak{\nabla}^{\frac{d}{12}}(1-\Delta_{\sim M})v}_{L_t^4L_x^3(I)} \\
&\leq C_\alpha M^{-\alpha+\frac{d}{12}}N^{\frac{d}{12}} \norm{\widetilde{\Delta}_Nu_0}_{L_x^2}\norm{\widetilde{\Delta}_Mv_0}_{L_x^2}\,.
\end{split}
\end{equation*}
Choosing $\alpha=\frac{d}{12}+\frac{1}{2}$, writing 
\[
\gamma(I)= (N^\frac{d-1}{2}M^{-\frac{1}{2}})^{-1}\norm{u(1-\Delta_{\sim M})v}_{L_{t,x}^2(I)}\,,
\]
and, by using sub-additivity, we easily obtain~\eqref{eq-fr-sub}. Finally,~\eqref{eq-red1} follows from a Littlewood-Paley decomposition of $u$ in the first term on the right-hand side of~\eqref{eq-dec-pr}.
\end{proof}
\paragraph{\textbf{Step 2 : Replacing the perturbed evolution $\e^{-itH}$ by the free one $\e^{it\Delta}$.}}
To prove Proposition~\ref{prop-bilinear}, it remains to estimate the first two terms on the right-hand side of~\eqref{eq-red1}. The idea is to use Bourgain's estimate~\eqref{eq-bili} in the free case. To do so, we shall understand how to describe the perturbed evolution $\e^{-itH}$ in terms of $\e^{it\Delta}$. For this purpose, we see $V$ as a forcing term, and we follow the proof of the inhomogeneous bilinear estimate from~\cite{ckstt2008}. Before that, we prove the following lemma that we will use to estimate the Duhamel term that comes from the potential, seen as the forcing term in what follows.
\begin{lemma}
\label{lemma-ck2}
For all $\alpha\in\N$, there exists $C=C(\alpha,V)$ such that for all dyadic integer $K,N$ and all $u_0\in L_x^2$ we have
\begin{equation}
\label{eq-lemma-ck2}
\parent{\int_{\R}\norm{\Delta_K\parent{V\widetilde{\Delta}_N\e^{-itH}u_0}}_{L_x^2}^2d\tau}^{1/2}\leq C \begin{cases}
K^{-\alpha}\norm{\widetilde{\Delta}_Nu_0}_{L_x^2}\quad &\text{if}\ K\geq 2^5N\,,\\
\norm{\widetilde{\Delta}_Nu_0}_{L_x^2} \quad &\text{if}\ K\lesssim N \,.
\end{cases}
\end{equation}
\end{lemma}
\begin{proof}
In the case when $K\lesssim N$, we just apply the local smoothing~\eqref{eq-local-smoothing}. Otherwise, we need to gain a negative power of $K$. To do so, we fix $K>2^5N$ and we do a Littlewood-Paley decomposition of $\widetilde{\Delta}_N\e^{-itH}u_0$ : 
\begin{multline*}
\norm{\Delta_K\parent{V\widetilde{\Delta}_N\e^{-itH}u_0}}_{L_x^2}^2\sim \sum_{\abs{l-k}\geq3}\norm{\Delta_K\parent{V\Delta_L\widetilde{\Delta}_N\e^{-itH}u_0}}_{L_x^2}^2
\\
+\sum_{\abs{l-k}\leq2}\norm{\Delta_K\parent{V\Delta_L\widetilde{\Delta}_N\e^{-itH}u_0}}_{L_x^2}^2=A+B\,.
\end{multline*}
\paragraph{\textbf{\textit{Estimate of term $A$} :}} We use Lemma~\ref{lemma-loc-weight} applied with $f=\widetilde{\Delta}_N\e^{-itH}u_0$ to get 
\[
\sum_{\abs{l-k}\geq3}\norm{\Delta_K\parent{V\Delta_L\e^{-itH}\widetilde{\Delta}_N\e^{-itH}u_0}}_{L_x^2}^2\leq C_\delta K^{-\alpha}\norm{\japbrak{x}^{-2}\e^{-itH}\widetilde{\Delta}_Nu_0}_{L_x^2}^2\,.
\]
Then, we conclude by the local smoothing estimate~\eqref{eq-local-smoothing} for $\e^{-itH}$ that
\begin{equation*}
\begin{split}
\int_\R \sum_{\abs{l-k}\geq3}\norm{\Delta_K\parent{V\Delta_L\e^{-itH}\widetilde{\Delta}_N\e^{-itH}u_0}}_{L_x^2}^2dt&\leq C_\alpha K^{-\alpha}\int_\R\norm{\japbrak{x}^{-2}\e^{-itH}\widetilde{\Delta}_Nu_0}_{L_x^2}^2dt\\
&\leq C_\alpha K^{-\alpha}\norm{\widetilde\Delta_Nu_0}_{L_x^2} \,.
\end{split}
\end{equation*}
\paragraph{\textbf{\textit{Estimate of term $B$} :}} We use~\eqref{eq-negligible} and the endpoint Strichartz estimate for $\e^{-itH}$, that is the Strichartz estimate~\ref{eq-global-str} with the admissible pair $(2,\frac{2d}{d-2})$. Let us fix $l$ such that $\abs{k-l}\leq2$ and write
\begin{equation*}
\begin{split}
\norm{\Delta_K\parent{V\Delta_L\widetilde{\Delta}_N\e^{-itH}u_0}}_{L_x^2} &= \underset{\norm{g}_{L_x^2\leq1}}{\sup}\abs{\prodscal{\Delta_K\parent{V\Delta_L\widetilde{\Delta}_N\e^{-itH}u_0}}{g}_{L_x^2}}\\
&= \abs{\prodscal{\e^{-itH}\widetilde{\Delta}_Nu_0}{\widetilde{\Delta}_N\Delta_L(V\Delta_Kg)}_{L_x^2}}\,.
\end{split}
\end{equation*}
By Hölder and Sobolev embedding, we have
\begin{equation}
\begin{split}
\label{eq:lemma-inter}
\norm{\Delta_K\parent{V\Delta_L\widetilde{\Delta}_N\e^{-itH}u_0}}_{L_x^2}&\leq \norm{\e^{-itH}\widetilde{\Delta}_Nu_0}_{L_x^\frac{2d}{d-2}}\norm{(\widetilde{\Delta}_N\Delta_L)V\Delta_Kg}_{L_x^d} \\
&\leq \norm{\e^{-itH}\widetilde{\Delta}_Nu_0}_{L_x^\frac{2d}{d-2}}\norm{(\widetilde{\Delta}_N\Delta_L)V\Delta_Kg}_{\mathcal{H}_x^\frac{d-2}{2}}\,.
\end{split}
\end{equation}
Since $\abs{k-n}\geq5$ and $\abs{l-k}\leq2$, we have $\abs{n-l}\geq3$. Therefore, we can apply~\eqref{eq-negligible}, and get
\[
\norm{(\widetilde{\Delta}_N\Delta_L)V\Delta_Kg}_{\mathcal{H}^\frac{d-2}{2}}\lesssim_\alpha K^{-\alpha}\norm{V\Delta_Kg}_{L_x^2}\lesssim_\alpha K^{-\alpha}\norm{g}_{L_x^2}\,.
\]
Hence, we conclude from~\eqref{eq:lemma-inter} that 
\[
\norm{\Delta_K\parent{V\Delta_L\widetilde{\Delta}_N\e^{-itH}\widetilde{\Delta}_Nu_0}}_{L_x^2}\lesssim_\alpha K^{-\alpha}\norm{g}_{L_x^2}\norm{\e^{-itH}\widetilde{\Delta}_Nu_0}_{L_x^\frac{2d}{d-2}}\lesssim_\alpha K^{-\alpha}\norm{\e^{-itH}\widetilde{\Delta}_Nu_0}_{L_x^\frac{2d}{d-2}}\,,
\]
and the endpoint Strichartz estimate~\eqref{eq-global-str} for $\e^{-itH}$ yields 
\[
\int_\R\sum_{\abs{k-l}\leq2}\norm{\Delta_K\parent{V\Delta_L\widetilde{\Delta}_N\e^{-itH}\widetilde{\Delta}_Nu_0}}_{L_x^2}^2
\lesssim_\alpha K^{-\alpha}\int_\R\norm{\e^{-itH}\widetilde{\Delta}_Nu_0}_{L_x^\frac{2d}{d-2}}^2dt\lesssim_\alpha K^{-\alpha}\norm{\widetilde{\Delta}_Nu_0}_{L_x^2}^2\,.
\]
This proves the estimate for term $B$, and finishes the proof of Lemma~\ref{lemma-ck2}.
\end{proof}
Let us fix an interval $I\subseteq \R$. Without loss of generality, we assume that $0\in I$ and, given a dyadic integer $K$, we write Duhamel's formula as follows :
\begin{equation}
\begin{split}
\label{eq-duhamel-withney}
\Delta_Ku(t) &= \e^{it\Delta}\Delta_K\widetilde\Delta_Nu_0 -i\int_0^t\e^{i(t-\tau)\Delta}\Delta_K(Vu(\tau))d\tau\,,\\
\Delta_{\sim M}v(t) &= \e^{it\Delta}\Delta_{\sim M}\widetilde\Delta_Mv_0 -i\int_0^t\e^{i(t-\tau)\Delta}\Delta_{\sim M}(Vv(\tau))d\tau\,.
\end{split}
\end{equation}
Then, we write
\[
F(t) = \int_0^t\e^{-i\tau\Delta}\Delta_K(Vu(\tau))d\tau,\quad G(t) =\int_0^t\e^{-i\tau\Delta}\Delta_{\sim M}(Vv(\tau))d\tau\,.
\]
Using decomposition~\eqref{eq-duhamel-withney} and developing the product, we obtain
\begin{equation}
\label{eq-I1-4}
\begin{split}
\norm{\Delta_Ku\Delta_{\sim M}v}_{L_{t,x}^2(I)}\leq &\norm{\e^{it\Delta}\parent{\Delta_K\widetilde\Delta_Nu_0}\e^{it\Delta}\parent{\Delta_{\sim M}\widetilde\Delta_Mv_0}}_{L_{t,x}^2(I)} \\
&+\norm{\e^{it\Delta}F(t)\e^{it\Delta}\parent{\Delta_{\sim M}\widetilde\Delta_Mv_0}}_{L_{t,x}^2(I)}\\
&+\norm{\e^{it\Delta}\parent{\Delta_K\widetilde\Delta_Nu_0}\e^{it\Delta}G(t)}_{L_{t,x}^2(I)}\\
&+ \norm{\e^{it\Delta}F(t)\e^{it\Delta}G(t)}_{L_{t,x}^2(I)}\\
&=I_1+I_2+I_3+I_4\,.
\end{split}
\end{equation}
\begin{proof}[Proof of Proposition~\ref{prop-bili-V}]
It remains to estimate each term in the above decomposition. 
\paragraph{\textbf{\textit{Estimate of term} $I_1$ :}} We can directly use the bilinear estimate~\eqref{eq-bili} for the free evolution to estimate $I_1$, where we have the interaction of two free Schrödinger evolutions of frequency-localized data. For any interval $I$
\begin{multline*}
\norm{\e^{it\Delta}\parent{\Delta_K\widetilde\Delta_Nu_0}\e^{it\Delta}\parent{\Delta_{\sim M}\widetilde\Delta_Mv_0}}_{L_{t,x}^2(I)}\\
\leq \norm{\e^{it\Delta}\parent{\Delta_K\widetilde\Delta_Nu_0}\e^{it\Delta}\parent{\Delta_{\sim M}\widetilde\Delta_Mv_0}}_{L_{t,x}^2(\R)}\\
\lesssim N^\frac{d-1}{2}M^{-\frac{1}{2}}\norm{\Delta_K\widetilde{\Delta}_{N}u_0}_{L_x^2}\norm{\widetilde{\Delta}_Mv_0}_{L_x^2}\\
\lesssim C(K,N)N^\frac{d-1}{2}M^{-\frac{1}{2}}\norm{\widetilde{\Delta}_Nu_0}_{L_x^2}\norm{\widetilde{\Delta}_Mv_0}_{L_x^2} \,,
\end{multline*}
where we have from~\eqref{eq-negligible} that for any $\alpha$, there exists $C_\alpha$ such that for all $N,K$ 
\begin{equation}
\label{eq-ckn}
C(K,N)=
\begin{cases}
1\quad &\text{if}\  K\lesssim N\,,\\
C_\alpha K^{-\alpha}\quad &\text{if}\  K\geq2^5N\,.
\end{cases}
\end{equation}
Moreover, we can use sub-additivity to deduce that for any partition $\bigcup_{\ell}I_\ell=\R$,
\[
\parent{\sum_\ell\norm{\e^{it\Delta}\parent{\Delta_K\widetilde\Delta_Nu_0}\e^{it\Delta}\parent{\Delta_{\sim M}\widetilde\Delta_Mv_0}}_{L_{t,x}^2(I_l)}^2}^{1/2}\lesssim C(K,N)\norm{\widetilde{\Delta}_Nu_0}_{L_x^2}\norm{\widetilde{\Delta}_Mv_0}_{L_x^2}
\]
as well. To estimate the others terms in~\eqref{eq-I1-4}, we shall use Lemma~\ref{lemma-loc-weight} and exploit the local smoothing effect to get some global in time estimates. For now on we fix an interval $I$ of size $\abs{I}\sim1$, and see how to handle terms $I_2,I_3$ and $I_4$.
\paragraph{\textbf{\textit{Estimate of term} $I_2$ :}} We use Minkowski and the bilinear estimate~\eqref{eq-bili} for the free Schrödinger equation to get 
\begin{equation*}
\begin{split}
\norm{\e^{it\Delta}F(t)\e^{it\Delta}\parent{\Delta_{\sim M}\widetilde\Delta_Mv_0}}_{L_{t,x}^2(I)}&=\norm{\e^{it\Delta}\parent{\int_0^t\e^{-i\tau\Delta}\Delta_K(Vu(\tau))d\tau}\e^{it\Delta}\parent{\Delta_{\sim M}\widetilde\Delta_Mv_0}}_{L_{t,x}^2(I)}\\
&\leq\int_I\norm{\e^{it\Delta}\parent{\e^{-i\tau\Delta}\Delta_K(Vu(\tau))}\e^{it\Delta}\parent{\Delta_{\sim M}\widetilde\Delta_Mv_0}}_{L_{t,x}^2(I)}d\tau\\
&\lesssim K^\frac{d-1}{2}M^{-\frac{1}{2}}\int_I\norm{\e^{-i\tau\Delta}\Delta_K(Vu(\tau))}_{L_x^2}\norm{\widetilde\Delta_Mv_0}_{L_x^2}d\tau\\
&\lesssim K^\frac{d-1}{2}M^{-\frac{1}{2}}\parent{\int_I\norm{\Delta_K(V\widetilde\Delta_N\e^{-i\tau H}u_0)}_{L_x^2}^2d\tau}^{1/2}\norm{\widetilde\Delta_Mv_0}_{L_x^2}\,,
\end{split}
\end{equation*}
where we used the assumption that $\abs{I}\sim1$. If $K\geq 2^5N$, we apply Lemma~\ref{lemma-ck2} to have 
\begin{equation*}
\begin{split}
K^\frac{d-1}{2}M^{-\frac{1}{2}}\parent{\int_I\norm{\Delta_K(V\widetilde\Delta_N\e^{-i\tau H}u_0)}_{L_x^2}^2d\tau}^{1/2}&\lesssim K^{\frac{d-1}{2}-\alpha}M^{-\frac{1}{2}}\norm{\widetilde\Delta_Nu_0}_{L_x^2}\norm{\widetilde\Delta_Mv_0}_{L_x^2}\\
&\lesssim K^{-1}M^{-\frac{1}{2}}\norm{\widetilde\Delta_Nu_0}_{L_x^2}\norm{\widetilde\Delta_Mv_0}_{L_x^2}\,,
\end{split}
\end{equation*}
by choosing $\alpha=\frac{d-1}{2}-1$. Otherwise, $K\leq 2^5$ and we use Minkowski,  we apply the bilinear estimate~\eqref{eq-bili} and the local smoothing estimate~\eqref{eq-local-smoothing} to get 
\begin{equation*}
\begin{split}
\norm{\e^{it\Delta}F(t)\e^{it\Delta}\parent{\Delta_{\sim M}\widetilde\Delta_Mv_0}}_{L_{t,x}^2(I)}&\lesssim \int_I\norm{\e^{it\Delta}\parent{\e^{-i\tau\Delta}\Delta_{\lesssim N}(Vu(\tau))}\e^{it\Delta}\parent{\Delta_{\sim M}\widetilde\Delta_Mv_0}}_{L_{t,x}^2(I)}d\tau \\
&\lesssim N^\frac{d-1}{2}M^{-\frac{1}{2}}\parent{\int_I\norm{V\e^{-i\tau H}\widetilde{\Delta}_Nu_0}_{L_x^2}^2d\tau}^{1/2}\norm{\widetilde\Delta_Mv_0}_{L_x^2}\\
&\lesssim N^\frac{d-1}{2}M^{-\frac{1}{2}}\norm{\widetilde\Delta_Nu_0}_{L_x^2}\norm{\widetilde\Delta_Mv_0}_{L_x^2}\,.
\end{split}
\end{equation*}
\paragraph{\textbf{\textit{Estimate of term} $I_3$ :}} Similarly, we apply the bilinear estimate~\eqref{eq-bili} and Minkowski to get
\begin{equation*}
\begin{split}
\norm{\e^{it\Delta}\parent{\Delta_K\widetilde\Delta_Nu_0}\e^{it\Delta}G(t)}_{L_{t,x}^2(I)} &=\norm{\e^{it\Delta}\parent{\Delta_K\widetilde\Delta_Nu_0}\e^{it\Delta}\parent{\int_0^t\e^{-i\tau\Delta}\Delta_{\sim M}(Vv(\tau))d\tau}}_{L_{t,x}^2(I)} \\
&\lesssim K^\frac{d-1}{2}M^{-\frac{1}{2}}\norm{\Delta_K\widetilde\Delta_Nu_0}_{L_x^2}\int_I\norm{\Delta_{\sim M}(V\e^{-i\tau H}\widetilde\Delta_M v_0)}_{L_x^2}d\tau\,.
\end{split}
\end{equation*}
It follows from~\eqref{eq-negligible} that
\[
K^\frac{d-1}{2}\norm{\Delta_K\widetilde\Delta_Nu_0}_{L_x^2}\lesssim K^{-1}\norm{\widetilde\Delta_Nu_0}_{L_x^2}\,.
\]
Moreover, we deduce from the fact that $\abs{I}\sim1$ and from the local smoothing estimate~\eqref{eq-local-smoothing} that  
\[
\int_I\norm{\Delta_{\sim M}(V\e^{-i\tau H}\widetilde\Delta_M v_0)}_{L_x^2}d\tau
\lesssim \parent{\int_I\norm{V\e^{-i\tau H}\widetilde\Delta_M v_0}_{L_x^2}^2d\tau}^{1/2}
\lesssim\norm{\widetilde\Delta_Mv_0}_{L_x^2}\,.
\]
\paragraph{\textbf{ \textit{Estimate of term} $I_4$ :}} We use Minkowski and~\eqref{eq-bili} to get 
\begin{equation*}
\begin{split}
&\norm{\e^{it\Delta}F(t)\e^{it\Delta}G(t)}_{L_{t,x}^2(I)}
= \norm{\e^{it\Delta}\parent{\int_0^t\e^{-i\tau\Delta}\Delta_K(Vu(\tau))d\tau}\e^{it\Delta}\parent{\int_0^t\e^{-i\tau\Delta}\Delta_{\sim M}(Vv(\tau))d\tau}}_{L^2(I)}\\
&\leq \iint_{I\times I}\norm{\e^{it\Delta}\parent{\e^{-i\tau_1\Delta}\Delta_K(Vu(\tau_1))}\e^{it\Delta}\parent{\e^{-i\tau_2\Delta}\Delta_{\sim M}(Vv(\tau_2))}}_{L_{t,x}^2}d\tau_1d\tau_2\\
&\lesssim K^\frac{d-1}{2}M^{-\frac{1}{2}}\parent{\int_I\norm{\Delta_K(V\widetilde\Delta_N\e^{-i\tau H}u_0}_{L_x^2}d\tau}\parent{\int_I\norm{\Delta_{\sim M}(V\widetilde\Delta_M\e^{-i\tau H}v_0)}_{L_x^2}d\tau}\,.
\end{split}
\end{equation*}
Then, we conclude as in case $I_2$ and $I_3$, by using Lemma~\ref{lemma-loc-weight} and the local smoothing estimate~\eqref{eq-local-smoothing}. We complete the proof of Proposition~\ref{prop-bilinear} in the case when $\abs{I}\sim1$ by summing over $K$. To prove the global estimate when $I=\R$, we consider a partition $\bigcup_\ell I_\ell$ made of intervals of size $\abs{I_\ell}\sim1$. Then, we exploit the  sub-additivity property of the $L^2$ norm, that is
\[
\norm{f}_{L_t^2(\R)}\sim\parent{\sum_{\ell }\norm{f}_{L_t^2(I_\ell)}^2}^{1/2}\,
\]
for a given function $f\in L^2(\R)$. Moreover, when estimating terms $I_2,I_3$ and $I_4$, we use that the estimate~\eqref{eq-lemma-ck2} from Lemma~\ref{lemma-ck2} is global in time. This is a consequence of the local smoothing effect.
\end{proof}
\section{Probabilistic scattering on the continuous spectral subspace}
\label{section-continuous}
The probabilistic nonlinear a priori estimates and other partial results, such as the critical-weighted strategy presented in this section will be reused in the next when proving Theorem~\ref{theorem-soliton}. Here, we chose to state and to prove Theorem~\ref{theorem-continuous} from a pedagogical perspective before addressing the proof of the main theorem. We recall that we search for a solution of~\eqref{eq-nls-continuous}, which is~\eqref{eq-nls} projected on the continuous spectral subspace, under the form $\psi = \epsilon u^\omega+v$ where $u^\omega=\e^{-itH}u_0^\omega$ and where $v$, that lies in a critical space embedded into $L^\infty(\R;\mathcal{H}^{s_c}(\R^d))$, is solution to the cubic NLS equation with a stochastic forcing term
\begin{equation}
\label{eq-nls-continuous-v}
\begin{cases}
i\partial_tv - Hv=\Pc \mathcal{N}(\epsilon u^\omega+v)\,,\\
v_{\mid t=0}=\psi_0\in\mathcal{H}^{s_c}(\R^d)\,.
\end{cases}
\end{equation}
For the sake of completeness, we first recall definitions and essential properties of \textit{critical spaces} introduced by~\cite{hadac2009,herr2011}. See also~(\cite{koch2013}, p. 49-67) for an expository presentation. These spaces of functions from an interval $I\subseteq\R$ to a Hilbert space $\mathcal{H}$ are constructed upon $V^q, U^p$ spaces. Roughly speaking, they can be seen as the extensions of~\textit{Bourgain spaces} $X^{s,b}$ for $b=1/2$, they embed into $L^\infty(\R,\mathcal{H}^s)$ and are well-behaved with respect to sharp cutoff functions in time. Furthermore, they are well suited for global-in-time a priori estimates thanks to the \textit{duality argument} detailed in section~\ref{sub-duality}, while Bourgain spaces are rather used for local-in-time estimates.
\subsection{Critical spaces}
\label{section-critical}
For now on, the Hilbert space $\mathcal{H}$ is $L^2(\R^d)$ unless otherwise specified. We fix a real number $1\leq q<+\infty$ and an interval $I=\intervaloo{a}{b}$ with $-\infty\leq a<b\leq+\infty$ and we denote the collection of finite partitions of $I$ by $\mathcal{Z}$ :
\begin{equation*}\mathcal{Z}\coloneqq \set{\{t_k\}_{k = 0\dots K}\mid a = t_0 < t_1 <\dots< t_K =b}\,.
\end{equation*}
\subsubsection{Definitions and embeddings}
\begin{definition}[Functions of bounded $q$-variation]
\label{def-Vp}
$V^q(I)$ is the set of functions $v:I\to \mathcal{H}$ endowed with the norm
\begin{equation*}
 \norm{v}_{V^q(I)} \coloneqq \underset{\{t_k\}_{k=0}^{K-1}\in\mathcal{Z}}{\sup} \parent{\sum_{k=1}^K\norm{v(t_{k})-v(t_{k-1})}_{\mathcal{H}}^q}^{1/q}.
\end{equation*}
\end{definition}
Functions in $V^q$ have one-sided limits everywhere, and they may have at most countably many discontinuities. In what follows, we consider the closed subspace $V_{rc}^q$ made of right-continuous functions $v$ in $V^q$ with $\underset{t\to-\infty}{\lim}v(t)=0$. We still write them $V^q=V_{rc}^q$. Next we introduce the predual space of $V^q$ (in a sense detailed in paragraph~\ref{sub-duality}), namely the atomic space $U^p$ with $\frac{1}{p}+\frac{1}{q}=1$.
\begin{definition}[Atomic space $U^p$]
\label{def-atom}
A function $a\,:I\to \mathcal{H}$ is a $p$-atom if there exists a partition $\{t_k\}_{k=0\dots K}$ in $\mathcal{Z}$ and $\{\phi_k\}_{k=0,\dots, K-1}$ some elements in $\mathcal{H}$ such that
\begin{equation*}
a (t) = \sum_{k=1}^K\mathbf{1}_{\intervalco{t_{k-1}}{t_k}}(t)\phi_{k-1}\,,\quad \sum_{k=0}^{K-1} \norm{\phi_k}_{\mathcal{H}}^p \leq 1\,.
\end{equation*}
The atomic space $U^p(I)$ is the set of functions $u\,:I\to \mathcal{H}$ endowed with the norm
\begin{equation}
\label{def-atomic-space} 
\norm{u}_{U^p(I)}\coloneqq \inf\set{\norm{(\lambda_j)}_{\ell^1}\mid u = \sum_{j\geq1}\lambda_j a_j\quad \text{for some $U^p$-atoms $(a_j)$}}.
\end{equation}
\end{definition}
Functions in $U^p$ are right-continuous, they admit left limits everywhere, and they may have at most countably many discontinuities.
As we shall see in the next paragraph, $U^p$ is the predual space of $V^q$ when $\frac{1}{p}+\frac{1}{q}=1$.
\begin{proposition}[Embeddings,~\cite{hadac2009} Proposition 2.2 and Corollary 2.6]
\label{prop-embedding}
\item For $1\leq p < q <\infty$,
\begin{equation}
\label{eq-embed}
U^p\hookrightarrow V^p\hookrightarrow U^q\hookrightarrow L_t^\infty(I,\mathcal{H})\,.
\end{equation}
\end{proposition}
Let us take from~\cite{benyi2015} the continuity property of these norms that is particularly crucial when one wants to use some bootstrap argument. 
\begin{lemma}[Time continuity, see Lemma A.6 in~\cite{benyi2015}]
\label{lemma:U2-continuity}
Let $J=\intervalco{a}{b}$, $u\in U^p(J)\cap C(J;H_x^s(\R^3))$ and $v\in V^q(J)\cap C(J;H_x^s(\R^3))$. The following mappings are continuous
\[
t\in J\to \norm{u}_{U^p\intervalco{a}{t}},\ t\in J\to \norm{v}_{V^q\intervalco{a}{t}}\,.
\]
\end{lemma}
\subsubsection{Duality argument}
\label{sub-duality}
There exists a unique bilinear map $\operatorname{B} : U^p\times V^q\to\C$ such that 
\begin{equation*}\operatorname{B}(u,v)=\sum_{i=1}^K\prodscal{u(t_i)-u(t_{i-1})}{v(t_i)}
\end{equation*}
when $u$ is a right-continuous step function with associated partition $\{t_k\}_{k=0}^K$ such that $u(t_0)\coloneqq u(a)=0$ and when $v$ is a function in $V^q$. We have
\begin{equation*}
\abs{\operatorname{B}(u,v)}\lesssim\norm{u}_{U^q}\norm{v}_{V^p}.
\end{equation*}
Moreover,
\begin{equation*}
v\in V^q \mapsto \operatorname{B}(\cdot,v)\in(U^p)^*
\end{equation*}
is a surjective isometry, and
\begin{equation}
\label{eq-B}
\norm{v}_{V^q}=\underset{u\in U^p,\norm{u}_{U^p}\leq1}{\sup} \abs{\operatorname{B}(u,v)},\quad \norm{u}_{U^p}=\underset{u\in V^q,\norm{u}_{V^q}\leq1}{\sup} \abs{\operatorname{B}(u,v)}.
\end{equation}
The second estimate follows from Hahn-Banach theorem. Furthermore, if $\partial_tu\in L^1(I)$ we have the explicit formula
\begin{equation}
\label{eq-B-expl}
\operatorname{B}(u,v)=\int_I (\partial_t u\mid v)dt\,.
\end{equation}
We refer to~\cite{hadac2009} and~\cite{herr2011} for proofs and details.  
\begin{definition}[Function spaces adapted to linear propagators]
\begin{enumerate}[(i)]
\item We define by
\begin{equation*}
U_H^p(I)\coloneqq\e^{-itH}\parent{U^p(I)\cap L^\infty(I;\ran\Pc)}\,,\quad V_H^q(I)\coloneqq\e^{-itH}\parent{V^q(I)\cap L^\infty(I;\ran\Pc)}
\end{equation*}
the critical spaces adapted to the perturbed linear propagator. They are Banach spaces when endowed with norms 
\begin{equation*}
\norm{u}_{U_H^p}\coloneqq\norm{\e^{itH}u}_{U^p}\,,\quad \norm{v}_{V_H^q}\coloneqq\norm{\e^{itH}v}_{V^q}\,.
\end{equation*}
We have similar definitions for the space $U_\Delta^p$ and $V_\Delta^q$ adapted to the free evolution $\e^{it\Delta}$. 
\item The spaces $DU^p$ and $DU_H^p$ are defined by
\[
DU^p=\set{f=\partial_tu\mid u\in U^p\,,\ f(a^+)=0}\,,\
DU_H^p=\e^{-itH}\set{f=\e^{-itH}\partial_tu\mid u\in U^p,\ f(a^+)=0}\,.
\]
These spaces are endowed with norms 
\begin{equation*}
\norm{f}_{DU^p}=\norm{u}_{U^p}\,,\quad \norm{f}_{DU_H^p}=\norm{\e^{itH}u}_{U^p}\,.
\end{equation*}
Spaces $DV^q$ and $DV_H^q$ are defined in the same way.
\end{enumerate}
\end{definition}
\begin{remark}
Here, $\partial_t$ has to be understood in the sense of distributions on $I$. Let us recall that for distributions in one dimension, $\partial_tu=0$ implies that $u$ is constant, and thus we note that there is no ambiguity in the above definition, since we imposed the condition $\underset{t\to a}{\lim}f=0$.
\end{remark}
The definition of $DU_H^p$ and $DV_H^q$ are motivated by the need to control the Duhamel integral representation of a solution $u$ to a forced Schrödinger equation. Namely, given $f$ in $L_t^1(I,\mathcal{H})$ and $u$ such that 
\begin{equation*}
\begin{cases}
i\partial_tu-Hu=f\,,\\
u\lvert_{t=0}=u_0\,,
\end{cases}
\end{equation*}
the Duhamel's formulation reads
\begin{equation*}
u(t)=\e^{-itH}u_0-i\int_0^t\e^{-i(t-\tau)H}f(\tau)d\tau\,.
\end{equation*}
Consequently, we get
\begin{equation*}
\norm{u}_{U_H^p}\leq\norm{u_0}_{\mathcal{H}}+\norm{f}_{DU_H^p}=\norm{u_0}_{\mathcal{H}}+\norm{\int_0^t\e^{i\tau H}f(\tau)d\tau}_{U^p}\,.
\end{equation*}
Furthermore, it follows from the duality argument~\eqref{eq-B} that
\begin{equation*}
\norm{f}_{DU_H^p} =\norm{\e^{-itH}\int_0^te^{i\tau H}f(\tau)d\tau}_{U_H^2} = \underset{\norm{v}_{V^q}\leq1}{\sup}\bigl\lvert \operatorname{B}\bigl(\int_0^t\e^{i\tau H}f(\tau)d\tau, v\bigr)\bigr\rvert\,.
\end{equation*}
Since $f$ is in $L_t^1(I,\mathcal{H})$ we apply the explicit formula for $\operatorname{B}$~\eqref{eq-B-expl} to get that
\begin{equation}
\label{eq-duality}
\norm{f}_{DU_H^p} = \underset{\norm{v}_{V_H^q}\leq1}{\sup}\bigl\lvert\int_I(f\mid v)d\tau\bigr\rvert\,.
\end{equation}
The corresponding formula for $DV_H^q$ can also be deduced from~\eqref{eq-B} : 
\begin{equation}
\label{eq-duality-DV}
\norm{f}_{DV_H^q} = \underset{\norm{u}_{U_H^p}\leq1}{\sup}\bigl\lvert\int_I(f\mid u)d\tau\bigr\rvert\,.
\end{equation} 
\subsubsection{Transferred linear and bilinear estimates}
\begin{proposition}
\label{prop-transfer}
For $q>2$ and $r$ with $\displaystyle\frac{2}{q}+\frac{d}{r}=\frac{d}{2}$ and $2(d+2)\leq dp$
we have
\begin{enumerate}[(i)]
\item Linear Strichartz estimates :
\begin{align}
\label{eq-str-transferred}
&\norm{v}_{L_t^q(\R ; L_x^r(\R^d))}\lesssim\norm{v}_{V_H^2}\,,\\
\label{eq-str-bernstein-transferred}
&\norm{v}_{L_{t,x}^p(\R\times\R^d)}\lesssim\norm{\langle\sqrt H\rangle^{\frac{d}{2}-\frac{d+2}{p}}v}_{V_H^2}\,.
\end{align}
\item Bilinear Strichartz estimate : Let $N,M$ be two dyadic numbers,
\begin{equation}
\label{eq-bili-transferred}
\norm{\widetilde\Delta_Nu\widetilde\Delta_Mv}_{L_{t,x}^2(\R\times\R^d)}\lesssim N^\frac{d-2}{2}\parent{\frac{N}{M}}^{\frac{1}{2}-}\norm{\widetilde\Delta_Nu}_{V_H^2}\norm{\widetilde\Delta_Mv}_{V_H^2}\,.
\end{equation}
\end{enumerate}
\end{proposition}
\begin{proof} To prove the embedding~\eqref{eq-str-transferred}, we first deduce from the global-in-time Strichartz estimate~\eqref{eq-global-str} that for any Schrödinger admissible pair $(p,q)$ the mixed Lebesgue space $L_t^qL_x^r$ embeds into $U^q$, which embeds into $V^2$ provided that $2<q$ (see Proposition~\ref{prop-embedding}). Estimate~\eqref{eq-str-bernstein-transferred} follows from~\eqref{eq-str-transferred} and from the Bernstein estimate~\eqref{eq-bernstein-v}. We refer to Lemma 3.3~\cite{benyi2015} and the references therein to see how to transfer the bilinear estimate~\eqref{eq-bili-v} into the space $V_H^2$.
\end{proof}
In order to derive probabilistic nonlinear estimates, we consider the Duhamel integral representation of the solution to~\eqref{eq-nls-continuous}. Provided that the nonlinear forcing term $\Pc\mathcal{N}(v+\epsilon u^\omega)$ lies in $L^1(\R,\mathcal{H}^\frac{d-2}{2})$, it follows from the \textit{duality argument}~\eqref{eq-duality} that the $U_H^2$-norm of the solution $u$ to~\eqref{eq-nls-continuous-v} can be estimated by
\[\norm{v}_{U_H^2(\mathcal{H}^\frac{d-2}{2})}\leq \norm{\psi_0}_{\mathcal{H}^{1/2}}\underset{\norm{w}_{V_H^2}\leq1}{\sup}\iint_{\R\times\R^d}\japbrak{\sqrt{H}}^\frac{d-2}{2}\parent{\Pc\mathcal{N}(v+\epsilon u^\omega)}wdtdx\,.\]
If we proceed to a Littlewood-Paley decomposition of each term $u_j$ that occurs in the forcing term, we are reduced to estimate sums of multilinear integrals of the form
\begin{equation}
\label{eq-multi}
I_{N_1,N_2,N_3,N_4}=\iint_{\R\times\R^d}\langle\sqrt H\rangle^\frac{d-2}{2}\parent{\widetilde{\Delta}_{N_1}u_1\widetilde{\Delta}_{N_2}u_2\widetilde{\Delta}_{N_3}u_3}\widetilde{\Delta}_{N_4}wdxdt\,,
\end{equation}
where each term $u_j$ can be either $u^\omega,v$ or their complex conjugate for $j\in\set{1,2,3}$, and $w$ is in $V_H^2$ with norm less than one. Although $\Pc\mathcal{N}(v+\epsilon u^\omega)$ does not a priori lies in $L^1(\R,\mathcal{H}^\frac{d-2}{2})$, it follows from Bernstein and Strichartz estimates that we have
\begin{equation}
\label{N-L1}
\widetilde{\Pi}_N \Pc  \mathcal{N}(v+\epsilon u^\omega)\in L^1(I,\mathcal{H}^\frac{d-2}{2})\,,
\end{equation}
for any dyadic number $N$ as observed in~\cite{benyi2015}. It turns out that the estimates on multilinear integrals obtained by the authors are uniform in $N$. Hence, we have to estimate $\widetilde\Pi_Nv$ rather than $v$, and to use a space slightly smaller than $U_H^2$ which keeps track of the frequency cut-off.\begin{definition}
The critical space $X^\frac{d-2}{2}(I)$ is the space made of tempered distributions $u\,:\,I\to \mathcal{H}^\frac{d-2}{2}(\R^d)\cap\ran(\Pc)$ such that $\norm{u}_{X^\frac{d-2}{2}(I)}$ is finite, where 
\begin{equation}
\label{eq-Xcr}
\norm{u}_{X^\frac{d-2}{2}(I)}\coloneqq \parent{\sum_{N\in2^\N}N^{d-2}\norm{\widetilde\Delta_Nu}_{U_H^2(I)}^2}^{1/2}\,.
\end{equation}
We define as well $DX^\frac{d-2}{2}(I)$ by replacing in~\eqref{eq-X} the $U_H^2(I)$-norm by the $DU_H^2(I)$ one.
\end{definition}
\begin{remark} In the case of $V_H^2$ it follows from Definition~\ref{def-Vp} and from the almost orthogonality property of Littlewood-Paley decomposition that 
\begin{equation}
\label{eq-Veq}
\norm{\langle\sqrt H\rangle^\frac{d-2}{2}v}_{V_H^2(I)}\sim\parent{\sum_{N\in2^\N}N^{d-2}\norm{\widetilde\Delta_Nv}_{V_H^2(I)}^2}^{1/2}\,.
\end{equation}
\end{remark}
\subsection{Probabilistic nonlinear estimates}
\label{section-nonlinear}
Thanks to the probabilistic and bilinear improved Strichartz estimates for the perturbed operator obtained in section~\ref{sec-imp-str} and their transferred versions collected in Proposition~\ref{prop-transfer} we are now able to reproduce in the perturbed framework the same analysis conducted by~\cite{benyi2015}. We get the following probabilistic nonlinear estimates.
\begin{proposition}
\label{prop-a-priori}
Let $0<\epsilon$, $0<R$ and $I=(a,b)$ with $-\infty\leq a<b\leq+\infty$. For all $v\in X^\frac{d-2}{2}(I)$ and $\omega\in\widetilde{\Omega}_{\epsilon,R}$ (see~\eqref{set-eps-R}) we have
\begin{equation}
\label{eq-a-priori}
\underset{2^\N}{\sup}\ N^\frac{d-2}{2}\norm{\widetilde{\Pi}_N \parent{\mathcal{N}(v+\epsilon u^\omega)}}_{DU_H^2(I)}\lesssim \norm{\psi_0}_{\mathcal{H}^{1/2}}+\norm{v}_{V_H^2(I)}^3+R^3\,,
\end{equation}
as well as the Lipschitz estimate : for all $v,v'\in V_H^2(I)$,
\[
\underset{2^\N}{\sup}\ N^\frac{d-2}{2}\norm{\widetilde{\Pi}_N \parent{\mathcal{N}(v'+\epsilon u^\omega)-\mathcal{N}(v+\epsilon u^\omega)}}_{DU_H^2(I)} 
\lesssim \parent{\norm{v}_{V_H^2(I)}^2+\norm{v'}_{V_H^2(I)}^2+R^2}\norm{v'-v}_{V_H^2(I)}\,.
\]
\end{proposition}
\subsubsection{Contribution of terms like $v^2u^\omega$ with high-low interactions :} We illustrate the case by case analysis performed in~\cite{benyi2015} for the flat Laplacian, and we transpose it to the perturbed setting. In order to be as brief as possible, we chose to detail one enlightening case. We consider the high-high-low regime, where 
\[v_1=\widetilde\Delta_{N_1}v,\ v_2=\widetilde\Delta_{N_2}v\,,\ u_3^\omega=\widetilde\Delta_{N_3}u^\omega,\ w_4=\widetilde\Delta_{N_4}w,\ \text{and}\ N_1\leq N_3^\frac{1}{d-1}\leq N_2\leq N_3,N_4\sim N\,.~\footnote{For convenience, we define $\Lambda_N=\set{(N_1,N_2,N_3,N_4)\mid\ N_1\leq N_3^\frac{1}{d-1}\leq N_2\leq N_3,N_4\sim N}$\,.}\]
We refer to~(\cite{benyi2015}, Proposition 4.1) for the other subcases. As we explained in a previous paragraph, the proof of Proposition~\ref{prop-a-priori} consists in getting some a priori estimates on multilinear integrals of the form~\eqref{eq-multi}. Under the assumption made on the frequencies, the main term is the one where the derivatives fall onto $\widetilde\Delta_{N_3}v_3$. Hence, it follows from Hölder inequality that~\footnote{\
For a dyadic number $N$ we shall denote small positive (resp. negative) power of $N$ by $N^{+0}$ (resp. $N^{-0}$).}
\begin{multline*}
\abs{\iint_{\R\times\R^d}\langle\sqrt H\rangle^\frac{d-2}{2}\parent{v_1v_2\epsilon u_3^\omega}\widetilde{\Delta}_{N_4}wdxdt}\lesssim N_3^\frac{d-2}{2}\norm{v_2}_{L_{t,x}^\frac{2(d+2)}{d}}\norm{\epsilon u_3^\omega}_{L_{t,x}^{d+2}}\norm{v_1w_4}_{L_{t,x}^2}\\
\lesssim N_3^\frac{d-2}{2}N_1^{\frac{d-1}{2}-}N_4^{-\frac{1}{2}+}\norm{v_2}_{V_H^2}\norm{\epsilon u_3^\omega}_{L_{t,x}^{d+2}}\norm{v_1}_{V_H^2}\norm{w_4}_{V_H^2}\\
\lesssim C(N_1,N_2,N_3,N_4)\parent{N_2^\frac{d-2}{2}\norm{v_2}_{V_H^2}}\parent{N_3^s\norm{\epsilon u_3^\omega}_{L_{t,x}^{d+2}}}\parent{N_1^\frac{d-2}{2}\norm{v_1}_{V_H^2}}\norm{w_4}_{V_H^2}\,,
\end{multline*}
with
\[
C(N_1,N_2,N_3,N_4) = N_3^{\frac{d-2}{2}-s}N_1^{\frac{d-1}{2}-0-\frac{d-2}{2}}N_2^{-\frac{d-2}{2}}N_4^{-\frac{1}{2}+0} \lesssim N_3^{\frac{d-2}{2}-s+\frac{1}{2(d-1)}-\frac{d-2}{2(d-1)}-\frac{1}{2}+0}\eqqcolon N_3^{\delta-s}\,.\]
We used the transferred Strichartz estimate~\eqref{eq-str-transferred} to control the term $v_1$ and the transferred bilinear estimate~\eqref{eq-bili-transferred} to control the term with $v_2w_4$. Note that under the assumption that $\frac{d-1}{d+1}\cdot\frac{d-2}{2}<s$, we have 
\[
\quad \delta=\frac{d-3}{d-1}\cdot\frac{d-2}{2}+0 < \frac{d-1}{d+1}\cdot \frac{d-2}{2} <s\,.
\]
Therefore, since the highest frequency comes with a negative power we are able to apply \textit{Schur's test} and to sum over $N_1,N_2,N_4$. To sum over the dyadic numbers $N_3$, we use first Hölder
\begin{multline*}
\sum_{N_3\in2^\N}N_3^{\delta-s}N_3^s\norm{\widetilde\Delta_{N_3}\epsilon u^\omega}_{L_{t,x}^{d+2}}\leq N^{\delta-s+0}\parent{\sum_{N_3\in2^\N}N_3^{-0}}^\frac{d+1}{d+2}\norm{N_3^s\widetilde\Delta_{N_3}\epsilon u^\omega}_{\ell_{N_3}^{d+2}L_{t,x}^{d+2}}\\
\lesssim N^{\delta-s+0} \norm{N_3^s\widetilde\Delta_{N_3}\epsilon u^\omega}_{L_t^{d+2}\ell_{N_3}^{d+2}L_x^{d+2}} \lesssim N^{\delta-s+0}\norm{N_3^s\widetilde\Delta_{N_3}\epsilon u^\omega}_{L_t^{d+2}\ell_{N_3}^2L_x^{d+2}}\,,
\end{multline*}
and we deduce from Littlewood-Paley's inequality~\eqref{eq-LP} and probabilistic Strichartz estimate~\eqref{eq-rand-str} that for $\omega\in\widetilde\Omega_{\epsilon,R}$,
\begin{equation}
\label{eq-sN3}
\sum_{N_3\in2^\N}N_3^{\delta-s}N_3^s\norm{\widetilde\Delta_{N_3}\epsilon u^\omega}_{L_{t,x}^{d+2}}\lesssim N^{\delta-s+0} \norm{\langle\sqrt H\rangle^{\delta+}\epsilon u^\omega}_{L_{t,x}^{d+2}}\lesssim N^{\delta-s+0}R\,.
\end{equation}
Hence,
\[
\sum_{N_3}\iint_{\R\times\R^d}\abs{\langle\sqrt H\rangle^\frac{d-2}{2}\parent{v_1v_2\epsilon u_3^\omega}w_4}dxdt
\lesssim N^{\delta-s+0}R\parent{N_1^\frac{d-2}{2}\norm{v_1}_{V_H^2}}\parent{N_2^\frac{d-2}{2}\norm{v_2}_{V_H^2}}\norm{w_4}_{V_H^4}\,.
\]
It follows from the \textit{Schur's test } (see~\cite{benyi2015}, Lemma 3.6) that
\[
\sum_{\Lambda_N}\iint_{\R\times\R^d}\abs{\langle\sqrt H\rangle^\frac{d-2}{2}\parent{v_1v_2\epsilon u_3^\omega}w_4}dxdt
\lesssim R\norm{N_1^\frac{d-2}{2}v_1}_{\ell^2_{N_1}V_H^2}\norm{N_2^\frac{d-2}{2}v_2}_{\ell^2_{N_2}V_H^2}\norm{w_4}_{\ell^2_{N_4}V_H^2}\,,
\]
and the contribution of the integral~\eqref{eq-multi} in this case is less than
\begin{equation*}
\underset{\norm{w}_{V_H^2\leq1}}{\sup}R\norm{\langle\sqrt H\rangle^\frac{d-2}{2}v}_{V_H^2}^2\norm{w}_{V_H^2}\lesssim R\norm{\langle\sqrt H\rangle^\frac{d-2}{2}v}_{V_H^2}^2\,.
\end{equation*}
\begin{remark}
\label{remark-t}
As observed in~\cite{benyi2015}, there is at least a portion of $\norm{\langle\sqrt H\rangle^s\epsilon u^\omega}_{L_{t,x}^p}$ that bounds from above the multilinear integrals where $u^\omega$ occurs, with $p\in\set{4,d+2,6\frac{d+2}{d+4}}$ (see for instance~\eqref{eq-sN3} where $p=d+2$). Hence, it follows from the probabilistic Strichartz estimate~\eqref{eq-rand-str} and from the monotone convergence theorem that for all $\omega\in\widetilde\Omega$,
\begin{equation*}
\underset{\abs{I}\to0}{\lim}\norm{\epsilon u^\omega}_{L_{t,x}^p(I\times\R^d)}=0\,.
\end{equation*}
In the $vv\bar v$-case, we also have 
\begin{equation}
\label{eq-Ishrink}
\underset{\norm{w}_{V_H^2}\leq1}{\sup}\abs{\int_I\int_{\R^d}\1_I(t)(\langle\sqrt H\rangle^\frac{d-2}{2}v)v^2wdxdt}\leq
\norm{\langle\sqrt H\rangle^\frac{d-2}{2}v}_{L_{t,x}^\frac{2(d+2)}{d}(I\times\R^d)}\norm{v}_{L_{t,x}^{d+2}(I\times\R^d)}^2\,,
\end{equation}
where the right-hand side of~\eqref{eq-Ishrink} is uniformly bounded by the global-in-time critical $X^\frac{d-2}{2}(\R)$-norm. Hence, the left-hand side of~\eqref{eq-Ishrink} also goes to zero when $\abs{I}\to0$ (in order to prove continuity or local well-posedness) or when $I\subset\intervalco{t}{\infty}$ with $t\to+\infty$ (in order to prove scattering). Consequently, we have in case that
\begin{equation}
\label{eq-I0}
\underset{I}{\lim}\ \norm{\mathcal{N}(v+\epsilon u^\omega)}_{DU_H^2(I)}=0\quad \text{whenever}\ \abs{I}\to0\ \text{or}\ I\subset\intervalco{t}{\infty}\ \text{with}\ t\to+\infty\,.
\end{equation}
\end{remark}
\subsection{Proof of Theorem~\ref{theorem-continuous}}
We are now ready to extend the probabilistic global Cauchy theory for~\eqref{eq-nls} to the case where we have a short-range potential, and when the equation is projected onto the continuous spectral subspace.
\paragraph{\textbf{Global existence}} We denote the interval $\intervalco{0}{\infty}$ by $I$ and we introduce the Banach space
\begin{equation*}
B_R=\set{v\in X^\frac{d-2}{2}(I)\cap C(I,\mathcal{H}^\frac{d-2}{2})\mid\norm{v}_{X^\frac{d-2}{2}(0,\infty)}\leq R}
\end{equation*}
for a fixed small $R>0$, as well as the operator
\begin{equation*}
\mathcal{T} : v\in B_R\mapsto \parent{t\mapsto\e^{-itH}\psi_0-i\int_0^t\e^{-i(t-\tau)H}\Pc\parent{\mathcal{N}(v+\epsilon u^\omega)}d\tau}\,.
\end{equation*}
It follows from the a priori estimates collected in Proposition~\ref{prop-a-priori} that for all $v$ in $B_R$, 
\[
\norm{\mathcal{T}v}_{X^\frac{d-2}{2}(0,\infty)}\leq \norm{\e^{-itH}\psi_0}_{X^\frac{d-2}{2}(0,\infty)}+\norm{\Pc\parent{\mathcal{N}(v+\epsilon u^\omega)}}_{DX^\frac{d-2}{2}(0,\infty)}
\leq \norm{\psi_0}_{\mathcal{H}^\frac{d-2}{2}}+2C_1R^3\,,
\]
and that for all $v,v'$ in $B_R$,
\[
\norm{\mathcal{T}(v')-\mathcal{T}(v)}_{X^\frac{d-2}{2}(0,\infty)}\leq \norm{\Pc\parent{\mathcal{N}(v'+\epsilon u^\omega)-\mathcal{N}(v+\epsilon u^\omega)}}_{DX^\frac{d-2}{2}(0,\infty)}\leq 3C_2R^2\norm{v'-v}_{X^\frac{d-2}{2}(0,\infty)}\,.
\]
Hence, $\mathcal{T}$ is a contraction mapping on $B_R$ provided that 
\begin{equation*}
2C_1R\leq 1/2,\ 3C_2R\leq 1/2\quad \text{and}\ \norm{\psi_0}_{\mathcal{H}^\frac{d-2}{2}}\leq R/2\,.
\end{equation*}
With this fixed $R=R_0$, we write $\widetilde\Omega_{\epsilon}=\widetilde\Omega_{\epsilon,R_0}$ the set defined in~\eqref{set-eps-R}. It satisfies
\begin{equation*}
\mathbf{P}(\Omega\setminus\widetilde\Omega_\epsilon)\leq C\exp\parent{-\widetilde c\epsilon^{-2}\norm{u_0}_{\mathcal{H}^s}^{-2}}
\end{equation*}
where $\widetilde c=cR_0^{-2}$ is a universal constant.
\paragraph{\textbf{Continuity and scattering}} : Let $v$ be a solution to~\eqref{eq-nls-continuous} with random initial data associated with an $\omega\in\widetilde\Omega_\epsilon$ . We need to show that the limit of $\mathcal{I}(t)$ exists in $\mathcal{H}^\frac{d-2}{2}$ when $t$ goes to infinity, where
\begin{equation*}
\mathcal{I}(t)\coloneqq \int_0^t\e^{i\tau H}\Pc\parent{\mathcal{N}(v+\epsilon u^\omega)}d\tau\,.
\end{equation*}
We use the Cauchy criterion, and chose $t_1,t_2\in\intervaloo{0}{\infty}$ with, say, $t_1<t_2$. We denote the interval $\intervaloo{t_1}{t_2}$ by $I$. It follows from~\eqref{eq-I0} in remark~\ref{remark-t} that
\begin{equation*}
\norm{\mathcal{I}(t_2)-\mathcal{I}(t_1)}_{\mathcal{H}^\frac{d-2}{2}}\leq\norm{\Pc\parent{\mathcal{N}(v+\epsilon u^\omega)}}_{X^\frac{d-2}{2}(I)}\underset{t_1,t_2\to+\infty}{\longrightarrow} 0\,.
\end{equation*}
This yields scattering with 
\begin{equation*}
v_+\coloneqq\int_0^\infty\e^{i\tau H}\Pc\parent{\mathcal{N}(v+\epsilon u^\omega)}d\tau\in\mathcal{H}^\frac{d-2}{2}(\R^d)\,,
\end{equation*}
and the same argument applied when $\abs{I}\to0$ proves that $\mathcal{T}v$ is in $C(I,\mathcal{H}^\frac{d-2}{2})$.
\paragraph{\textbf{Uniqueness}} : We use a connectedness argument. Let $v_1,v_2$ be two solutions to~\eqref{eq-nls-continuous} in the space $X^\frac{d-2}{2}\intervalco{0}{\infty}$. Define $A=\set{t\in\intervalco{0}{\infty}\mid v_1(t)=v_2(t)}$. By continuity in time, $A$ is closed. Let $t\in A$ and, for small $\tau$, we define $I=\intervaloo{t-\tau}{t+\tau}$ if $t>0$, or $I=\intervalco{0}{\tau}$ otherwise. It follows from remark~\ref{remark-t} that
\begin{equation*}
\norm{v_2-v_1}_{X^\frac{d-2}{2}(I)}=\smallO{\tau}\norm{v_2-v_1}_{X^\frac{d-2}{2}(I)}\,.
\end{equation*}
Hence, $v_2=v_1$ on $I$ if $\tau$ is small enough, which proves that $A$ is open. We conclude that $A=\intervalco{0}{\infty}$ whenever $A$ is nonempty, and this finishes the proof of Theorem~\ref{theorem-continuous}.
\subsection{Local smoothing and the critical weighted strategy }
First, we transfer in $U_H^2$ the local smoothing effect stated in Proposition~\ref{prop-local-smoothing}
\label{section-local-smoothing}
\begin{proposition}[Local smoothing in $U^2$]
The local smoothing estimate reads in $U_{H}^2(\mathcal{H}^{1/2})$ :
\begin{equation}
\label{eq-transferred-local-smoothing}
\int_\R\norm{\Pc(u)}_{\mathcal{H}^{1,-1/2-}}^2dt\lesssim_V\norm{u}_{U_H^2(\mathcal{H}^{1/2})}^2\,.
\end{equation}
\end{proposition}
The proof is straightforward and follows from the transference principle and from Proposition~\ref{prop-local-smoothing} with $F=0$. Note in particular that the result is true when $V=0$ and when the space $U_H^2$ is replaced by $U_\Delta^2$.
\begin{remark}
\label{remark-duality}
Estimate~\eqref{eq-transferred-local-smoothing} means that $U_H^2(\mathcal{H}^{1/2})$ embeds into $\disp L_t^2(\R;\mathcal{H}^{1,-1/2-}(\R^d))$, and we can deduce from this and from the duality argument~\eqref{eq-duality} that
\begin{equation}
\label{eq-dual-embedding}
\disp L_t^2(\R;L_x^{2,1/2+})\hookrightarrow DV_H^2(\mathcal{H}^{1/2})\,.
\end{equation}
Indeed, provided that $f\in L^1(\R,L_x^2)$, we have
\begin{equation*}
\begin{split}
\norm{\japbrak{\sqrt{H}}^{1/2}f}_{DV_H^2} &= \underset{\norm{u}_{U_H^2}\leq1}{\sup}\abs{\int_\R(\japbrak{\sqrt{H}}^{1/2}f\mid u)dt}= \underset{\norm{u}_{U_H^2}\leq1}{\sup}\abs{\int_\R(f\mid \japbrak{\sqrt{H}}^{1/2}u)dt}\\
&\leq \underset{\norm{u}_{U_H^2}\leq1}{\sup}\norm{\japbrak{x}^{1/2+}f}_{L_{t,x}^2}\norm{\japbrak{x}^{-1/2-}\japbrak{\sqrt{H}}\parent{\japbrak{\sqrt{H}}^{-1/2}u}}_{L_{t,x}^2}\\
&\leq \underset{\norm{u}_{U_H^2}\leq1}{\sup}\norm{\japbrak{x}^{1/2+}f}_{L_{t,x}^2}\norm{u}_{U_H^2}\lesssim\norm{f}_{L_t^2(\R;L_x^{2,1/2+})}\,.
\end{split}
\end{equation*}
The general case follows from the density of step functions in $V^p$ (see the proof in~\cite{koch2013}, Lemma 4 p.56). Furthermore, embedding~\eqref{eq-embed} directly yields the other dual embedding  
\begin{equation}
\label{eq-dual}
DU_H^2(\mathcal{H}^{1/2})\hookrightarrow DV_H^2(\mathcal{H}^{1/2})\,.
\end{equation}
\end{remark}
\begin{remark}
In light of the nonlinear analysis performed in section~\ref{section-nonlinear}, we shall use the slightly better estimate
\begin{equation*}
\int_\R\norm{\Pc u}_{\mathcal{H}^{1,-1/2-}}dt\lesssim\underset{2^\N}{\sup}\ N^{1/2}\norm{\widetilde{\Pi}_N u}_{U_H^2}\,.
\end{equation*}
Let us denote the Schwartz class by $\mathcal{S}(\R^d)$ and the space of tempered distributions $\mathcal{S}'(\R^d)$, and deduce the above refined estimate from~\eqref{eq-transferred-local-smoothing} and duality. Given $u\in\mathcal{H}^{1,-1/2-}$, we have
\begin{equation*}
\norm{\Pc u }_{\mathcal{H}^{1,-1/2-}}=\underset{\varphi\in\mathcal{S},\norm{\varphi}_{L^2}\leq1}{\sup}\abs{(\Pc u \mid \japbrak{x}^{-1/2-}\japbrak{\sqrt H}\varphi)}\,.
\end{equation*}
Since $\japbrak{x}^{-1/2-}\langle\sqrt H\rangle\varphi\in \mathcal{S}(R^d)$ and $u=\underset{N\to\infty}{\lim}\ \widetilde{\Pi}_N u$ in $\mathcal{S}'(\R^d)$, we have
\begin{equation*}
\norm{\Pc u }_{\mathcal{H}^{1,-1/2-}}=\underset{\varphi\in\mathcal{S},\norm{\varphi}_{L^2}\leq1}{\sup}\abs{\parent{\underset{N\to\infty}{\lim}\ \widetilde{\Pi}_N  u\mid \japbrak{x}^{-1/2-}\langle\sqrt H\rangle\varphi}}\leq\underset{2^\N}{\sup}\ \norm{\widetilde{\Pi}_N  u }_{\mathcal{H}^{1,-1/2-}}\,.
\end{equation*}
\end{remark}
\subsubsection{Critical-weighted strategy} The next Proposition details the \textit{critical-weighted strategy}. As explained in the introduction, this strategy makes it possible to handle both linear and nonlinear terms that arise in the stability equation around a ground state. We state two versions of this strategy. The first one is suited to the operator $-\Delta$ and has an interest in itself. One can use it when $Vu$ is not absorbed by the linear operator, but when it is seen instead as a source term. The second version is suited to the dynamic around the nonlinear ground states, when $V$ is absorbed by the linear operator so that $\e^{-itH}$ preserves $\ran\Pc$ and the decomposition of the phase space from Lemma~\ref{lemma-proj}. This second version plays a key role in the proof of Theorem~\ref{theorem-soliton}. See also~\cite{koch2012} for a similar approach in the context of Korteweg–de Vries equation. 
\begin{proposition}
\label{proposition-trick}
Let $u$ be a solution to the forced~\footnote{\ Here, $f$ (resp. $g$) has to be thought of as a nonlinear term controlled in the critical space (resp. a linear and localized term controlled in a weighted $L^2$ space).} Schrödinger equation
\begin{equation}
\label{eq-trick}
\begin{cases}
i\partial_tu+(\Delta-V) u &= f+g\,,\\
u\lvert_{t=0}=u_0\,.
\end{cases}
\end{equation}
There exists $C=C(V)$ depending on weighted Sobolev norms of $V$ such that for all $u\in\ran\Pc(H)$,
\begin{equation}
\label{eq-trick-flat}
\norm{u}_{V_\Delta^2(\mathcal{H}^{1/2})}+\norm{u}_{L_t^2(\R;\mathcal{H}^{1,-1/2-})}\leq C\parent{ \norm{u_0}_{\mathcal{H}^{1/2}}+\underset{2^\N}{\sup}\ N^{1/2} \norm{\Pi_N  f}_{DU_\Delta^2}+\norm{g}_{L_t^2(\R;L_x^{2,1/2+})}}\,,
\end{equation}
\begin{equation}
\label{eq-trick-V}
\norm{u}_{V_H^2(\mathcal{H}^{1/2})}+\norm{u}_{L_t^2(\R;\mathcal{H}^{1,-1/2-})}\leq C\parent{ \norm{u_0}_{\mathcal{H}^{1/2}}+\underset{2^\N}{\sup}\ N^{1/2} \norm{\widetilde\Pi_N  f}_{DU_H^2}+\norm{g}_{L_t^2(\R;L_x^{2,1/2+})}}\,.
\end{equation}
\end{proposition}
\begin{remark}
We prove the estimate with $\underset{2^\N}{\sup}\ \norm{\widetilde{\Pi}_N  f}_{DU^2}$ instead of $\norm{f}_{DU^2}$ in order to use~\eqref{eq-a-priori}.
\end{remark}
\begin{proof} Estimate~\eqref{eq-trick-V} is slightly easier to prove since we don't have to control the term $Vu$ which is absorbed by the left-hand side in~\eqref{eq-trick}. Also, the local smoothing estimates we use are the same for $-\Delta$ and $-\Delta+V$ provided that $u$ in $\ran(\Pc)$ so we only write the proof of~\eqref{eq-trick-flat}. To that end, we pass the term $Vu$ on the right-hand side of equation~\eqref{eq-trick}
\begin{equation*}
\begin{cases}
i\partial_t u +\Delta u &=f+g+Vu,\\
u\lvert_{t=0}=u_0.
\end{cases}
\end{equation*}
By definition of the space $DV_\Delta^2$, Duhamel's formulation gives
\begin{equation}
\label{eq-first}
\begin{split}
\norm{u}_{V_\Delta^2(\mathcal{H}^{1/2})}&\lesssim\underset{2^\N}{\sup}\ \norm{\Pi_N  f}_{DV_\Delta^2(\mathcal{H}^{1/2})}+\norm{g}_{DV_\Delta^2(\mathcal{H}^{1/2})}+\norm{Vu}_{DV_\Delta^2(\mathcal{H}^{1/2})}\,,\\
&\lesssim \underset{2^\N}{\sup}\ N^{1/2}\norm{\Pi_N   f}_{DU_\Delta^2}+\norm{g}_{L_t^2(\R;L_x^{2,1/2+})}+\norm{Vu}_{L_t^2(\R;L_x^{2,1/2+})}\,.
\end{split}
\end{equation}
Note that we used the dual embedding~\eqref{eq-dual} to control $\Pi_N f$ in $DU_\Delta^2$ and the local smoothing dual embedding~\eqref{eq-dual-embedding} to control $g$ and $Vu$. As for the weighted norm of $u$, we decompose it into $u=v+w$ where $v,w\in\ran\Pc(H)$ and are solution to 
\begin{equation*}
\begin{cases}
i\partial_t v+\Delta v = f\,,&\quad i\partial_t w+(\Delta-V) w = g-Vv\,,\\
v\lvert_{t=0}=0\,,&\quad  w\lvert_{t=0}=u_0\,.
\end{cases}
\end{equation*}
To control $u$ in $\displaystyle L_t^2(\R;\mathcal{H}^{1,-1/2-}(\R^3))$ we first control the weighted norm of $v$ and then the weighted norm of $w$. To do so, we apply the transferred local smoothing estimate~\eqref{eq-transferred-local-smoothing} with $V=0$.
\begin{equation}
\label{eq-v}
\norm{v}_{L_t^2(\R;\mathcal{H}^{1,-1/2-})}\leq\underset{2^\N}{\sup}\ \norm{\Pi_N  v}_{L_t^2(\R;\mathcal{H}^{1,-1/2-})}\lesssim \underset{2^\N}{\sup}\ N^{1/2}\norm{\Pi_N  v}_{U_\Delta^2}=\underset{2^\N}{\sup}\ N^{1/2}\norm{\Pi_N  f}_{DU_\Delta^2}\,.
\end{equation}
Afterwards, we use the fact that $w\in\Pc(H)$ and we apply the local smoothing estimate for $H$ : 
\begin{equation}
\label{eq-weight-w}
\norm{w}_{L_t^2(\R;\mathcal{H}^{1,-1/2-})}\lesssim\norm{u_0}_{\mathcal{H}^{1/2}}+\norm{g}_{L_t^2(\R;L_x^{2,1/2+})}+\norm{Vv}_{L_t^2(\R;L_x^{2,1/2+})}\,.
\end{equation}
To estimate the second term in the right-hand side, we use the fact that $V$ is localized.
\begin{multline}
\label{eq-Vv}
\norm{Vv}_{L_t^2(\R;L_x^{2,1/2+})}\lesssim \norm{\japbrak{x}^{1+}V}_{L_x^\infty(\R^d)}\norm{v}_{L_t^2(\R;L_x^{2,-1/2-})}\\
\leq C(V)\norm{v}_{L_t^2(\R;\mathcal{H}^{1,-1/2-})} \leq C(V)\underset{2^\N}{\sup}\  N^{1/2}\norm{\Pi_N  f}_{DU_\Delta^2}\,,
\end{multline}
where we used~\eqref{eq-v} to get the last inequality. By collecting estimates~\eqref{eq-weight-w} and~\eqref{eq-Vv} we obtain
\begin{equation}
\label{eq-weight-w2}
\norm{w}_{L_t^2(\R;\mathcal{H}^{1,-1/2-})}\lesssim_V\norm{u_0}_{L_x^2}+\underset{2^\N}{\sup}\  2^N\norm{\Pi_Nf}_{DU_\Delta^2}+\norm{g}_{L_t^2(\R;L_x^{2,1/2+})}\,.
\end{equation}
Finally, we deduce from~\eqref{eq-v} and~\eqref{eq-weight-w2} the estimate for $u$ in the weighted space.
\begin{equation}
\label{eq-u}
\norm{u}_{L_t^2(\R;\mathcal{H}^{1,-1/2-})}\lesssim\norm{u_0}_{\mathcal{H}^{1/2}}+\underset{2^\N}{\sup}\ 2^N\norm{\Pi_Nf}_{DU_\Delta^2}+ \norm{g}_{L_t^2(\R;L_x^{2,1/2+})}\,.
\end{equation}
It remains to control $Vu$ in $V_\Delta^2(\mathcal{H}^{1/2})$ to close the estimate~\eqref{eq-first}. This is again a consequence of the fact that $V$ is localized, and we proceed as in~\eqref{eq-Vv} to get
\begin{equation}
\label{eq-Vu}
\norm{Vu}_{L_t^2(\R;L_x^{2,1/2+})}\lesssim_V\norm{u}_{L_t^2(\R;\mathcal{H}^{1,-1/2-})}\,.
\end{equation}
We conclude by injecting~\eqref{eq-Vu} and~\eqref{eq-u} into~\eqref{eq-first} to control $u$ in $V_\Delta^2(\mathcal{H}^{1/2})$.
\end{proof}
\section{Probabilistic asymptotic stability for small ground states}
\label{section-soliton}
For now on, we fix the dimension $d=3$ and we make the extra assumption that $\sigma(H)=\set{e_0}\cup\sigma_c(H)$ where $e_0<0$ is a simple negative eigenvalue with positive and normalized eigenfunction $\phi_0$. We recall that $\sigma_c(H)=\intervalco{0}{+\infty}$, with no resonance nor eigenvalue at zero.

Since complex conjugate and numerical constants play no role in what follows, we might sometimes drop them from the notation for the sake of clarity. We will also drop the dependence on time in the notation and write $\norm{\Dz_z ^\alpha Q}\coloneqq \underset{t}{\sup}\norm{\Dz_z ^\alpha Q(z(t))}$, as well as $\bigO{z}\coloneqq \mathopen{}O\mathopen{}(\underset{t}{\sup}\abs{z(t)})$.
\subsection{Local existence}
First, let us briefly transpose the probabilistic Cauchy theory, at least locally in time, to our setting where the Schrödinger operator $-\Delta+V$ has some discrete spectrum. More precisely, we prove that under smallness assumptions on the $\mathcal{H}^{1/2}$-norm of the initial data $\psi_0$,~\eqref{eq-nls} still admits a unique local solution under the form 
\[
\psi(t) = \epsilon u^\omega(t) + v(t),\quad v\in \mathcal{H}^{1/2}\,.
\]
Indeed, the discrete part of the solution does not contribute in short time, and we shall be able to reproduce the scheme developed in section~\ref{section-continuous}, with the same gain of regularity for the nonlinear part of the solution. Note that by time reversibility, we only consider forward-in-time solutions in what follows.
\begin{proposition}[Local existence]
\label{prop-local-existence} 
There exist $\delta_0>0$ and $c>0$ such that for all $T\lesssim1$, all $\alpha_0\in\C$, $\nu_0\in \ran\Pc\cap \mathcal{H}^{1/2}$  with 
\[
\norm{v_0}_{\mathcal{H}^{1/2}}+\abs{\alpha_0}<\delta_0\,,
\]
and all $\omega\in\widetilde{\Omega}_{\epsilon,R}$  with $R=c\delta_0$, 
the Cauchy problem~\eqref{eq-nls} with data
\[
\psi(t_0) = \epsilon u_0^\omega + \nu_0 + \alpha_0\phi_0
\]
 admits a unique solution $\psi$ on $\intervalco{t_0}{t_0+T}$ under the form 
\[
\psi(t) = \epsilon u^\omega(t) + v(t),\quad v\in C(\intervalco{t_0}{t_0+T};\mathcal{H}^{1/2})\,.
\]
Uniqueness holds for $\Pc v\in X^{1/2}\intervaloo{0}{T}$ and $\prodscal{v}{\phi_0}\in L_t^\infty\intervaloo{0}{T}$.
\end{proposition}
We recall that the set of random initial data with improved Strichartz estimates defined in~\eqref{set-eps-R}. In light of the above Proposition, we fix $R=c\delta_0$ for now on, and take an initial value $u_0^\omega$ that corresponds to some $\omega$ in $\widetilde{\Omega}_{\epsilon,c\delta_0}$. We also recall that $\epsilon$ is much smaller than $R$.
\begin{proof}
For simplicity, we assume that $t_0=0$. At each time $t$, we decompose $v(t)$ into $v(t)=\nu(t)+\alpha(t)\phi_0$, with $\nu(t)=\Pc v(t)$, and $\alpha(t)=\prodscal{v(t)}{\phi_0}$. Then, $\psi$ is solution to~\eqref{eq-nls} if and only if $(\nu,\alpha)$ are solution to the coupled system 
\begin{equation*}
\begin{cases}
i\partial_t\nu + (\Delta-V)\nu = \Pc\mathcal{N}\parent{\epsilon u^\omega+\nu+\alpha\phi_0}\,,\\
\nu(0) = \nu_0\,,
\end{cases}
\ \text{and}\quad
\begin{cases}
i\dot{\alpha} + e_0\alpha = \prodscal{\mathcal{N}(\epsilon u^\omega+\nu+\alpha\phi_0)}{\phi_0}\,,\\
\alpha(0)=\alpha_0\,.
\end{cases}
\end{equation*}
It follows from the Duhamel's formulation that solving the above system reduces to find a fixed point for the map $\Gamma = (\Gamma_1,\Gamma_2)$, where 
\begin{equation*}
\begin{split}
\Gamma_1(\nu)(t) &= \e^{-itH}\nu_0 -i\int_0^t\e^{-i(t-\tau)H}\Pc\mathcal{N}\parent{\epsilon u^\omega(\tau)+\nu(\tau)+\alpha(\tau)\phi_0}d\tau\,,\\
\Gamma_2(\alpha)(t) &= \e^{ie_0t}\alpha_0 -i\int_0^t\e^{i(t-\tau)e_0}\prodscal{\mathcal{N}\parent{\epsilon u^\omega(\tau)+\nu(\tau)+\alpha(\tau)\phi_0}}{\phi_0}d\tau\,.
\end{split}
\end{equation*}
Given $T>0$ and $A$ to be chosen later on, we search a fixed point for $\Gamma$ in the Banach space 
\[
\mathcal{E}_{T,A} = \set{(\nu,\alpha)\in X^{1/2}\intervaloo{0}{T}\times L^\infty(\intervaloo{0}{T};\C)\mid\ \norm{\nu}_{X^{1/2}\intervaloo{0}{T}}+\underset{0\leq t\leq T}{\sup}\abs{\alpha(t)}\leq A}\,.
\]
In this setting, we have the following a priori estimates : given $(\nu,\alpha)\in \mathcal{E}_{T,A}$,
\begin{equation}
\label{eq-gamma-1}
\norm{\Gamma_1(\nu)}_{X^{1/2}\intervaloo{0}{T}}\leq C\norm{\nu_0}_{\mathcal{H}^{1/2}}+ C\parent{1+T^{1/5}}\parent{A^3+R^3}\,,
\end{equation}
and
\begin{equation}
\label{eq-gamma-2}
\underset{0\leq t\leq T}{\sup}\abs{\Gamma_2(\alpha)(t)}\leq \alpha_0 +
C(1+T)\parent{A^3+R^3}\,.
\end{equation}
Under our assumptions, we now fix $A=R=2C_0\delta_0\eqqcolon c\delta_0$ and the above estimates yield
\[
\norm{\Gamma_1(\nu)}_{X^{1/2}\intervaloo{0}{T}} + \underset{0\leq t\leq T}{\sup}\abs{\Gamma_2(\alpha)(t)} \leq C_0\parent{\delta_0+A^3+T^{1/2}}\leq 2C_0\delta_0
\]
provided that $\delta_0C_0<1/2$, and where $C_0=C_0(\norm{\phi_0}_{\mathcal{H}^1},\norm{\abs{x}^{-1/2-}\phi}_{L_x^\infty})$. This proves that $\Gamma$ preserves the space $\mathcal{E}_{T,A}$. Let us briefly explain how we obtained estimates~\eqref{eq-gamma-1} and~\eqref{eq-gamma-2}.  To prove~\eqref{eq-gamma-1}, we first note that the terms without the discrete part $\alpha\phi_0$ are handled by the analysis conducted in section~\ref{section-critical}, globally in time. Therefore, by dropping the complex conjugate from the notations, it remains to control some terms of the form 
\[ 
(\alpha\phi_0)^3,\quad (\alpha\phi_0)^2(\epsilon u^\omega + \nu),\quad \alpha\phi_0(\epsilon u^\omega + \nu)^2\,.
\]
To handle the terms without any power of the linear part $u^\omega$, we use the Leibniz rule and we obtain 
\begin{multline*}
\norm{(\alpha\phi_0)^3}_{DX_H^{1/2}\intervaloo{0}{T}}+\norm{(\alpha\phi_0)^2\nu}_{DX_H^{1/2}\intervaloo{0}{T}}+\norm{(\alpha\phi_0)\nu^2}_{DX_H^{1/2}\intervaloo{0}{T}}\\\lesssim T\parent{\underset{0\leq t\leq T}{\sup}\abs{\alpha(t)}^3\norm{\phi_0}_{\mathcal{H}^{1/2}}^3+\norm{\nu}_{X_H^{1/2}\intervaloo{0}{T}}}\lesssim TA^3\,.
\end{multline*}
As for the terms with some power of $u^\omega$, we can use the Leibniz rule and the local smoothing estimate~\eqref{eq-local-smoothing} as follows. We do a Littlewood-Paley decomposition of each term in the definition of Writing the definition of the $X^{1/2}$-norm : given a fixed $N\in2^\N$, we denote by $N_1, N_2, N_3$ the frequencies at which $\phi_0,u^\omega,u^\omega$ are localized.
\paragraph{\textbf{\textit{Case $\alpha\phi_0(u^\omega)^2$} :}}  We consider the worst case where $\phi_0$, which is the smoothest term, comes with the lowest frequency $N_1\leq N_2\leq N_3\sim N$. In the high-low-low case, when we have $N_2\ll N_3$, $\widetilde\Delta_{N_3}u^\omega$ cannot absorb all the derivatives, and we need to use the local smoothing effect. To that hand, we apply Hölder with $1=\frac{1}{2}+\frac{1}{5}+\frac{3}{10}$ and get
\begin{equation*}
\begin{split}
N^{1/2}\norm{P_N\parent{\alpha\phi_0(\epsilon u^\omega)^2}}_{U_H^2\intervaloo{0}{T}} &\leq 
\epsilon^2N_3^{1/2} 
\underset{\norm{w}_{V_H^2\intervaloo{0}{T}}\leq1}{\sup}\abs{\int_0^T\int_{\R^3}\alpha(t)\widetilde{\Delta}_{N_1}\phi_0\widetilde{\Delta}_{N_2} u^\omega\widetilde{\Delta}_{N_3}u^\omega wdxdt}\\
&\lesssim\underset{0\leq t\leq T}{\sup}\abs{\alpha(t)}N_3^{1/2}\norm{\widetilde{\Delta}_{N_1}\phi_0\widetilde{\Delta}_{N_3}\epsilon u^\omega}_{L_{t,x}^2\intervalco{0}{T}}\norm{\epsilon u^\omega}_{L_{t,x}^5\intervalco{0}{T}}\norm{w}_{V_H^2}\\
&\lesssim A\norm{\japbrak{x}^{1/2+}\phi_0}_{L_{x}^\infty}N_3^{1/2}\norm{\japbrak{x}^{-1/2-}\widetilde{\Delta}_{N_3}\epsilon u^\omega}_{L_{t,x}^2\intervaloo{0}{T}}\norm{\epsilon u^\omega}_{L_{t,x}^5\intervalco{0}{T}}\,.
\end{split}
\end{equation*}
Then, we use the probabilistic Strichartz estimate~\eqref{est-R} for the $L_{t,x}^5$ norm and the local smoothing effect~\eqref{eq-local-smoothing} to get 
\[
N^{1/2}\norm{P_N\parent{\alpha\phi_0(\epsilon u^\omega)^2}}_{U_H^2\intervaloo{0}{T}}\lesssim A\norm{ \widetilde\Delta_{N_3}\epsilon u_0^\omega}_{L_x^2}R\lesssim A \epsilon R\lesssim A^3+R^3\,,
\]
since $\epsilon$ has to be chosen much smaller than $R$ in~\eqref{set-eps-R} for the probability of the bad set of initial data to be small.~\footnote{\ In this case, note that the estimates do not depend on $T$.}
\paragraph{\textbf{\textit{Case $(\alpha\phi_0)^2u^\omega$} :}} Similarly, we assume that $N_1\leq N_2\leq N_3\sim N$ and do the same computations with $\alpha(t)\phi_0$ instead of $\epsilon u^\omega$. The only difference is that 
\[
\norm{\alpha\phi_0}_{L_{t,x}^5\intervalco{0}{T}}\leq AT^{1/5}\norm{\phi_0}_{L_x^5}\,.
\]
This yields
\[
N^{1/2}\norm{P_N(\alpha\phi_0)^2\epsilon u^\omega}_{U_H^2\intervaloo{0}{T}}\lesssim T^{1/5} A^2\epsilon\lesssim T^{1/5}(A^3+R^3)\,.
\]
Then, we obtain~\eqref{eq-gamma-1} by summing over $N_1,N_2,N_3$ and $N$ as in section~\ref{section-critical}. To get estimate~\eqref{eq-gamma-2}, we just use Hölder, the endpoint Strichartz estimate and the embedding $L_t^3(\intervaloo{0}{T};L_x^4)\hookrightarrow X_H^{0}\intervaloo{0}{T}$
\begin{equation*}
\begin{split}
\underset{0<t<T}{\sup}&\abs{\int_0^t\e^{i(t-\tau)e_0}\prodscal{\mathcal{N}\parent{\epsilon u^\omega(\tau)+\nu(\tau)+\alpha(\tau)\phi_0}}{\phi_0}d\tau}\\
&\lesssim \int_0^T\int_{\R^3}\parent{\abs{\nu(\tau,x)+\epsilon u^\omega(\tau,x)}^3+\abs{\alpha(\tau)\phi_0(x)}^3}\abs{\phi_0(x)}dxd\tau\\
&\lesssim \norm{\phi_0}_{L_x^2}\parent{\int_0^T\norm{\nu(\tau)+\epsilon u^\omega}_{L_x^6}^3d\tau}+T\underset{0<t<T}{\sup}\abs{\alpha(t)}^3\norm{\phi_0}_{L_x^4}^4\\
&\lesssim C(\phi_0)\parent{\norm{\japbrak{\nabla}^{1/4}\nu}_{L_t^3L_x^4}^3+\norm{\epsilon \japbrak{\nabla}^{1/4} u^\omega}_{L_t^3L_x^4}^3+TA^3}\,.
\end{split}
\end{equation*}
By the embedding $L_t^3(\intervaloo{0}{T};L_x^4)\hookrightarrow X_H^{0}\intervaloo{0}{T}$, the right-hand side is less than 
\begin{equation*}
\begin{split}
&\lesssim C(\phi_0)\parent{\norm{\nu}_{X_H^{1/4}\intervaloo{0}{T}}^3+\epsilon^3\norm{u_0}_{\mathcal{H}^{1/4}}^3+TA^3}\\
&\lesssim C(\phi_0)(1+T)\parent{A^3+R^3}\,.
\end{split}
\end{equation*}
This finishes the proof of estimates~\eqref{eq-gamma-1} and~\eqref{eq-gamma-2}. The proof of the Lipschitz estimates for $\Gamma$ follows similarly, and establishes the contraction mapping property for $\Gamma$.
\end{proof}
\subsection{Nonlinear ground states}
Before we dive into the proof of asymptotic stability we establish the main global estimates on the local flow constructed in the above paragraph, we collect from \cite{gustafson2004} some properties of the curve of \textit{nonlinear ground states} that bifurcates from the eigenspace spanned by $\phi_0$. In particular, we will be able to decompose the solution into a ground state plus a radiation term at each time, where the radiation term will be shown to scatter in the last paragraph. First, we recall how the ground states are constructed.
\begin{lemma}[Nonlinear ground states (Lemma 2.1 in~\cite{gustafson2004})]
\label{lemma-gs}
There exists $\delta>0$ small enough such that for all $z\in\C$ with $\abs{z}\leq\delta$, there exists a nonlinear ground state $Q$ and $E$ solution to 
\begin{equation}
\label{eq-gs}
\parent{\Delta-V+\abs{Q}^2}Q=EQ
\end{equation}
under the form
\begin{equation*}
Q(z)=z\phi_0+q(z),\quad E(z)=e_0+e(z).
\end{equation*}
We have uniqueness for $(q,e)$ in the class 
\begin{equation*}
\set{(q,e)\in (\mathcal{H}^2\cap\ran(\Pc))\times\R\mid \norm{q}_{\mathcal{H}^2}\leq\delta,\ \abs{e}\leq\delta}.
\end{equation*}
Thanks to the gauge invariance of the nonlinear part, 
\begin{equation*}
Q(z\e^{i\alpha})=Q(z)\e^{i\alpha},\quad E(z)=E(\abs{z}).
\end{equation*}
In addition, $q$ and its derivatives are small in $\mathcal{H}^2(\R^3)$ :
\begin{equation}
\label{-Dz Dz Q}
q=\bigO{z^3},\quad \Dz_z Q=\bigO{z^2},\quad \Dz_z ^2q = \bigO{z}.
\end{equation}
The first two derivatives of $e$ are also small and satisfy
\begin{equation}
\label{eq-De}
\abs{\Dz_z e}=\bigO{z},\quad \abs{\Dz_z ^2e}=\bigO{1}.
\end{equation}
\end{lemma}
The following elliptic estimates on some Sobolev weighted norms of the ground states and its derivatives with respect to the parameter $z$ are crucial, and our arguments heavily rely on the smallness of these quantities.
\begin{lemma}[Weighted elliptic estimates for the derivatives of ground states, see~\cite{gustafson2004}]
\label{lemma-weight}
There exists $\delta>0$ such that for all $\abs{z}<\delta$ and all $k\in\set{1,2}$ we have
\begin{equation}
\begin{split}
\label{eq-w}
&\norm{\japbrak{x}^kq}_{\mathcal{H}^2(\R^3)}=\bigO{z^3},\\
&\norm{\japbrak{x}^k\Dz_z Q}_{\mathcal{H}^2(\R^3)}=\bigO{z^2}\,.
\end{split}
\end{equation}
\end{lemma}
We recall that the generalized continuous spectral spaces $\mathcal{H}_c(z)$ parametrized by $z$ were introduced to encode orthogonality conditions~\eqref{orth} : 
\begin{equation*}
\mathcal{H}_c(z)\coloneqq\set{\eta\in L^2(\R^3)\mid \prodscalr{i\eta}{\partial_{z_1}Q}=\prodscalr{i\eta}{\partial_{z_2}Q}=0}.
\end{equation*}
In particular, note that $\ran(\Pc)$ corresponds to $\mathcal{H}_c(0)$. The following result (Lemma 2.2 in~\cite{gustafson2004}) yields a bijection from $\mathcal{H}_c(0)$ to $\mathcal{H}_c(z)$. This useful correspondence reduces the dynamic of $\eta$ in ansatz~\eqref{ansatz} to the dynamic of its continuous spectral part $\nu$. 
\begin{lemma}[Continuous spectral subspace comparison]
\label{lemma-proj}
There exists $\delta>0$ such that any function $\psi\in L^2(\R^3)$ can be uniquely decomposed into
\[
\psi = Q(z)+\eta,\quad \eta\in\mathcal{H}_c(z)\,.
\]
If $\norm{\psi}_{\mathcal{H}^s}\leq\delta$ for some $s\geq0$ then $\eta\in\mathcal{H}^s$ and $\norm{\eta}_{\mathcal{H}^s}+\abs{z}\lesssim\norm{\psi}_{\mathcal{H}^s}$. Moreover, for all $\abs{z}\leq\delta$ there exists a bijective operator 
\[
\function{\operatorname{R}(z)}{ \mathcal{H}_c(0)}{\mathcal{H}_c(z)}{u}{u+\alpha(z)(u)\phi_0\,.}
\]
Here, $\alpha(z) u$ is solution to
\begin{equation*}
\label{eq-alpha}
\Lambda(z)\alpha(z)u = \prodscalr{iu}{\Dz_z Q}\,,
\end{equation*}
with
\begin{equation*}
\Lambda(z)=\begin{pmatrix}
0 & - 1  \\
1 & 0  \\
\end{pmatrix} +\begin{pmatrix}
\prodscalr{\phi_0}{\Im(\partial_{z_2}q)}  &-\prodscalr{\phi_0}{\Re(\partial_{z_2}q)}  \\
\prodscalr{\phi_0}{\Im(\partial_{z_1}q)} & \prodscalr{\phi_0}{\Re(\partial{z_1}q)}  \\
\end{pmatrix}\eqqcolon\begin{pmatrix}
0 & - 1  \\
1 & 0  \\
\end{pmatrix}+\gamma(z)\,.
\end{equation*}
Finally, $\operatorname{R}(z)-I$ is compact and continuous in the operator norm on any space $Y$ that satisfies $\mathcal{H}^{-2}\subset Y\subset\mathcal{H}^2$.
\end{lemma}
\subsection{Global a priori estimates}
Now that we have a local in time solution $\psi$ for small initial data, we need to obtain some global estimates on $\phi$ to extend it as a global solution and understand to its asymptotic behavior. First, we recall that in a small neighborhood of 0 the solution can be uniquely decomposed into
\begin{equation*}
\psi = Q(z)+\eta\,,
\end{equation*} 
where $\eta$ satisfies the time-dependent orthogonality conditions $\eta(t)\in\mathcal{H}_c(z(t))$. In what follows, the continuous spectral part of the nonlinear part~\footnote{\ We emphasize that the so-called linear part of the solution $u^\omega(t)=\e^{-itH}u_0^\omega$ lies in $\mathcal{H}_c(0)$ at each time $t$, since $u_0$ is in $\mathcal{H}_c(0)$ and the randomization preserves this space, as well as the flow $\e^{-itH}$.} of the solution is denoted by $\nu$ :
\[
\nu = \Pc(\eta) - \e^{-itH}\epsilon u^\omega\in\mathcal{H}_c(0)\,.
\]
Then, we deduce from Lemma~\ref{lemma-proj} that $\eta = R(z)\parent{\Pc\eta}$ so that the solution decouples into
\[
\psi = Q(z) + \eta = Q(z) + \operatorname{R}(z)\parent{\epsilon u^\omega+\nu}\,.
\]
In particular, we note that
\begin{equation*}
\Pc(\eta)=\epsilon u^\omega+\nu,\quad \Pp(\eta)= (\operatorname{R}(z)-I)\Pc(\eta)=\alpha(z)\Pc(\eta)\phi_0=\alpha(z)\parent{\epsilon u^\omega+\nu}\phi_0\,.
\end{equation*}
In what follows, we fix $\delta>0$ as in Lemmas~\ref{lemma-weight},~\ref{lemma-proj} and $\delta_0>0$ as in the local existence result from Proposition~\ref{prop-local-existence}. Then, we take $\epsilon$ and $\delta'<\delta$ such that 
\[
\norm{\nu}_{L^2}+\abs{z}\leq\delta'\implies \norm{\psi}_{\mathcal{H}^s}\leq\delta\,.
\]
The goal is to obtain some global bounds on $\eta$ in a critical space contained in $L_t^\infty\mathcal{H}_x^{1/2}$ and on $z$ to prove that the solutions stays small. Then we bootstrap the local existence result to obtain global existence, and further exploit the global estimates to prove asymptotic stability. We denote by $m$ the following gauge transformation of the parameter $z$ 
\begin{equation}
\label{eq-gauge-transformation}
m(z)=z\exp\parent{i\int_0^tE(z)(\tau)d\tau}\,.
\end{equation}
It turns out that $m$ is the interesting evolution parameter. Indeed, if we assume orthogonality conditions~\eqref{orth} to hold then $m$ solves the ODE
\begin{equation}
\label{eq-modulation}
\begin{cases}
\frac{d}{dt} m(z) = \dot z+iE(z)z = - A(z,\eta)^{-1}\prodscalr{F}{\Dz_z Q}\,,\\
m(z)\lvert_{t=0}=z_0\,,
\end{cases}
\end{equation}
where for $j,k\in\set{1,2}$ we have
\begin{equation}
\label{eq-A}
A(z,\eta)_{j,k}= \prodscalr{i\operatorname{D}_jQ}{\operatorname{D}_kQ}+\prodscalr{i\eta}{\operatorname{D}_j\operatorname{D}_kQ}=j-k+\bigO{\delta^2}.
\end{equation}
The forcing term is
\begin{equation*}
F=F(z,\eta)=\Pc\parent{\mathcal{N}(Q+\eta)-\abs{Q}^2\eta-Q^2\overline{\eta}}\,.
\end{equation*} 
\begin{remark}
As a consequence of~\eqref{-Dz Dz Q}, the matrix $A$ is invertible provided that $\delta$ is small enough. 
\end{remark}
\begin{remark}
If we drop numerical constants and complex conjugate in the notations, we can write
\begin{equation*}
F=\Pc(\eta^3)+\Pc(Q\eta^2)\,.
\end{equation*}
In particular, we observe that orthogonality conditions~\eqref{orth} cancelled the terms which are linear in $\eta$. Since we do not control the $L^1$-in-time norm of the radiation term, we couldn't have handled these linear terms. That's why the time dependent orthogonality conditions are crucial. 
\end{remark}
\begin{remark}
\label{remark-m}
Since $E(z)=E(\abs{z})=E(\abs{m(z)})$ we have that  $\overline z=m\parent{\overline{m(z)}}$. Hence, $m$ is a bijective operator, so we can recover $z$ from $m(z)$. 
\end{remark}
To get the evolution equation for $\nu$, we inject the ansatz~\eqref{ansatz} into~\eqref{eq-nls} and we use the fact that $Q$ is solution to~\eqref{eq-gs}. This yields that the radiation term $\eta$ solves the equation
\begin{equation*}
i\partial_t\eta+(\Delta-V)\eta = \mathcal{N}(Q+\eta)-\mathcal{N}(Q)- i\Dz_z Q(\dot z +iEz)\,.
\end{equation*}
By projecting the above equation on the continuous spectral subspace, and noting that $\Pc(\Dz_z Q)=\Dz_z Q$, we get
\begin{equation}
\label{eq-nu}
\begin{cases}
i\partial_t\nu+(\Delta-V)\nu=\Pc\parent{\mathcal{N}(Q+\eta) - \mathcal{N}(Q)} - i\Dz_z q\, \dot m(z) \eqqcolon f+g\,,\\
\nu\lvert_{t=0}=\nu_0\,,
\end{cases}
\end{equation}
where we decomposed the forcing term into the sum of a nonlinear term $f$ and a localized term $g$. More precisely, we chose to collect in $f$ the higher order nonlinear terms $\mathcal{N}(\nu+\epsilon u^\omega)$ as well as the modulation term :
\begin{equation*}
f = \Pc\parent{\mathcal{N}(\nu+\epsilon u^\omega)}-i\Dz_z Q\dot m(z)\,.
\end{equation*}
On the other hand, $g$ contains the localized lower order terms which involve at least a power of $Q$ or of $\Pp(\eta)$ :
\begin{equation*}
g=\Pc\parent{Q^2\eta+Q\eta^2+\abs{\Pp(\eta)}^2\Pp(\eta)}\,.
\end{equation*}
To get some global in time bounds on $\nu$, we then follow the \textit{critical-weighted strategy} detailed in Proposition~\ref{proposition-trick}, and we are reduced to control $f$ in $DU_H^2(\mathcal{H}^{1/2})$ and $g$ in $\displaystyle L_t^2(\R;L_x^{2,1/2+})$. Hence, given an interval $I\subseteq\R$ where the solution is defined, we   will control $\nu$ in the critical-weighted space
\begin{equation}
\label{eq-X}
\mathcal X(I) = V_H^2(I;\mathcal{H}_x^{1/2}) \cap L_t^2(I ; \mathcal{H}_x^{1,-1/2-}(\R^d))\,,
\end{equation}
endowed with its natural norm, and $\dot{m}$ in $L_t^1(I)$. The aim is to get bounds that are independent of $I$.
\subsubsection{Global-in-time a priori estimates on $\nu$ and $\dot{m}$}
In the following Lemma, we prove some preliminary estimates on the radiation term $\eta$ that will be needed in the analysis. We recall that $\widetilde\Omega_{\epsilon,c\delta_0}$ is the set defined in~\eqref{set-eps-R} made of randomized initial data which display some improved Strichartz estimates, with $R=c\delta_0$ (see Proposition~\ref{prop-local-existence}).
\begin{lemma}[Preliminary estimates]
\label{lemma-R-I}
Take $\omega\in\widetilde\Omega_{\epsilon,R}$ and let $\eta, z, \nu$ be as in ansatz~\eqref{ansatz}. For any interval $I\subseteq\intervalco{0}{\infty}$, 
\begin{align}
\label{eq-R-I}
\norm{\eta}_{L^2(I ; L_x^{2,-1/2-})}&\lesssim\norm{\nu}_{\mathcal{X}(I)}+\epsilon\norm{u_0}_{L^2(\R^3)},\\
\label{eq-R-I-4}
\norm{\eta}_{L_t^4L_x^4(I\times\R^3)}&\lesssim\norm{\nu}_{\mathcal{X}(I)}+R\,.
\end{align}
\end{lemma}
\begin{proof}
We can estimate the discrete part of $\eta$ by its continuous part. Indeed, we recall that 
\begin{equation}
\label{est-R}
\abs{\Pp(\eta)}=\abs{(\operatorname{R}-\Id)(z)(\nu+u^\omega)}\lesssim \abs{\phi_0}\abs{\prodscal{\nu+u^\omega}{\Dz_z Q}}\,.
\end{equation}
Hence,
\begin{equation*}
\begin{split}
\norm{\Pp(\eta)}_{L_t^2(I ; L_x^{2,-1/2-})}&\leq \norm{\japbrak{x}^{-1/2-}\phi_0}_{L_x^2}\norm{\japbrak{x}^{1/2+}\Dz_z Q}_{L_x^\infty}\norm{\nu+\epsilon u^\omega}_{L_t^2(I ; L_x^{2,-1/2-})}\\
&\lesssim\norm{\japbrak{x}\Dz_z Q}_{L_t^\infty(I;\mathcal{H}^2)}\norm{\nu+\epsilon u^\omega}_{L_t^2(I ; L_x^{2,-1/2-})}\\
&\lesssim\norm{\nu}_\mathcal{X}+\epsilon\norm{u_0}_{L_x^2}\,,
\end{split}
\end{equation*}
where we used the local smoothing estimate~\eqref{eq-local-smoothing} to control the weighted norm of the perturbed linear propagation $u^\omega$ in the last inequality, as well as the fact that $\norm{u_0}_{L_x^2}\sim\norm{u_0^\omega}_{L_x^2}.$
To prove~\eqref{eq-R-I-4} we apply Hölder in~\eqref{est-R} and Sobolev embedding for $\norm{\phi_0}_{L^4(\R^3)}$. It comes
\begin{equation*}
\norm{(\operatorname{R}-\Id)(\nu+\epsilon u^\omega)}_{L_t^4L_x^4}\lesssim\norm{\phi_0}_{\mathcal{H}^{3/4}}\norm{\nu+\epsilon u^\omega}_{L_t^4L_x^4}\norm{\Dz_z Q}_{L_x^{4/3}}\,.
\end{equation*}
Applying Hölder once again, we get 
\begin{equation*}
\norm{\Dz_z Q}_{L_x^{4/3}}\leq\parent{\int_{\R^3}\japbrak{x}^{-3+}dx}^{1/3}\norm{\japbrak{x}^{3/4+}\Dz_z Q}_{L_x^2}\lesssim\norm{\japbrak{x}^2\Dz_z Q}_{L_x^2}\,.
\end{equation*}
Then, it follows from~\eqref{eq-w} that
\begin{equation*}
\norm{(\operatorname{R}-\Id)(\nu+\epsilon u^\omega)}_{L_t^4L_x^4}\lesssim\norm{\japbrak{x}^2\Dz_z Q}_{L_t^\infty L_x^2}\norm{\nu+\epsilon u^\omega}_{L_t^4L_x^4}\lesssim\norm{\nu}_{L_t^4L_x^4}+\epsilon\norm{u^\omega}_{L_t^4L_x^4}\,.
\end{equation*}
Note that we used the Sobolev weighted estimate~\eqref{eq-w} on $\Dz_z Q$. To control the deterministic term $\nu$ in $\displaystyle L_{t,x}^4$, we use Sobolev embedding, Strichartz estimate for the admissible pair $(4,3)$ and the transference principle in $V_H^2$. Since we chose $\omega\in\widetilde\Omega_{\epsilon,R}$ we can use the improved global-in-time Strichartz estimate~\eqref{eq-rand-str} to control the random term $u^\omega$ and obtain
\begin{equation*}
\norm{\nu+\epsilon u^\omega}_{L_t^4L_x^4}\lesssim\norm{\langle\sqrt H\rangle^{1/4}\nu}_{L_t^4L_x^3}+R\lesssim \norm{\langle\sqrt H\rangle^{1/4}\nu}_{V_H^2}+R
\lesssim\norm{\nu}_\mathcal{X}+R\,.
\end{equation*}
This gives both the desired estimate for the discrete part of $\eta$ and its continuous part $\nu+\epsilon u^\omega$, and finishes the proof of Lemma~\ref{lemma-R-I}.
\end{proof}
We can now state and prove the main global-in-time a priori estimate.
\begin{proposition}[Global a priori estimates]
\label{prop-a-priori-eta}
There exists $C=C(\phi_0)$ such that for all interval $I\subseteq\R$, for all $\omega\in\widetilde\Omega_{\epsilon,R}$ and $\psi_0\in\mathcal{H}^{1/2}$, if 
\[
\psi = Q(z) + \eta = Q(z) + \operatorname{R}(z)\parent{\epsilon u^\omega+\nu}\,,\quad
\psi_{\lvert{t=t_0}}= \epsilon u_0^\omega+\psi_0\,,
\]
is solution to~\eqref{eq-nls} on $I$ with 
$
\norm{\nu_0}_{\mathcal{H}^{1/2}}+\abs{z_0}\leq\delta_0$, $\norm{\nu}_{\mathcal{X}(I)}\leq\delta'$, $ \underset{t\in I}{\sup}\abs{z}\leq\delta'$, then~\footnote{\ We recall that $\delta'<\delta$ is chosen such that  $\norm{\psi}_{\mathcal{H}^s}\leq\delta$. This forces the solution to remain small and allows us to apply Lemma~\ref{lemma-weight} and Lemma~\ref{lemma-proj}.}
\begin{align}
\label{t-nu}
\norm{\nu}_{\mathcal{X}(I)}&\leq C\parent{ \norm{\nu}_{\mathcal{X}(I)}^3+R^3+\abs{z_0}^3+\norm{\dot m(z)}_{L^1(I)}^3+\norm{\nu_0}_{\mathcal{H}^{1/2}}}\,,\\
\label{t-x}
\norm{\dot m(z)}_{L^1(I)}&\leq C\parent{\norm{\nu}_{\mathcal{X}(I)}^3+R^3+\abs{z_0}^3+\norm{\dot m(z)}_{L^1(I)}^3}\,.
\end{align}
\end{proposition}
\begin{proof}[Proof of~\eqref{t-nu}] It follows from \textit{weighted-critical} estimate~\eqref{eq-trick-V} that
\begin{equation*}
\norm{w}_\mathcal{X} \lesssim \norm{\nu_0}_{\mathcal{H}^{1/2}}+\underset{2^\N}{\sup}\ N^{1/2}\norm{\widetilde{\Pi}_Nf}_{DU_H^2}+\norm{g}_{L_t^2(\R;L_x^{2,1/2+})}\,.
\end{equation*}
We recall that $f=\mathcal{N}(\nu+u^\omega)+i\Dz_z Q\dot m(z)$ and we observe that $\widetilde\Pi_N f\in L^1(I,\mathcal{H}^{1/2})$ as explained in~\eqref{N-L1}. Hence, we apply the probabilistic nonlinear estimate~\eqref{eq-a-priori} on the continuous spectral subspace stated in section~\ref{section-continuous} to get that for all $\omega\in\widetilde\Omega_{\epsilon,R}$
\begin{equation*}
\underset{2^\N}{\sup}\ N^{1/2}{\norm{\widetilde{\Pi}_N   \mathcal{N}(\nu+u^\omega)}_{DU_H^2}}\lesssim \norm{\nu}_\mathcal{X}^3+R^3\,.
\end{equation*}
To handle the modulation term, we observe that $\Pc(\Dz_z Q)=\Dz_z Q$, and we deduce from the duality argument between $U^2$ and $V^2$ detailed in remark~\ref{remark-duality} that
\begin{equation*}
\begin{split}
\norm{\Dz_z Q(z)\dot m(z)}_{DU_H^2}&\leq \underset{\norm{v}_{V_H^2}\leq 1}{\sup}\int_\R\prodscal{\Dz_z Q\dot m(z)}{v} dt \\
&\leq \underset{\norm{v}_{V_H^2}\leq 1}{\sup}\norm{v}_{L_t^\infty L_x^2}\norm{\Dz_z Q(z)}_{L_t^\infty L_x^2}\norm{\dot m(z)}_{L_t^1} \\
&=\norm{\Dz_z Q(z)}_{L_x^2}\norm{\dot m(z)}_{L_t^1}\,,
\end{split}
\end{equation*}
where we used the embedding $V^2\hookrightarrow \displaystyle L_t^\infty(\R; L_x^2)$. Consequently, 
\begin{equation*}\underset{2^\N}{\sup}\ N^{1/2}\norm{\widetilde{\Pi}_N   \Dz_z Q(z)\dot m(z)}_{DU_H^2}\lesssim\norm{\Dz_z Q}_{\mathcal{H}^{1/2}}\norm{\dot m(z)}_{L_t^1}=\bigO{z^2}\norm{\dot m(z)}_{L_t^1}\,.
\end{equation*}
Note that we used~\eqref{-Dz Dz Q} to control $\norm{\Dz_z Q}_{\mathcal{H}^{1/2}}$. Now, we need to estimate the localized lower order terms collected in $g$, defined in~\eqref{eq-nu}.
\begin{equation*}
\norm{g}_{L_t^2(\R;L_x^{2,1/2+})}\leq \norm{\Pc(Q^2\eta)}_{L_t^2(\R;L_x^{2,1/2+})}+\norm{\Pc(Q\eta^2)}_{L_t^2(\R;L_x^{2,1/2+})}+\norm{\abs{\Pp(\eta)}^2\Pp(\eta)}_{L_t^2(\R;L_x^{2,1/2+})}\,.
\end{equation*}
It follows from~\eqref{eq-w} and~\eqref{eq-R-I} that
\[
\norm{\Pc(Q^2\eta)}_{L_t^2(\R;L_x^{2,1/2+})}\leq \norm{\japbrak{x}^{1/2+}Q}_{L_x^\infty}^2\norm{\eta}_{L_t^2(\R;L_x^{2,-1/2-})}
\lesssim\norm{z}_{L_t^\infty}^2\parent{\norm{\nu}_\mathcal{X}+\epsilon\norm{u_0}_{L_x^2}}\,.
\]
Similarly,~\eqref{eq-w} and~\eqref{eq-R-I-4} yield
\[
\norm{\Pc(Q\eta^2)}_{L_t^2(\R;L_x^{2,1/2+})}\leq\norm{\japbrak{x}^{1/2+}Q}_{L_x^\infty}\norm{\eta}_{L_t^4L_x^4}^2\\
\lesssim\norm{z}_{L_t^\infty}\parent{\norm{\nu}_\mathcal{X}+R}^2\lesssim \norm{z}_{L_t^\infty}^3+\norm{\nu}_\mathcal{X}^3+R^3\,.
\]
To conclude we recall that
\begin{equation*}
\abs{\Pp(\eta)}\lesssim\abs{\phi_0}\abs{\prodscalr{\nu+u^\omega}{\Dz_z Q}}^3\,,
\end{equation*}
and therefore
\begin{equation*}
\begin{split}
\norm{\abs{\Pp(\eta)}^2\Pp(\eta)}_{L_t^2(\R;L_x^{2,1/2+})}&\lesssim\norm{\phi_0}_{L_x^{2,1/2+}}\parent{\int_\R\prodscalr{\nu+u^\omega}{\Dz_z Q}^6}^{1/2}\\
&\lesssim\parent{\int_\R\prodscalr{\nu+u^\omega}{\Dz_z Q}^2}^{1/2}\underset{\R}{\sup}\abs{\prodscalr{\nu+u^\omega}{\Dz_z Q}}^2 \\
&\lesssim\norm{\japbrak{x}\Dz_z Q}_{L_x^2}\norm{\nu+u^\omega}_{L_t^2(\R;L_x^{2,-1/2-})}\norm{\nu+u^\omega}_{L_t^\infty L_x^2}\norm{\Dz_z Q}_{L_x^2} \\
&\lesssim \parent{\norm{\nu}_\mathcal{X}+\epsilon\norm{u_0}_{L_x^2}}^3\lesssim \norm{\nu}_\mathcal{X}^3+\epsilon^3\norm{u_0}_{L_x^2}^3\,.
\qedhere
\end{split}
\end{equation*}

\end{proof}
\begin{proof}[Proof of~\eqref{t-x}]
We recall that 
\begin{equation*}\abs{\dot m(z)}\lesssim \abs{\prodscalr{\eta^3}{\Dz_z Q}}+\abs{\prodscalr{Q\eta^2}{\Dz_z Q}}\,.
\end{equation*}
We apply Cauchy-Schwarz inequality firstly in space and then in time to get that
\begin{equation*}
\begin{split}
\int_\R \abs{\prodscalr{\eta^3}{\Dz_z Q}} dt &\leq \norm{\japbrak{x}^{1/2+}\Dz_z Q}_{L_x^\infty}\int_\R\norm{\japbrak{x}^{-1/2-}\eta}_{L_x^2}\norm{\eta}_{L_{t,x}^4}^2dt \\
&\lesssim \norm{\japbrak{x}\Dz_z Q}_{\mathcal{H}^2}\norm{\eta}_{L_t^2(\R;L_x^{2,-1/2-})}\norm{\eta}_{L_{t,x}^4}^2\,.
\end{split}
\end{equation*}
By using estimate~\eqref{eq-w} on the weighted norm of $\Dz_z Q=J\phi_0+\Dz_z Q$ and estimates of Lemma~\ref{lemma-R-I} on $\eta$, we conclude that
\[
\int_\R \abs{\prodscalr{\eta^3}{\Dz_z Q}} dt \lesssim\parent{\norm{\nu}_\mathcal{X}+\epsilon\norm{u_0}_{L^2}}\parent{\norm{\nu}_\mathcal{X}+R}^2
\lesssim\norm{\nu}_\mathcal{X}^3+\epsilon^3\norm{u_0}_{L^2}^3+R^3\,.
\]
Similarly, we have
\begin{equation*}
\begin{split}
\int_\R \abs{\prodscalr{Q\eta^2}{\Dz_z Q}} dt &\leq \norm{\japbrak{x}^{1/2+}Q}_{L^\infty}\norm{\japbrak{x}^{1/2+}\Dz_z Q}_{L^\infty}\norm{\japbrak{x}^{-1}\eta}_{L_t^2L_x^2}^2 \\
&\lesssim \norm{\japbrak{x}Q}_{\mathcal{H}^2}\norm{\japbrak{x}\Dz_z Q}_{\mathcal{H}^2}\norm{\eta}_{L_t^2(\R;L_x^{2,-1/2-})}^2\\
&\lesssim\norm{z}_{L_t^\infty}\parent{\norm{\nu}_\mathcal{X}^2+\epsilon^2\norm{u_0}_{L^2}^2}\lesssim \norm{z}_{L_t^\infty}^3+\norm{\nu}_\mathcal{X}^3+\epsilon^3\norm{u_0}_{L^2}^3\,.
\end{split}
\end{equation*}
This concludes the proof of Proposition~\ref{prop-a-priori-eta}.
\end{proof}
\subsection{Proof of Theorem~\ref{theorem-soliton}} 
Now that we have a local solution with some global a priori estimates both on the radiation term $\eta$ and on the modulation parameter $z$, we shall be able to prove global well-posedness, and then deduce from these global estimates asymptotic stability. 
\paragraph{\textbf{Global existence and uniqueness}} 
It remains to finely tune the parameters that appear in Proposition~\ref{prop-a-priori}. First, we take $\delta_0$ such that the local existence result of Proposition~\ref{prop-local-existence} holds true, $R=c\delta_0$. Then, if the solution exists up to a certain time $T$ we prove that
\begin{equation}
\label{eq-a-priori-boot}
\norm{\nu}_{\mathcal{X}\intervalco{0}{T}}+\norm{\dot m(z)}_{L^1\intervalco{0}{T}}\leq\delta'\implies \norm{\nu}_{\mathcal{X}\intervalco{0}{T}}+\norm{\dot m(z)}_{L^1\intervalco{0}{T}}\leq\frac{\delta'}{2}\,.
\end{equation}
Indeed, we have from~\eqref{t-nu} and~\eqref{t-x} that 
\[
\norm{\nu}_{\mathcal{X}\intervalco{0}{T}}\leq C\parent{2(\delta')^3+(c\delta_0)^3+\delta_0^3+\delta_0}\,,\quad 
\norm{\dot m(z)}_{L^1\intervalco{0}{T}}\leq C\parent{2(\delta')^3+(c\delta_0)^3+\delta_0^3}\,.\]
The a priori estimate~\eqref{eq-a-priori-boot} follows if we chose $\delta_0$ even smaller, say such that $\delta_0\leq(4C)^{-1}\delta'$, and $\delta'\lesssim \parent{c^3C}^{-1/2}$. Then, we use a continuity argument (see Lemma~\ref{lemma:U2-continuity}) and prove that the solution is global, with the bounds 
\[
\norm{\nu}_{\mathcal{X}(\R)}+\norm{\dot m(z)}_{L^1(\R)}\leq\delta'\,.
\]
Since $\dot m(z)$ lies in $L^1\intervalco{0}{\infty}$ the convergence of the modulation parameter follows immediately. There exists $z_+\in\C$ such that
\begin{equation*}
\underset{t\to+\infty}{\lim}m(z)(t)=z_+\,.
\end{equation*}
Note that $\abs{z}$ converges, and hence $E(z)$ is also convergent.
\paragraph{\textbf{Scattering}}
We turn to the proof of the so-called completeness property of the flow, that is the fact that the radiation part $\eta$ scatters. To this end, we first prove scattering for $\nu = \Pc(\eta)-\epsilon u^\omega$ by following the proof of the scattering from Theorem~\ref{theorem-continuous} we made in section~\ref{section-continuous}. We recall that $\nu$ is solution to~\eqref{eq-nu}, and we prove that the Cauchy criterion is satisfied :
\begin{equation*}
\underset{t_1,t_2\to+\infty}{\lim}\norm{\e^{-it_1(\Delta-V)}\nu(t_1)-\e^{-it_2(\Delta-V)}\nu(t_2)}_{\mathcal{H}^{1/2}(\R^3)}  = 0\,.
\end{equation*}
Indeed, by the Duhamel integral formulation,
\begin{equation*}
\norm{\e^{-it_1(\Delta-V)}\nu(t_1)-\e^{-it_2(\Delta-V)}\nu(t_2)}_{\mathcal{H}^{1/2}(\R^3)}\leq\norm{f+g}_{DV^2\intervaloo{t_1}{t_2}}\,.
\end{equation*}
As observed in remark~\ref{remark-t}, it follows from the proof of Proposition~\ref{prop-a-priori-eta} that there exists a constant $C$ depending on some norms of the solution such that 
\begin{multline}
\label{eq-cauchy-pple}
\norm{f+g}_{DV_H^2\intervaloo{t_1}{t_2}}\leq \disp\norm{f+g}_{DU_H^2\intervaloo{t_1}{t_2}}\\
\leq C\parent{\norm{\nu}_{\mathcal{X}\intervaloo{t_1}{+\infty}},\norm{\dot m(z)}_{L_t^1\intervaloo{t_1}{+\infty}},\norm{u^\omega}_{L_t^q(\intervaloo{t_1}{+\infty} ; L_x^r)}}\,,
\end{multline}
and $C$ goes to zero when $t_1$ goes to $+\infty$. Hence, the Cauchy criterion is satisfied and there exists $\nu_+\in\mathcal{H}^{1/2}\cap\ran(\Pc)$ such that 
\begin{equation*}
\underset{t\to+\infty}{\lim}\norm{\nu(t)-\e^{-itH}\parent{\epsilon u_0^\omega+\nu_+}}_{\mathcal{H}^{1/2}(\R^3)}=0\,.
\end{equation*}
It remains to prove that the discrete part of $\eta$ goes to zero when $t$ goes to infinity. First, the weak convergence of
\[\e^{-itH}\parent{\epsilon u_0^\omega+\nu}\underset{t\to+\infty}{\disp \rightharpoonup}0\quad \text{in}\ \mathcal{H}^{1/2}(\R^3)\,\]
follows from the Riemann-Lebesgue theorem and the Plancherel formula for the distorted Fourier transform. Then, we use the compactness of the operator $\operatorname{R}-\Id$ from $\mathcal{H}^{1/2}(\R^3)$ to $\C$ stated in Lemma~\ref{lemma-proj} to get that
\begin{equation*}
\underset{t\to+\infty}\lim \Pp(\eta) = \underset{t\to+\infty}{\lim}(\operatorname{R}-\Id)\e^{-itH}\parent{\epsilon u_0^\omega+\nu}=0\quad \text{in}\ \mathcal{H}^{1/2}(\R^3)\,.
\end{equation*}
This concludes the proof of Theorem~\ref{theorem-soliton} with $\eta_+=\nu_+$.
\appendix
\section{Weighted elliptic estimates for the ground states}
\label{appendix-weight}
We show how to prove estimates of Lemma~\ref{lemma-weight} for the ground states $Q(z)$ and its derivative. For the convenience of the reader we recall the equations satisfied by $e,q,\Dz e,\Dz_z q$ (see A.4 in~\cite{gustafson2004}), and we omit numerical constants as well as complex conjugate from the notations. The nonlinear ground state $Q(z)=z\phi_0+q$ and $E(z)=e_0+e$ are solution to~\eqref{eq-GS}, and they satisfy
\begin{equation}
\label{eq-e}
\begin{split}
&(H-e_0)q=-\Pc\ \mathcal{N}(Q)+eq\,,\\
&ez=\prodscal{\phi_0}{\mathcal{N}(Q)}\,.
\end{split}
\end{equation}
We recall that $\mathcal{N}(u)=\abs{u}^2u$, $\Dz_z=(\partial_{z_1},\partial_{z_2})$ and we use the notation $\operatorname{J}=D_z(z)=(1,i)$. Then, if we differentiate the above expressions with respect to $z$ we obtain
\begin{equation}
\label{eq-Q-e}
\begin{split}
&(H-e_0)\Dz_z Q=-\Pc\ \Dz_z \mathcal{N}(Q)+q\Dz_z e+e\Dz_z Q\,,\\
&z\Dz_z e+\operatorname{J}e=\prodscal{\phi_0}{\Dz_z \mathcal{N}(Q)}\,,
\end{split}
\end{equation}
with
\[
\Dz_z \mathcal{N}(Q)=Q^2(\operatorname{J}\phi_0+\Dz_z Q)=\bigO{z^2}\ \text{in}\ \mathcal{H}^2\,.
\]
\begin{proof}[Proof of Lemma~\ref{lemma-weight}]
We first prove the estimate for $q$ in $\mathcal{H}^1$ with weight $\japbrak{x}$. Given a Schwartz function $\varphi$, we have
\begin{equation}
\label{eq-a1}
(H-e_0)\japbrak{x}\varphi = [-\Delta,\japbrak{x}]\varphi+\japbrak{x}(H-e_0)\varphi
=-\frac{2x\cdot\nabla\varphi}{\japbrak{x}}-\frac{3+2\abs{x}^2}{\japbrak{x}^3}\varphi+\japbrak{x}(H-e_0)\varphi\,.
\end{equation}
Then, we multiply~\eqref{eq-a1} by $\japbrak{x}\overline\varphi$ and we integrate over $\R^3$ to obtain
\begin{equation}
\label{eq-A1}
\norm{\nabla(\japbrak{x}\varphi)}_{L^2}^2-e_0\norm{\japbrak{x}\varphi}_{L^2}^2=\int_{\R^3}\parent{\japbrak{x}^2V(x)+3+\frac{3+2\abs{x}^2}{\japbrak{x}^2}}\abs{\varphi(x)}^2dx
+\prodscal{\japbrak{x}^2(H-e_0)\varphi}{\varphi}\,.
\end{equation}
Since the first term on the right-hand side is equivalent to $\norm{\varphi}_{L^2}^2$, it suffices to handle the second term. Formally taking $\varphi=q$ and using~\eqref{eq-e}, we have 
\begin{equation*}
\begin{split}
\prodscal{\japbrak{x}(H-e_0)q}{\japbrak{x}q} &= \prodscal{-\Pc\mathcal{N}(Q)+eq}{\japbrak{x}^2q}\\
&=-\prodscal{\mathcal{N}(Q)}{\japbrak{x}^2q}+\prodscal{\phi_0}{\mathcal{N}(Q)}\prodscal{\phi_0}{\japbrak{x}^2q}+e\norm{\japbrak{x}q}_{L^2}\,.
\end{split}
\end{equation*}
Since $q=\bigO{z^3}$ in $\mathcal{H}^2$ (see~\eqref{-Dz Dz Q}), we have 
\[
\abs{\prodscal{\mathcal{N}(Q)}{\japbrak{x}^2q}}+\abs{\prodscal{\phi_0}{\mathcal{N}(Q)}\prodscal{\phi_0}{\japbrak{x}^2q}}\lesssim\abs{z}^6\parent{ \norm{\japbrak{x}q}_{L^2}^2+C(\norm{\japbrak{x}^2\phi_0}_{L^\infty})}\,.
\]
Therefore, we deduce from~\eqref{eq-A1} that
\[
\norm{\nabla(\japbrak{x}q)}_{L^2}^2-(e_0+e+C\abs{z}^6)\norm{\japbrak{x}q}_{L^2}^2\lesssim\norm{q}_{L^2}^2+\abs{z}^6C(\norm{\japbrak{x}^2\phi_0}_{L^\infty})\,.
\]
To conclude we observe that there exists $\gamma>0$ and $c_\gamma>0$ such that $E(z)=e_0+e(z)<-c_\delta$ for all $\abs{z}\leq\delta$.
Next, we just mention that the estimate for the weight $\japbrak{x}^2$ follows similarly, using that
\[[-\Delta,\japbrak{x}^2]\varphi = 4x\cdot\nabla\varphi+6\varphi\,.\]
Let us now derive the estimate with the weight $\japbrak{x}$ in $\mathcal{H}^2$ for $\Dz_zq$. To this end, we first differentiate ~\eqref{eq-a1} with respect to $x_k$ and we get
\begin{equation*}
(H-e_0)\partial_k(\japbrak{x}\varphi)+(\partial_kV)\japbrak{x}\varphi=\partial_k\parent{[-\Delta,\japbrak{x}]}\varphi+\partial_k(\japbrak{x}(H-e_0)\varphi)\,.
\end{equation*}
In addition, we multiply by $\partial_k(\japbrak{x}\overline\varphi)$ on both sides, and we integrate over $\R^3$ to see that
\begin{multline}
\label{eq-D2B}
\norm{\nabla(\partial_k\japbrak{x}\varphi)}_{L^2}^2-e_0\norm{\partial_k\japbrak{x}\varphi}_{L^2}^2=-\int_{\R^3}\parent{V\partial_k(\japbrak{x}\varphi)+\partial_k\parent{[-\Delta,\japbrak{x}]}\varphi}\partial_k(\japbrak{x}\overline\varphi)dx\\
+\prodscal{\partial_k(\japbrak{x}(H-e_0))\varphi}{\partial_k(\japbrak{x}\varphi)}\,.
\end{multline}
Since
\begin{equation*}
\partial_k\parent{[-\Delta,\japbrak{x}]}\varphi = -2\parent{\frac{\partial_k\varphi}{\japbrak{x}}+\frac{x\cdot\partial_k\nabla\varphi}{\japbrak{x}}-x_k\frac{x\cdot\nabla\varphi}{\japbrak{x}^3}}+(3+2\abs{x}^2)\frac{3x_k\varphi}{\japbrak{x}^5}-\frac{4x_k\varphi}{\japbrak{x}^3}-(3+2\abs{x}^2)\frac{\partial_k\varphi}{\japbrak{x}^3}\,,
\end{equation*}
we have that
\begin{equation*}
\label{eq-D2B2}
\abs{\int_{\R^3}\parent{V\partial_k(\japbrak{x}\varphi)+\partial_k\parent{[-\Delta,\japbrak{x}]}\varphi}\partial_k(\japbrak{x}\overline\varphi)dx}\lesssim\norm{\japbrak{x}\varphi}_{\mathcal{H}^1}^2+\norm{\varphi}_{\mathcal{H}^2}^2\,.
\end{equation*}
Moreover, by using~\eqref{-Dz Dz Q},~\eqref{eq-De},~\eqref{eq-Q-e} and the estimates for $q$ and $\Dz_z q$ with weight $\japbrak{x}$ in $\mathcal{H}^1$ already proved in the first part of the proof, we show that 
\begin{equation*}
\begin{split}
\abs{\prodscal{\partial_k(\japbrak{x}(H-e_0))\varphi}{\partial_k(\japbrak{x}\varphi)}}\leq C(\phi_0)\abs{z}^4\,.
\end{split}
\end{equation*}
Hence, 
\begin{equation*}
\norm{\nabla(\partial_k\japbrak{x}\varphi)}_{L^2}^2-(e_0+e(z))\norm{\partial_k\japbrak{x}\varphi}_{L^2}^2\lesssim C(\phi_0)\abs{z}^4\,.
\end{equation*}
This concludes the proof since $e_0+e(z)\leq -c_\gamma<0$ for $\gamma$ small enough and $\abs{z}\leq\gamma$.
\end{proof}
\begin{center}
\textbf{Acknowledgements}
\end{center}
The author would like to thank his thesis supervisors Nicolas Burq and Frederic Rousset for suggesting him this problem and for giving him many helpful comments related to this work.
\printbibliography[
title={References}
]
\end{document}